\documentclass[12pt,twoside,english]{article}
\usepackage[T1]{fontenc}
\usepackage[latin9]{inputenc}
\usepackage{geometry}
\usepackage{babel}
\usepackage{mathrsfs}
\usepackage{amsthm}
\usepackage{amsmath}
\usepackage{amssymb}
\usepackage{esint}
\usepackage[unicode=true,pdfusetitle,
 bookmarks=true,bookmarksnumbered=false,bookmarksopen=false,
 breaklinks=false,pdfborder={0 0 1},backref=false,colorlinks=false]
 {hyperref}

\makeatletter
\numberwithin{equation}{section}
\numberwithin{figure}{section}

\usepackage{stmaryrd}\usepackage{mathtools}

\theoremstyle{plain}
\newtheorem{thm}{Theorem}[section]
  \theoremstyle{definition}
  \newtheorem{defn}[thm]{Definition}
  \theoremstyle{remark}
  \newtheorem{rem}[thm]{Remark}
  \theoremstyle{plain}
  \newtheorem{lem}[thm]{Lemma}
  \theoremstyle{plain}
  \newtheorem{prop}[thm]{Proposition}
  \theoremstyle{definition}
  \newtheorem{example}[thm]{Example}
  \theoremstyle{plain}
  \newtheorem{cor}[thm]{Corollary}

\numberwithin{equation}{section}
\numberwithin{figure}{section}
\numberwithin{table}{section}
\numberwithin{thm}{section}
 \let\footnote=\endnote

\DeclareMathOperator*{\esssup}{ess\,sup}
\DeclareMathOperator*{\essinf}{ess\,inf}
\DeclareMathOperator*{\argmax}{ess\,max}

\makeatother

\begin{document}

\title{Path-Dependent Optimal Stochastic Control and Viscosity Solution
of Associated Bellman Equations%
\thanks{supported in part by the Natural Science Foundation of China (Grants
\#10325101 and \#11171076), the Science Foundation for Ministry of
Education of China (No.20090071110001), and the Chang Jiang Scholars
Programme.%
}}

\author{Shanjian Tang and Fu Zhang%
\thanks{Institute of Mathematical Finance and Department of Finance and Control
Sciences, School of Mathematical Sciences, Fudan University, Shanghai
200433, China. Email: sjtang@fudan.edu.cn (Shanjian Tang), 09110180028@fudan.edu.cn
(Fu Zhang)%
} }
\maketitle
\begin{abstract}
In this paper we study the optimal stochastic control problem for
a path-dependent stochastic system under a recursive path-dependent
cost functional, whose associated Bellman equation from dynamic programming
principle is a path-dependent fully nonlinear partial differential
equation of second order. A novel notion of viscosity solutions is
introduced by restricting the semi-jets on an $\alpha$-H\"older space $\mathbf{C}^\alpha$ for $\alpha\in (0, {1\over2})$. Using Dupire's functional It\^o calculus, we characterize
the value functional of the optimal stochastic control problem as
the unique viscosity solution to the associated path-dependent Bellman
equation.
\end{abstract}
\textbf{Keyword:} path-dependent optimal stochastic control, path-dependent
Bellman equation, viscosity solution, backward SDE, dynamic programming.

\section{Introduction}

Since the initial investigation of Krylov \cite{krylov1972control},
it has been the subject of many studies to verify that the value function
of an optimal stochastic control problem should be the unique solution
of the associated Bellman equation from the dynamic programming principle
for the optimal stochastic control problem. Nowadays, the notion of
viscosity solution invented in 1983 by Crandall and Lions \cite{crandall1983viscosity}
has become a universal tool to study such a broad fundamental subject.
For detailed exposition of such a tool and the related general dynamic
programming theory on optimal stochastic control, see among others
the survey paper of Crandall, Ishii \& Lions \cite{crandall1992user}
and the monographs of Fleming \& Soner \cite{fleming2006controlled}
and Yong \& Zhou \cite{yong1999stochastic}.

Such a theme has also been developing in terms of backward stochastic
differential equations (BSDEs). For instance, a BSDE depending on
a Markovian diffusion in a Markovian way via the generator and the
terminal condition, is associated to a second-order partial differential
equation (PDE), and a fully coupled forward and backward stochastic
differential equation (FBSDE) is associated to a quasi-linear PDE.
For relevant details, see Pardoux \& Peng \cite{pardoux1992backward},
Ma, Proter \& Yong \cite{ma1994solving}, and Pardoux \& Tang \cite{pardoux1999forward}.
These developments appear quite natural in view of the close relation
between BSDEs and minimax problems as exposed by Tang \cite{tang2006dual}.
Furthermore, the second-order BSDE (2BSDE) is associated to a fully
nonlinear PDE, and for such a relation see Cheritdito, Soner, Touzi
\& Victoir \cite{cheridito2007second} and Soner, Touzi \& Zhang \cite{soner2011wellposedness}.
More generally, Peng \cite{peng1992stochastic} shows that an optimal
stochastic control problem---where the coefficients of both the system
and the cost functional depend on the history path of the underlying
Brownian motion---should be associated to a fully nonlinear backward
stochastic PDE as the underlying Bellman equation.

Recently Dupire \cite{dupire2009functional} introduced horizontal
and vertical derivatives in the path space of a path functional, non-trivially
generalized the classical It\^o formula to a functional It\^o formula
(see Cont \& Fournie \cite{cont2010change,cont2010functional} for
a more general and systematic research), and provided a functional
extension of the classical Feynman-Kac formula. His insightful work
is becoming a foundation to stochastic analysis of path functionals,
and has stimulated extensions of some above-mentioned developments
to the functional case. In fact, he has shown that a path functional
in $\mathscr{C}^{1,2}(\Lambda)$ solves a linear path-dependent PDE
(PPDE) if its composition with a Brownian motion generates a martingale.
In the plenary lecture at the International Congress of Mathematicians
of the year 2010, Peng \cite{peng2010backward} pointed out that a
non-Markovian BSDE is a PPDE. In view of Dupire's functional It\^o
formula, it is very natural to associate a BSDE to a ``semi-linear''
PPDE, and a stochastic optimal control problem of BSDE to a path-dependent
Bellman equation. Peng \& Wang \cite{peng2011bsde} studied the former
relation and give some sufficient conditions for a semi-linear PPDE
to admit a classical solution. However, a PPDE, even for the simplest
heat equation, rarely has a classical solution, and a path-dependent
Bellman equation, even in the simpler state-dependent case, also appeals
to a generalized solution in many occasions. Therefore, generalized
solution of general PPDEs are demanding, and it has to be developed.

This paper incorporates Dupire's functional It\^o calculus to discuss
the optimal stochastic control problem for a path-dependent stochastic
system under a recursive path-dependent cost functional. The associated
Bellman equation from the dynamic programming principle in this general
setting is a path-dependent fully nonlinear partial differential equation
of second order, and we are concerned with its viscosity solution.
The familiar optimal control of SDEs with delay can be addressed within
our current framework.

In the classical theory of viscosity solution to PPDEs (see Crandall, Ishii
\& Lions \cite{crandall1992user} and Fleming \& Soner \cite{fleming2006controlled}),
local compactness of the state space and smoothness of the norm play a crucial role. The main
difficulty for our path-dependent case lies in both facts that the path
space $\Lambda$ is an infinite dimensional Banach space and lacks of a local
compactness, and that the maximal norm $\|\cdot\|_0$ is not smooth. The arguments for the case of Hilbert space introduced
by Lions (see \cite{lions1988viscosity,lions1989viscosity}) contain
a limiting procedure based on the structure of Hilbert space, and
are difficult to be adapted to our path-dependent case where we have to work in a subspace $\mathbf{C}^\alpha$ (which has a local compactness, but whose $\alpha$-H\"older norm is not smooth). 

In the generalization of Kim's $i$-smooth theory \cite{kim1999functional},
Luyakonov \cite{lukoyanov2007viscosity} developed a theory of viscosity
solution to fully non-linear path-dependent (also called functional
in literature) Hamilton-Jacobi (HJ) equations of first order---which
include conventionally called Bellman and Isaacs equations arising
from deterministic optimal control problems and differential games
for time-delayed ordinary differential equations. He used the so-called
\textit{co-invariant derivatives} (it is \textit{Clio}-derivative
in Kim's terminology), which coincide with the restriction of Dupire's
derivatives on continuous paths though their definitions appear different.
A deterministic functional differential system, subject to some proper
conditions, starting at a uniformly Lipschitz continuous path, actually
evolves forever in the locally compact space $\mathbf{C}^{0,1}$ of
all uniformly Lipschitz continuous paths. This property of deterministic
dynamical systems allows Luyakonov \cite{lukoyanov2007viscosity}
to define the jet functionals and to localize the application of dynamical
programming all in $\mathbf{C}^{0,1}$ (in fact in a sequence of expanding
compact subsets) of paths. Due to these conveniences, his proof of
existence and uniqueness of viscosity solutions appears fairly straightforward.
In our stochastic case, the situation changes in a dramatic manner.
Even starting at a uniformly Lipschitz continuous path, due to the
essential diffusion nature, our dynamic system could not live in $\mathbf{C}^{0,1}$
anymore, and it is impossible in general to enclose under proper conditions
our dynamical system within any given compact space. We have to choose
an $\alpha$-H\"older continuous paths space $\mathbf{C}^{\alpha}$
($\forall\alpha\in(0,\frac{1}{2}$)) to substitute $\mathbf{C}^{0,1}$,
and also define our jet functional on a family of expanding compact
subsets $\{\mathbf{C}_{\mu}^{\alpha}:\mu\text{ is sufficiently large}\}$
in $\mathbf{C}^{\alpha}$. We give an example to illustrate the following
phenomenon in our stochastic dynamic system (see Remark \ref{rem:exist-th}
for details): starting at the boundary of $\mathbf{C}_{\mu}^{\alpha}$,
our stochastic system might leave away from the set $\mathbf{C}_{\mu}^{\alpha}$
with probability one within an arbitrary small time. This essential
nature prevents us from starting the dynamic programming
at the boundary of $\mathbf{C}_{\mu}^{\alpha}$ to show the Bellman
equation holds there for the value functional. More precisely, to show that the value
functional is a viscosity solution, we could not follow the conventional
way to start the dynamic programming directly at the minimum/maximum path
of the difference between the value functional and a jet functional
since the extremum path might be at the boundary. To get around the difficulty, we construct a specific
perturbation $\gamma_{t}^{\varepsilon}\in \mathbf{C}_{\mu}^{\alpha}$ (around the extremum path $\gamma_t$) where we can start the dynamic
programming and use the exit time $\hat\tau^\varepsilon$  from  $\mathbf{C}_{\mu}^{\alpha}$ (see the definition of $\hat\tau^\varepsilon$ in the existence proof of Section 5; the
probability for our dynamical system starting from $\gamma_{t}^{\varepsilon}$ at time $t$
to stay within $\mathbf{C}_{\mu}^{\alpha}$ up to time $t+\delta$ is shown to converge
to one, uniformly with respect to sufficiently large $\mu$, as  $\delta\to 0+$) to localize our dynamic programming within the compact subset
$\mathbf{C}_{\mu}^{\alpha}$. It seems to be the first here for us to define  $\hat\tau^\varepsilon$ in association to the $\alpha$-H\"older modulus of the system's history path. We finally prove that the value functional
is a viscosity solution in Section 5. In the proof of the uniqueness,
we adapt the smoothing (and viscosity vanishing for the degenerate
case) methodology of Lions \cite{lions1983optimal} to our path-dependent
case, and also use the natural approximating arguments of parameterized
state-dependent PDEs. In the passage to the limit, our a priori maximal
and H\"older estimates on the second-order derivatives of the solutions
to the approximate PPDEs play a crucial role. Our methodology is expected
to be used to study the path-dependent Isaacs equation arising from
stochastic differential games (see \cite{TangZhang2012game}) and
other related fully nonlinear PPDEs. We note that using Gateaux or
Frechet derivatives, Goldys \& Gozzi \cite{goldys2006second} and
Fuhrman, Masiero \& Tessitore \cite{fuhrman2010stochastic} considered
time-delay stochastic optimal control problems with the diffusion
being independent of the control variable, and studied the mild solutions
to the associated semi-linear functional Bellman equations on Hilbert
space and Banach space, respectively.

Defining the ``semi-jets'' on non compact subsets in terms of a
nonlinear expectation, Ekren et al. \cite{ekren2011viscosity} studied
in a quite different way viscosity solution to second order path-dependent
PDE. They prove the existence and uniqueness of viscosity solutions
(in their sense) for a semi-linear PPDE by Peron's approach. In the
subsequent works, Ekren, Touzi, and Zhang (see for details \cite{TouziZhang2012Optimal,TouziZhang2012viscosityI,TouziZhang2012viscosityII})
use their previous notion of viscosity solution to study the fully
nonlinear PPDE, Pham and Zhang (see \cite{pham2012two}) discuss path-dependent
Bellman-Isaacs equations. However, their relevant results on the path-dependent
Bellman equation require stronger conditions: Their Assumption 2.8
requires, as they have noted, the diffusion coefficient $\sigma$
to be path-invariant (see \cite[page 8]{TouziZhang2012viscosityII}
for details), and, in the degenerate case, their Assumption 7.1 (i)
requires further approximating structures (see \cite[page 29]{TouziZhang2012viscosityII}
for details). Our Theorems \ref{thm:(repres-theorem-nondegen} (on
the non-degenerate case) and \ref{thm:Repr-Th-degen} (on the degenerate
case) give more general uniqueness results on path-dependent Bellman
equation. Furthermore, Ekren, Touzi, and Zhang \cite{TouziZhang2012Optimal,TouziZhang2012viscosityI,TouziZhang2012viscosityII}
directly work with an abstract fully nonlinear PPDE, and use a more
complicated definition of super- and sub-jets in their notion of viscosity
solution, in particular their definitions involve the unnatural and
advanced notion of nonlinear expectation.



Backward stochastic PDE is another tool to study non-Markovian optimal
control problems and FBSDEs. Peng \cite{peng1992stochastic,peng1997bsde}
established the non-Markovian stochastic dynamic programming principle
where he derived the backward stochastic Bellman equations in a heuristic
way. Ma \& Yong \cite{ma1997adapted} gave the relationship between
FBSDEs and a class of semi-linear BSPDEs, and further developed the
stochastic Feynman-Kac formula. For Sobolev and classical solution
of BSPDE, we refer to Zhou \cite{zhou1992duality}, Tang \cite{tang2005semi},
Du \& Meng \cite{du2010revisit}, Du \& Tang \cite{du2011strong},
Du, Tang \& Zhang \cite{du2011Linear-Degenerate} and Qiu \& Tang
\cite{qiu2012maximum}. For viscosity solution of BSPDE or SPDE, we
refer to Lions \cite{lions1998fully,lions1998fully2}, Buckdahn \&
Ma \cite{buckdahn2001stochastic1,buckdahn2001stochastic2,buckdahn2002pathwise-Taylor,buckdahn2007pathwise}
and Boufoussi et al. \cite{boufoussi2007generalized}. In Example
\ref{relation-PHJB-SHJB}, the relationship is exposed between path-dependent
Bellman equations and backward stochastic Bellman equations.

The rest of the paper is organized as follows. In Section \ref{sec:Preliminaries.},
we introduce the calculus for path functionals of \cite{cont2010change,cont2010functional}
and \cite{dupire2009functional}, and preliminary results on BSDEs.
In Section \ref{sec:Stochastic-optimal-problems}, we formulate the
path-dependent stochastic optimal control problem and discuss the
dynamic programming principle, which is crucial in the proof of the
existence of a viscosity solution. In Section \ref{sec:FBSDE-and-viscosity},
we define classical and viscosity solutions to our path-dependent
Bellman equation, state our main results, and prove a verification
theorem in the context of classical solutions. In Section 5, we prove
that the value functional is a viscosity solution, which implies our
existence of a viscosity solution to the path-dependent Bellman equations.
Finally in Section \ref{sec:Uniqueness-of-viscosity} we prove the
uniqueness of viscosity solutions for the path-dependent Bellman equations.

\section{\label{sec:Preliminaries.}Preliminaries}

\subsection{Calculus of path functionals}

\subsubsection{Space of cadlag paths}

Let $n$ be a positive integer and $T$ be a fixed positive number.
For each $t\in[0,T]$, define $\hat{\Lambda}_{t}(\mathbb{R}^{n}):=\mathrm{D}([0,t],\mathbb{R}^{n})$
as the set of all cadlag (right continuous with left limit) $\mathrm{\mathbb{R}}^{n}$-valued
functions on $[0,t]$. For $\gamma\in\hat{\Lambda}_{T}(\mathbb{R}^{n})$,
$\gamma(s)$ is the value of $\gamma$ at time $s\in[0,T]$, and for
some $t\in[0,T]$ we denote the part of $\gamma$ up to time $t$ by $\gamma_{t}:=\{\gamma(s),s\in[0,t]\}$$\in\hat{\Lambda}_{t}(\mathbb{R}^{n})$.
Define $\hat{\Lambda}(\mathbb{R}^{n}):=\bigcup_{t\in[0,T]}\hat{\Lambda}_{t}(\mathbb{R}^{n})$.
Write $\hat{\Lambda}$ for $\hat{\Lambda}(\mathbb{R}^{n})$ if there
is no confusion.

For convenience, define for $0\leq t<\bar{t}\leq T$, and $\bar{\gamma}_{\bar{t}},\gamma_{t}\in\hat{\Lambda}$,
\begin{eqnarray*}
\gamma_{t}^{x}(s) & := & \gamma_{t}(s)\chi_{[0,t)}(s)+(\gamma_{t}(t)+x)\chi_{\{t\}}(s),\quad s\in[0,t];\\
\gamma_{t,\bar{t}}^{x}(s) & := & \gamma_{t}(s)\chi_{[0,t)}(s)+(\gamma_{t}(t)+x)\chi_{[t,\bar{t}]}(s),\quad s\in[0,\bar{t}];\\
\gamma_{t,\bar{t}} & := & \gamma_{t,\bar{t}}^{\mathbf{0}};\\
\bar{\gamma}_{\bar{t}}^{\gamma_{t}}(s) & := & \gamma_{t}(s)\chi_{[0,t)}(s)+(\bar{\gamma}_{\bar{t}}(s)-\bar{\gamma}_{\bar{t}}(t)+\gamma_{t}(t))\chi_{[t,\bar{t}]}(s),\quad s\in[0,\bar{t}].
\end{eqnarray*}

We define the quasi-norm and metric in $\hat{\Lambda}$ as follows:
for each $0\leq t\leq\bar{t}\leq T$ and $\gamma_{t},\bar{\gamma}_{\bar{t}}\in\hat{\Lambda}_{T}(\mathbb{R}^{n})$,
\begin{eqnarray}
\|\gamma_{t}\|_{0} & := & \sup_{0\leq s\leq t}|\gamma_{t}(s)|,\nonumber \\
d_{p}(\gamma_{t},\bar{\gamma}_{\bar{t}}) & := & \sqrt{|t-\bar{t}|}+\sup_{0\leq s\leq\bar{t}}|\gamma_{t,\bar{t}}(s)-\bar{\gamma}_{\bar{t}}(s)|.\label{eq:metric}
\end{eqnarray}
Here $|\cdot|$ is the standard metric of the Euclid space, $d_{p}$
is called parabolic metric. It is easy to verify that $(\hat{\Lambda}_{t}(\mathbb{R}^{n}),\|\cdot\|_{0})$
is a Banach space, $(\hat{\Lambda},d_{p})$ is a complete metric spaces.

\begin{defn} \label{Continuity}(Continuity). Let $\mathbb{E}$ be
a Banach space. A map $v:\hat{\Lambda}\rightarrow\mathbb{E}$ is said
to be continuous at $\gamma_{t}$, if for any $\varepsilon>0$ there
exits $\delta>0$ such that for each $\bar{\gamma}_{\bar{t}}\in\hat{\Lambda}$
such that $d_{p}(\gamma_{t},\bar{\gamma}_{\bar{t}})<\delta$, we have
$|v(\gamma_{t})-v(\bar{\gamma}_{\bar{t}})|<\varepsilon$. $v$ is
said to be continuous on $\hat{\Lambda}$ and is denoted by $v\in\mathscr{C}(\hat{\Lambda},\mathbb{E})$
if $v$ is continuous at each $\gamma_{t}\in\hat{\Lambda}$. Moreover,
we write $v\in\mathscr{C}_{b}(\hat{\Lambda},\mathbb{E})$ if $v$
is bounded, and if $v(\gamma_{t})\leq C(1+\|\gamma_{t}\|_{0})$ for
all $\gamma_{t}\in\hat{\Lambda}$ holds for some constant $C$, we
write $v\in\mathscr{C}_{l}(\hat{\Lambda},\mathbb{E})$ (the subscript
indicates linear growth).

(uniform continuity). A continuous map $v:\hat{\Lambda}\rightarrow\mathbb{E}$
is said to be uniformly continuous, and is denoted by $u\in\mathscr{C}_{u}(\hat{\Lambda},\mathbb{E})$,
if for any $\varepsilon>0$, there exists a $\delta>0$ such that
for each $\gamma_{t},\bar{\gamma}_{\bar{t}}\in\hat{\Lambda}$ satisfying
$d_{p}(\gamma_{t},\bar{\gamma}_{\bar{t}})<\delta$, we have $|v(\gamma_{t})-v(\bar{\gamma}_{\bar{t}})|\le\varepsilon$.

We write $\mathscr{C}(\hat{\Lambda}),\,\mathscr{C}_{b}(\hat{\Lambda}),\,\mathscr{C}_{l}(\hat{\Lambda})$
and $v\in\mathscr{C}_{u}(\hat{\Lambda})$ for $\mathscr{C}(\hat{\Lambda},\mathbb{R}),\,\mathscr{C}_{b}(\hat{\Lambda},\mathbb{R}),\,\mathscr{C}_{l}(\hat{\Lambda},\mathbb{R}),$
and $v\in\mathscr{C}_{u}(\hat{\Lambda},\mathbb{R})$, respectively.\end{defn}

Now we define the vertical and horizontal derivatives of Dupire \cite{dupire2009functional}.

\begin{defn} \label{derivative}(Vertical derivative). Consider functional
$v:\hat{\Lambda}(\mathbb{R}^{n})\rightarrow\mathbb{R}$ and $\gamma_{t}\in\hat{\Lambda}(\mathbb{R}^{n})$.
The vertical (space) derivative of $v$ at $\gamma_{t}$ is defined
as
\[
D_{x}v(\gamma_{t}):=(D_{1}v(\gamma_{t}),\cdots,D_{n}v(\gamma_{t}))
\]
where
\begin{eqnarray}
D_{i}v(\gamma_{t}) & := & \lim_{h\rightarrow0}\frac{1}{h}\big[v(\gamma_{t}^{he_{i}})-v(\gamma_{t})\big],\quad i=1,\cdots n,\label{eq:define-vertical-derivative}
\end{eqnarray}
if all the limits exist, with $e_{i},i=1,\cdots,n,$ being coordinate
unit vectors of $\mathbb{R}^{n}$. If \eqref{eq:define-vertical-derivative}
is well-defined for all $\gamma_{t}$, the map $D_{x}v:=(D_{1}v,\cdots,D_{n}v):\hat{\Lambda}(\mathbb{R}^{n})\rightarrow\mathbb{R}^{n}$
is called the vertical derivative of $v$. We define the Hessian $D_{xx}v(\gamma_{t})$
in an obvious way. Then $D_{xx}v$ is an $\mathbb{S}(n)$-valued functional
defined on $\hat{\Lambda}(\mathbb{R}^{n})$, where $\mathbb{S}(n)$
is the space of all $n\times n$ symmetric matrices.

(Horizontal derivative). The horizontal derivative at $\gamma_{t}\in\hat{\Lambda}$
of a functional $v:\hat{\Lambda}\rightarrow\mathbb{R}$ is defined
as
\begin{eqnarray}
D_{t}v(\gamma_{t}) & := & \lim_{h\rightarrow0,h>0}\frac{1}{h}[v(\gamma_{t,t+h})-v(\gamma_{t})],\label{eq:defination-horizon-derivative}
\end{eqnarray}
if the limit exists. If \eqref{eq:defination-horizon-derivative}
is well-defined for all $\gamma_{t}\in\hat{\Lambda}$, the functional
$D_{t}v:\hat{\Lambda}\rightarrow\mathbb{R}$ is called the horizontal
derivative of $v$. Note that it is a right derivative. \end{defn}

\begin{defn} \label{hight order derivative} Define $\mathscr{C}^{j,k}(\hat{\Lambda})$
as the set of functionals $v:\hat{\Lambda}\rightarrow\mathbb{R}$
which are $j$ times horizontally and $k$ times vertically differentiable
in $\hat{\Lambda}$ such that all these derivatives are continuous.
Moreover, we write $v\in\mathscr{C}_{b}^{j,k}(\hat{\Lambda})$ if
$v$ together with all its derivatives are bounded, and $v\in\mathscr{C}_{l}^{j,k}(\hat{\Lambda})$
if $v\in\mathscr{C}^{j,k}(\hat{\Lambda})$ and $v$ grows in a linear
way.\end{defn}

\begin{rem} For $v(\gamma_{t})=f(t,\gamma_{t}(t))$ with $f\in\mathscr{C}^{1,1}(\mathbb{R}\times\mathbb{R}^{n},\mathbb{R})$,
we have
\[
D_{t}v(\gamma_{t})=\partial_{t}f(t,\gamma_{t}(t)),\quad\quad D_{x}v(\gamma_{t})=\partial_{x}f(t,\gamma_{t}(t)),
\]
which shows the coincidence of Dupire's derivatives with the classical
ones.

\end{rem}

\subsubsection{Space of continuous paths}

Let $\Lambda_{t}(\mathbb{R}^{n}):=\mathbf{C}_{0}([0,t],\mathbb{R}^{n})$
be the set of all continuous $\mathbb{R}^{n}$-valued functions defined
over $[0,t]$ which vanish at time zero, and $\Lambda(\mathbb{R}^{n}):=\bigcup_{t\in[0,T]}\Lambda_{t}(\mathbb{R}^{n})$.
In the sequel, for notational simplicity, we use $\mathbf{0}$ to
denote $\gamma_{0}$ or vectors and matrices whose components are
all zero. Clearly, $\Lambda(\mathbb{R}^{n})\subset\hat{\Lambda}(\mathbb{R}^{n})$.
$(\Lambda_{t}(\mathbb{R}^{n}),\|\cdot\|_{0})$ is a Banach space,
and $(\Lambda(\mathbb{R}^{n}),d_{p})$ is a complete metric space.
We write $\Lambda$ for $\Lambda(\mathbb{R}^{n})$ if there is no
confusion.

Let $\mathbb{E}$ be a Banach space. $\hat{v}:\hat{\Lambda}\rightarrow\mathbb{E}$
and $v:\Lambda\rightarrow\mathbb{E}$ are called consistent on $\Lambda$
if $v$ is the restriction of $\hat{v}$ on $\Lambda$.

\begin{defn} \label{(Continuity)-derivitive-cont-space} Consider
a map $v:\Lambda\rightarrow\mathbb{E}$.

(i) We write $v\in\mathscr{C}(\Lambda)$ if $v$ is continuous at
every path $\gamma_{t}\in\Lambda$ under $d_{p}$. We write $v\in\mathscr{C}_{b}(\Lambda)$
(resp. $v\in\mathscr{C}_{l}(\Lambda)$, $v\in\mathscr{C}_{u}(\Lambda)$)
if $v\in\mathscr{C}(\Lambda)$ and $v$ is bounded (resp. linearly
growth, uniformly continuous).

(ii) We write $v\in\mathscr{C}^{j,k}(\Lambda)$ if there exists $\hat{v}\in\mathscr{C}^{j,k}(\hat{\Lambda})$
which is consistent with $v$ on $\Lambda$, we shall define
\begin{equation}
D_{t}^{i}v:=D_{t}^{i}\hat{v},\, D_{x}^{\beta}v:=D_{x}^{\beta}\hat{v},\qquad\text{on }\Lambda,\label{eq:derivetive on continuous path}
\end{equation}
where $0\leq i\leq j$ and multi index $\beta=(\beta_{1},\cdots,\beta_{n})$
with the non-negative integers $\alpha_{1},\cdots,\alpha_{n}$ satisfying
$\beta_{1}+\cdots+\beta_{n}\leq k$. Similarly, we define the spaces
$\mathscr{C}_{b}^{j,k}(\Lambda)$, $\mathscr{C}_{l}^{j,k}(\Lambda)$
and $\mathscr{C}_{u}^{j,k}(\Lambda)$ in an obvious way.\end{defn}

\begin{rem} By \cite{dupire2009functional} and \cite{cont2010functional},
the derivatives of $v$ in \eqref{eq:derivetive on continuous path}
is independent of the choice of $\hat{v}$, i.e., if $\hat{v}'\in\mathscr{C}^{j,k}(\hat{\Lambda})$
is another functional consistent with $v$, then $D_{x}^{i}\hat{v}=D_{x}^{i}\hat{v}'$
and $D_{x}^{\alpha}\hat{v}=D_{x}^{\alpha}\hat{v}'$ on $\Lambda$.
Therefore, Definition \ref{(Continuity)-derivitive-cont-space} (ii)
is well defined. \end{rem}

\begin{defn}(H\"older continuity). For $\alpha\in(0,1]$, we say
that a functional $v$ defined on $\mathbf{Q}\subset\Lambda$ is H\"older
continuous on $\mathbf{Q}$ with exponent $\alpha$ if the quantity
\[
[v]_{\alpha;\mathbf{Q}}:=\sup_{\gamma_{t},\gamma'_{t'}\in\mathbf{Q},\gamma_{t}\neq\gamma'_{t'}}\frac{|v(\gamma_{t})-\gamma(\gamma'_{t'})|}{d_{p}^{\alpha}(\gamma_{t},\gamma'_{t'})}
\]
 is finite. Let $v\in\mathscr{C}^{1,2}(\Lambda)$, define
\begin{align}
|v|_{\alpha;\mathbf{Q}} & :=\sup_{\gamma_{t}\in\mathbf{Q}}|v(\gamma_{t})|+[\gamma_{t}]_{\alpha;\mathbf{Q}},\nonumber \\
|v|_{2,\alpha;\mathbf{Q}} & :=|v|_{\alpha;\mathbf{Q}}+|D_{t}v|_{\alpha;\mathbf{Q}}+\sum_{1\leq|\beta|\leq2}|D_{x}^{\beta}v|_{\alpha;\mathbf{Q}}.\label{eq:2apha-holde}
\end{align}
 If \eqref{eq:2apha-holde} is finite, we write $v\in\mathscr{C}^{1+\frac{\alpha}{2},2+\alpha}(\mathbf{Q})$.\end{defn}

\subsubsection{Filtration and localization}

Now we introduce the filtration of $\Lambda_{T}$. Let $\mathscr{G}_{T}:=\mathscr{B}(\Lambda_{T})$,
the smallest Borel $\sigma$-field generated by metric space $(\Lambda_{T},\|\cdot\|_{0})$.
For any $t\in[0,T]$, define $\mathscr{G}_{t}:=\theta_{t}^{-1}(\mathscr{G}_{T})=\sigma(\theta_{t}^{-1}(\mathscr{G}_{T}))$,
where $\theta_{t}:\Lambda_{T}\rightarrow\Lambda_{T}$ is the mapping
\begin{equation}
(\theta_{t}\gamma)(s)=\gamma(t\wedge s),\quad0\leq s\leq T,\text{ for any }\gamma\in\Lambda_{T},\label{eq:theta-stop-fun}
\end{equation}
and $\sigma$ means the smallest $\sigma$-field generated by the
underlying class of subsets. $\mathscr{G}:=\{\mathscr{G}_{t},t\in[0,T]\}$
is a filtration. $\mathscr{G}_{t}$ is just the smallest $\sigma$-algebra
generated by the collection of finite-dimensional cylinder sets of
the form
\[
\big\{\gamma\in\Lambda_{T};(\gamma(t_{1}),\cdots\gamma(t_{k}))\in A\big\};\quad k\geq1,A\in\mathscr{B}(\mathbb{R}^{k}),
\]
where, for all $i=1,\cdots,k$, $t_{i}\in[0,t]$. For more details,
see the monograph of Stroock and Varadhan \cite[Section 1.3]{stroock1979multidimensional}.

Define $\pi_{t}:\Lambda_{T}\rightarrow\Lambda_{t}$ as follows:
\[
(\pi_{t}\gamma)(s)=\gamma_{t}(s),\quad(s,\gamma)\in[0,t]\times\Lambda_{T}.
\]
It is easy to observe that $\pi_{t}^{-1}(\mathscr{G}_{T})=\sigma(\pi_{t}^{-1}(\mathscr{G}_{T}))=\mathscr{B}(\Lambda_{t})$,
the smallest Borel $\sigma$-field generated by metric space $(\Lambda_{t},\|\cdot\|_{0})$.

A map $H:[0,T]\times\Lambda_{T}\rightarrow\mathbb{E}$ is called a
functional process. Moreover, we say a process $H$ is adapted to
the filtration $\mathscr{G}$, if $H(t,\cdot)$ is $\mathscr{G}_{t}$-measurable
for any $t\in[0,T]$. Obviously an adapted process $H$ has the property
that, for any $\gamma^{1},\gamma^{2}\in\Lambda_{T}$ satisfying $\gamma^{1}(s)=\gamma^{2}(s)$
for all $s\in[0,t]$, $H(t,\gamma^{1})=H(t,\gamma^{2})$. Hence $H(t,\gamma_{T})$
can be view as $H(\gamma_{t})$, that is to say a functional adapted
process equals a path functional defined on $\Lambda$. Let $B=\{B(t),t\in[0,T]\}$
be the canonical process on $\Lambda$, i.e. $B(t,\gamma)=\gamma(t)$.
Define $B_{t}:=\{B(s),s\in[0,t]\}$, and it is $\mathscr{G}_{t}$-measurable.
For any continuous map $v:\Lambda\rightarrow\mathbb{E}$, we know
that $\{v(B_{t}),t\in[0,T]\}$ is an adapted functional process.

Let $P_{0}$ be the Wiener measure of space $(\Lambda_{T},\mathscr{G})$,
under which the canonical process $B$ (i.e. $B(t,\gamma)=\gamma(t)$)
is a standard Brownian motion. For any integrable $\mathscr{G}_{T}$-measurable
variable $\xi$, denote $P_{0}[\xi]$ the integration of $\xi$ under
measure $P_{0}$.

Let $X:=\{X(t),t\in[0,T]\}$ be an $n$-dimensional adapted continuous
stochastic process on the probability space $(\Omega,\mathscr{F},(\mathscr{F}_{t})_{0\leq t\leq T},P)$.
For $t\in[0,T]$, $X(t)$ is the value of $X$ at time $t$ and $X_{t}$
is the path of $X$ up to time $t$, i.e., $X_{t}:=\{X(r),r\in[0,t]\}$.
$X$ can be viewed as a map from $\Omega$ to $\Lambda_{T}$, the
continuity and adaption of $X$ imply that $X_{t}$ is $\mathscr{F}_{t}/\mathscr{G}_{t}$-measurable
for any $t\in[0,T]$. For any $v\in\mathscr{C}(\Lambda)$, $\{v(X_{t}),t\in[0,T]\}$
is an $\mathscr{F}$-adapted stochastic process.

The following functional It\^o formula was initiated by Dupire \cite{dupire2009functional}
and later extended to a more general context by Cont \& Fournie \cite{cont2010functional}.

\begin{thm}[Functional It\^o formula] Suppose $X$ is a continuous
semi-martingale and $v\in\mathscr{C}^{1,2}(\Lambda)$. Then for any
$t\in[0,T]$:
\begin{align}
v(X_{t}(\omega))-v(X_{0}(\omega)) & =\int_{0}^{t}D_{s}v(X_{s}(\omega))ds+\frac{1}{2}\int_{0}^{t}D_{xx}v(X_{s}(\omega))d\langle X\rangle(s,\omega)\nonumber \\
 & +\int_{0}^{t}D_{x}v(X_{s}(\omega))dX(s,\omega),\quad\quad\qquad\text{a.s. -}\omega.\label{eq:ito formu}
\end{align}
\end{thm}

A time functional $\tau:\Lambda_{T}\rightarrow[0,T]$ (resp. $\tau:\Lambda_{T}\rightarrow[0,T]$)
is said to be a $\mathscr{G}$-stopping time, if $\{\gamma_{t}:\gamma\in\Lambda_{T},\,\tau(\gamma)\leq t\}\in\mathscr{G}_{t}$
for every $t\in[0,T]$. Let $\mathscr{T}$ be the set of all $\mathscr{G}$-stopping
times. Define $\mathscr{T}_{+}:=\{\tau\in\mathscr{T}:\tau>0\}$.

Let $u\in\mathscr{C}(\Lambda)$. Obviously, for any constant $\lambda>0$,
$\tau(\gamma):=\inf\big\{ t:u(\gamma_{t})\geq\lambda\big\}\wedge T$
is a stopping time.





\subsubsection{Space shift }

For any fixed $\gamma_{t}\in\hat{\Lambda}$ and $s\in[t,T]$, we introduce
the shifted spaces of cadlag and continuous paths.

Let $\hat{\Lambda}_{s}^{\gamma_{t}}:=\{\bar{\gamma}_{s}\in\hat{\Lambda}_{s};\bar{\gamma}_{s}(r)=\gamma_{t}(r)\text{ for any }r\in[0,t]\}$
be the space of cadlag paths originating at $\gamma_{t}$, and $\hat{\Lambda}^{\gamma_{t}}:=\cup_{t\leq s\leq T}\hat{\Lambda}_{s}^{\gamma_{t}}$.
Let $\Lambda_{s}^{\gamma_{t}}:=\{\bar{\gamma}_{s}\in\hat{\Lambda}_{s}^{\gamma_{t}};\bar{\gamma}_{s}\text{ is continuous on }[t,s]\}$
be the spaces with continuous paths originating at $\gamma_{t}$ and
$\Lambda_{s}^{\gamma_{t}}:=\cup_{t\leq s\leq T}\Lambda_{s}^{\gamma_{t}}$.
It is easy to verify that $(\hat{\Lambda}_{s}^{\gamma_{t}},\|\cdot\|_{0})$
and $(\Lambda_{s}^{\gamma_{t}},\|\cdot\|_{0})$ are Banach spaces,
and $(\hat{\Lambda}^{\gamma_{t}},d_{p})$ and $(\Lambda^{\gamma_{t}},d_{p})$
are complete metric spaces.

For any map $v:\hat{\Lambda}^{\gamma_{t}}\rightarrow\mathbb{R}$,
define the derivatives in the spirit of Definition \ref{derivative},
and define the spaces $\mathscr{C}(\hat{\Lambda}^{\gamma_{t}})$,
$\mathscr{C}_{b}(\hat{\Lambda}^{\gamma_{t}})$, $\mathscr{C}_{u}^{j,k}(\hat{\Lambda}^{\gamma_{t}})$
in the spirit of Definitions \ref{Continuity} and \ref{hight order derivative}.
We define $\mathscr{C}(\Lambda^{\gamma_{t}})$, $\mathscr{C}_{b}(\Lambda^{\gamma_{t}})$,
$\mathscr{C}_{u}^{j,k}(\Lambda^{\gamma_{t}})$ in the spirit of Definition
\ref{(Continuity)-derivitive-cont-space}.

Let $\mathscr{G}_{T}^{\gamma_{t}}:=\mathscr{B}(\Lambda_{T}^{\gamma_{t}})$
, the smallest Borel $\sigma$-field generated by metric space $(\Lambda_{T}^{\gamma_{t}},\|\cdot\|)$.
For $s\in[t,T]$, define $\mathscr{G}_{s}^{\gamma_{t}}:=\theta_{s}^{-1}(\mathscr{G}_{T}^{\gamma_{t}})=\sigma(\theta_{s}^{-1}(\mathscr{G}_{T}^{\gamma_{t}}))$,
with $\theta$ being defined by \eqref{eq:theta-stop-fun}. $\mathscr{G}^{\gamma_{t}}:=\{\mathscr{G}_{s}^{\gamma_{t}},s\in[t,T]\}$
is said to be $\mathscr{G}^{\gamma_{t}}$-filtration. A time functional
$\tau:\Lambda_{T}^{\gamma_{t}}\rightarrow[t,T]$ is called $\mathscr{G}^{\gamma_{t}}$-stopping
time, if $\{\gamma_{s}:\gamma\in\Lambda_{T}^{\gamma_{t}},\,\tau(\gamma)\leq s\}\in\mathscr{G}_{s}^{\gamma_{t}}$
for any $s\in[t,T]$. The set of all $\mathscr{G}^{\gamma_{t}}$-stopping
is denoted $\mathscr{T}^{\gamma_{t}}$. Denote $\mathscr{T}_{+}^{\gamma_{t}}:=\{\tau\in\mathscr{T}_{+}^{\gamma_{t}}:\tau>t\}$.


Let $\{H_{t},t\in[0,T]\}$ be a $\mathscr{G}$-progressively measurable
functional process, and $\xi$ be a $\mathscr{G}_{T}$ measurable
functional variable, let $\gamma_{t}\in\hat{\Lambda}$. Define the
process $H^{\gamma_{t}}$ on $[0,T]\times\Lambda_{T}^{\gamma_{t}}$
and the variable $\xi^{\gamma_{t}}$ on $\Lambda_{T}^{\gamma_{t}}$,
as the restriction on $\Lambda^{\gamma_{t}}$ of $H$ and $\xi$,
respectively; that is,

\begin{align*}
H^{\gamma_{t}}: & =H\big|_{[0,T]\times\Lambda_{T}^{\gamma_{t}}},\quad\quad\xi^{\gamma_{t}}:=\xi\big|_{\Lambda_{T}^{\gamma_{t}}}.
\end{align*}
Then, $\{H_{s}^{\gamma_{t}},s\in[t,T]\}$ is a $\mathscr{G}^{\gamma_{t}}$-progressively
measurable functional process, and $\xi^{\gamma_{t}}$ is a $\mathscr{G}_{T}^{\gamma_{t}}$-measurable
functional variable.


\subsubsection{Space of $\alpha$-H\"older continuous paths}

For any $\alpha\in(0,1]$, we say that $\gamma\in\Lambda(\mathbb{R}^{n})$
is $\alpha$-H\"older continuous if
\[
\llbracket\gamma_{t}\rrbracket_{\alpha}:=\sup_{0\leq s<r\leq t}\frac{|\gamma_{t}(s)-\gamma_{t}(r)|}{|s-r|^{\alpha}}<\infty.
\]
We call $\llbracket\gamma_{t}\rrbracket_{\alpha}$ the $\alpha$-H\"older
modulus of $\gamma_{t}$. Define the $\alpha$-H\"older space:
\[
\mathbf{C}^{\alpha}(\mathbb{R}^{n}):=\{\gamma_{t}\in\Lambda:\llbracket\gamma_{t}\rrbracket_{\alpha}<\infty\}.
\]
Clearly, $\mathbf{C}^{\alpha}(\mathbb{R}^{n})\subset\Lambda(\mathbb{R}^{n})$.
We write $\mathbf{C}^{\alpha}$ for $\mathbf{C}^{\alpha}(\mathbb{R}^{n})$
if there is no confusion.

For any $\alpha\in(0,1]$ and $\mu>0$, denote
\begin{align}
\mathbf{C}_{\mu}^{\alpha} & :=\{\gamma_{t}\in\Lambda:\llbracket\gamma_{t}\rrbracket_{\alpha}\leq\mu\}.\label{eq:D-nu-m compact-set}
\end{align}

Analogous to the proof of the Arzela-Ascoli theorem, we show the following
compact property.

\begin{prop} For $\alpha\in(0,1]$, $\mathbf{C}_{\mu}^{\alpha}$
is a compact subset of $(\Lambda,d_{p})$.\end{prop}

\begin{proof} Since $(\Lambda,d_{p})$ is a metric space, it suffices
to prove that every sequence $\{\gamma_{t_{k}}^{k}\}_{k=1}^{\infty}\subset\mathbf{C}_{\mu}^{\alpha}$
has a convergence subsequence and the limit lies in $\mathbf{C}_{\mu}^{\alpha}$.

For $\{t_{k}\}_{k=1}^{\infty}\subset[0,T]$, it has a convergence
subsequence, still denoted by $\{\gamma_{t_{k}}^{k}\}_{k=1}^{\infty}$.
Without loss of generality, we suppose that $t_{k}$ converges increasingly
to$t\in[0,t]$ as $k\to\infty$. From \eqref{eq:D-nu-m compact-set},
we have
\begin{equation}
|\gamma_{t_{k}}^{k}(r)|\leq|\gamma_{t_{k}}^{k}(0)|+\mu r^{\alpha}\leq\mu T^{\alpha},\,\forall r\in[0,t_{k}],k=1,2\cdots.\label{eq:gamma-k bounded}
\end{equation}
Let $\mathbb{Q}$ be the set of rational numbers, and $\{r_{1},r_{2},\cdots\}$
be an enumeration of $\mathbb{Q}\cap[0,t)$. By \eqref{eq:gamma-k bounded},
we can choose a subsequence $\{\gamma_{t_{k}^{(1)}}^{(1)k}\}_{k=1}^{\infty}$
of $\{\gamma_{t_{k}}^{k}\}_{k=1}^{\infty}$ such that $r_{1}\in[0,t_{k}^{(1)}]$
for all $k=1,2,\cdots$ and $\{\gamma_{t_{k}^{(1)}}^{(1)k}(r_{1})\}_{k=1}^{\infty}$
converges to a limit, denoted by $\gamma_{t}(r_{1})$. From $\{\gamma_{t_{k}^{(1)}}^{(1)k}\}_{k=1}^{\infty}$,
choose a further subsequence $\{\gamma_{t_{k}^{(2)}}^{(2)k}\}_{k=1}^{\infty}$
such that $r_{2}\in[0,t_{k}^{(2)}]$ for all $k=1,2,\cdots$ and $\{\gamma_{t_{k}^{(2)}}^{(2)k}(r_{2})\}_{k=1}^{\infty}$
converges to a limit $\gamma_{t}(r_{2})$. Continue this process,
and then let $\{\bar{\gamma}_{\bar{t}_{k}}^{k}\}_{k=1}^{\infty}=\{\gamma_{t_{k}^{(k)}}^{(k)k}\}_{k=1}^{\infty}$
be the ``diagonal sequence''. We have $\{\bar{\gamma}_{\bar{t}_{k}}^{k}(r):r\in[0,\bar{t}_{k}]\text{for all }k\}$
has a unique accumulation point $\gamma_{t}(r)$ for any $r\in\mathbb{Q}\cap[0,t)$.

For any $r,s\in\mathbb{Q}\cap[0,t)$,
\[
|\bar{\gamma}_{\bar{t}_{k}}^{k}(r)-\bar{\gamma}_{\bar{t}_{k}}^{k}(s)|\leq\mu|r-s|^{\alpha},\quad\forall k\geq K,
\]
where $K$ is a sufficiently large integer such that $r,s\in[t,\bar{t}_{k}]$
for all $k\geq K$. Setting $n\to\infty$, we have
\[
|\gamma_{t}(r)-\gamma_{t}(s)|\leq\mu|r-s|^{\alpha},\quad\forall s,r\in\mathbb{Q}\cap[0,t).
\]
Hence $\gamma_{t}$ has a continuous extension on $[0,t]$, still
denoted by $\gamma_{t}$, and it lies in $\mathbf{C}_{\mu}^{\alpha}$.

It remains to show the limit $\lim_{k\to\infty}d_{p}(\bar{\gamma}_{\bar{t}_{k}}^{k},\gamma_{t})=0$.
In fact, for any $\varepsilon>0$, define
\[
r_{j}=j(\frac{\varepsilon}{\mu})^{\frac{1}{\alpha}},\quad j=1,2,\cdots,\lfloor t(\frac{\mu}{\varepsilon})^{\frac{1}{\alpha}}\rfloor,
\]
where $\lfloor s\rfloor$ denotes the greatest integer less than or
equal to $s$. Then, for $s\in[0,t]$, there is some $r_{j}$ such
that $|r_{j}-s|\leq(\frac{\varepsilon}{\mu})^{\frac{1}{\alpha}}$.
For a sufficiently large $K$, we have $r_{k}\in[0,\bar{t}_{k}]$
and $|\bar{\gamma}_{\bar{t}_{k}}^{k}(r_{j})-\gamma_{t}(r_{j})|<\varepsilon$,
for all $j=1,2,\cdots\lfloor t(\frac{\mu}{\varepsilon})^{\frac{1}{\alpha}}\rfloor$
and $k>K$. Consequently,
\begin{align*}
|\bar{\gamma}_{\bar{t}_{k}}^{k}(s)-\gamma_{t}(s)| & \leq|\bar{\gamma}_{\bar{t}_{k}}^{k}(s)-\bar{\gamma}_{\bar{t}_{k}}^{k}(r_{j})|+|\bar{\gamma}_{\bar{t}_{k}}^{k}(r_{j})-\gamma_{t}(r_{j})|+|\gamma_{t}(r_{j})-\gamma_{t}(s)|\\
 & \leq3\varepsilon,\quad\qquad\qquad\forall s\in[0,t],\, k>K.
\end{align*}
Furthermore,
\[
d_{p}(\bar{\gamma}_{\bar{t}_{k}}^{k},\gamma_{t})=\max_{0\leq s\leq t}d_{p}(\bar{\gamma}_{\bar{t}_{k},t}^{k}(s),\gamma_{t}(s))+\sqrt{|\bar{t}_{k}-t|}\leq\varepsilon+\sqrt{|\bar{t}_{k}-t|},
\]
which leads to $\lim_{k\to\infty}d_{p}(\bar{\gamma}_{\bar{t}_{k}}^{k},\gamma_{t})=0$.
\end{proof}

Define the following (random) time for the path to oscillate beyond
a given $\alpha$-H\"older modulus:
\[
\tau_{\mu}^{\alpha}(\gamma):=\inf\{t>0:\llbracket\gamma_{t}\rrbracket_{\alpha}>\mu\},\quad\gamma\in\Lambda_{T}.
\]
It is a $\mathscr{G}$-stopping time due to the following
\begin{align*}
\{\tau_{\mu}^{\alpha}\leq t\} & =\left\{ \gamma\in\Lambda_{T}:\llbracket\gamma_{t}\rrbracket_{\alpha}\leq\mu\right\} =\left\{ \gamma\in\Lambda_{T}:\sup_{s,r\in\mathbb{Q}\cap[0,t],s\neq r}\frac{|\gamma(s)-\gamma(t)|}{|s-r|^{\alpha}}\leq\mu\right\} \\
 & =\bigcap_{s,r\in\mathbb{Q}\cap[0,t],s\neq r}\left\{ \gamma\in\Lambda_{T}:|\gamma(s)-\gamma(t)|\leq\mu|s-r|^{\alpha}\right\} \in\mathscr{G}_{t}.
\end{align*}
This kind of exit time will play a crucial role in the subsequent
proof of the existence of a viscosity solution.

\subsection{Backward stochastic differential equations}

Let $(\Omega,\mathscr{F},(\mathscr{F}_{t})_{0\leq t\leq T},P)$ be
a probability space with the usual condition (see Karatzas and Shreve
\cite{karatzas1991brownian}), and $\{W(t),t\in[0,T]\}$ be a $d$-dimensional
standard Brownian motion. Let $\mathscr{N}$ be the collection of
all $P$-null sets in $\Omega,$ for any $0\leq t\leq r\leq T$, $\mathscr{F}_{r}^{t}$
denotes the completion of $\sigma(W(s)-W(t);t\leq s\leq r)$, i.e.,
$\mathscr{F}_{r}^{t}:=\sigma(W(s)-W(t);t\leq s\leq r)\cup\mathscr{N}$.
We also write $\mathscr{F}^{t}$ for $\{\mathscr{F}_{s}^{t},s\in[t,T]\}$.

For any $t\in[0,T]$, denote by $\mathscr{M}^{2}(t,T)$ the space
of all $\mathscr{F}^{t}$-adapted, $\mathbb{R}^{d}$-valued processes
$\{Y(s),s\in[t,T]\}$ such that $E[\int_{t}^{T}|Y(s)|^{2}ds]<\infty$
and by $\mathscr{S}^{2}(0,T)$ the space of all $\mathscr{F}^{t}$-adapted,
$\mathbb{R}$-valued continuous processes $\{Y(s),s\in[t,T]\}$ such
that $E[\sup_{s\in[t,T]}|Y(s)|^{2}]<\infty$.

\begin{lem} \label{lem:bsde1} Consider $f:\Omega\times[0,T]\times\mathbb{R}\times\mathbb{R}^{d}\rightarrow\mathbb{R}$
such that $\{f(t,y,z),t\in[0,T]\}$ is progressively measurable for
each $(y,z)\in\mathbb{R}\times\mathbb{R}^{d}$, and the following
two conditions are satisfied:

(i) $f$ is uniformly Lipschitz continuous about $(y,z)\in\mathbb{R}\times\mathbb{R}^{d}$;

(ii) $f(\cdot,0,\mathbf{0})\in\mathscr{M}^{2}(0,T)$. For any $\xi\in L^{2}(\Omega,\mathscr{F}_{T},P)$,
the BSDE
\begin{equation}
Y(t)=\xi+\int_{t}^{T}f(s,Y(t),Z(s))ds-\int_{t}^{T}Z(s)dW_{s},\quad0\leq t\leq T,
\end{equation}
has a unique adapted solution $(Y,Z)\in\mathscr{S}^{2}(0,T)\times\mathscr{M}^{2}(0,T).$
\end{lem}

We recall the following comparison theorem on BSDEs (see El Karoui,
Peng and Quenez \cite{el1997backward})

\begin{lem} \label{lem:comparison-theorem} Let two BSDEs of data
$(\xi_{1},f_{1})$ and $(\xi_{2},f_{2})$, satisfy all the assumptions
of Lemma~\ref{lem:bsde1}. Denote by $(Y^{1},Z^{1})$ and $(Y^{2},Z^{2})$
their respective adapted solutions. Then we have:

(1) (Monotonicity). If $\xi_{1}\geq\xi_{2}$ and $\tilde{f}_{1}\geq\tilde{f}_{2}$,
a.s., then $Y^{1}(t)\geq Y^{2}(t)$, a.s., for all $t\in[0,T]$.

(2) (Strict monotonicity). If, in addition to (1), we also $P\{\xi_{1}>\xi_{2}\}$>0,
then $P\{Y^{1}(t)>Y^{2}(t)\}>0$ for any $t\in[0,T)$, and in particular,
$Y^{1}(0)>Y^{2}(0)$. \end{lem}

\section{\label{sec:Stochastic-optimal-problems} Formulation of the path-dependent
optimal stochastic control problem and dynamic programming principle}

Let the set of admissible control processes $\mathcal{U}$ be the
set of all $\mathscr{F}$-progressively measurable process valued
in some compact metric space $U$. For any $t\in[0,T]$, $p\geq1$,
$L^{p}(\Omega,\mathscr{F}_{t};\Lambda_{t},\mathscr{G}_{t})$ is the
set of all $\mathscr{F}_{t}/\mathscr{G}_{t}$-measurable maps $\Gamma_{t}:\Omega\rightarrow\Lambda_{t}$
satisfying $E\|\Gamma_{t}\|_{0}^{p}<\infty$.

Consider the following functionals $b:\Lambda\times U\rightarrow\mathbb{R}^{n}$,
$\sigma:\Lambda\times U\rightarrow\mathbb{R}^{n\times d}$, $g:\Lambda_{T}\rightarrow\mathbb{R}$
and $f:\Lambda\times\mathbb{R}\times\mathbb{R}^{d}\times U\rightarrow\mathbb{R}$.
We make the following assumption.

\medskip{}

(H1) \textit{There exists a constant $C>0$ such that, for all } $(t,\gamma_{T},y,z,u),(t',\gamma'_{T},y',z',u')\in[0,T]\times\Lambda_{T}\times\mathbb{R}\times\mathbb{R}^{d}\times U$,
\begin{align*}
|b(\gamma_{t},u)-b(\gamma'_{t'},u')| & \leq C(d_{p}(\gamma_{t},\gamma'_{t'})+|u-u'|),\\
|\sigma(\gamma_{t},u)-\sigma(\gamma'_{t'},u')| & \leq C(d_{p}(\gamma_{t},\gamma'_{t'})+|u-u'|),\\
|f(\gamma_{t},y,z,u)-f(\gamma'_{t'},y',z',u')| & \leq C(d_{p}(\gamma_{t},\gamma'_{t'})+|y-y'|+|z-z'|+|u-u'|),\\
|g(\gamma_{T})-g(\gamma'_{T})| & \leq C\|\gamma_{T}-\gamma'_{T}\|_{0}.
\end{align*}

\medskip{}

For given $t\in[0,T)$, $\mathscr{F}_{t}/\mathscr{G}_{t}$-measurable
map $\Gamma_{t}:\Omega\rightarrow\Lambda_{t}$ and admissible control
$u\in\mathcal{U}$, consider the following SDE:
\begin{equation}
\left\{ \begin{array}{rcl}
X^{\Gamma_{t},u}(s) & = & \Gamma_{t}(s),\quad\text{all }\omega,\, s\in[0,t];\\
X^{\Gamma_{t},u}(s) & = & {\displaystyle \Gamma_{t}(t)+\int_{t}^{s}b(X_{r}^{\Gamma_{t},u},u(r))\, dr}\\
 &  & {\displaystyle +\int_{t}^{s}\sigma(X_{r}^{\Gamma_{t},u},u(r))\, dW(r),\quad\text{a.s.-}\omega,\text{ }s\in[t,T].}
\end{array}\right.\label{eq:diffu proc}
\end{equation}

\begin{lem} \label{lem:FSDE} Take $p\geq2$. Let Assumption (H1)
hold. For
\[
t\in[0,T],\Gamma_{t}\in L^{p}(\Omega,\mathscr{F}_{t};\Lambda_{t},\mathscr{B}(\Lambda_{t})),\text{ and }u\in\mathcal{U},
\]
SDE \eqref{eq:diffu proc} admits a unique strong solution $X^{\Gamma_{t},u}:\Omega\rightarrow\Lambda_{T}$
such that $X_{s}^{\Gamma_{t},u}:\Omega\rightarrow\Lambda_{s}$ is
$\mathscr{F}_{s}/\mathscr{B}(\Lambda_{s})$ measurable for all $s\in[t,T]$,
and $E[\|X_{T}^{\Gamma_{t},u}\|_{0}^{p}]<\infty$. Moreover, there
is a positive constant $C_{p}$ such that for any $t\in[0,T]$, $u,u'\in\mathcal{U}$,
and $\Gamma_{t},\Gamma_{t}'\in L^{p}(\Omega,\Lambda_{t};\mathscr{F}_{t}/\mathscr{B}(\Lambda_{t}))$,
we have, $P$-a.s.,
\begin{align*}
E[\|X_{T}^{\Gamma_{t},u}-X_{T}^{\Gamma_{t}',u'}\|_{0}^{p}|\mathscr{F}_{t}] & \leq C_{p}\big(\|\Gamma_{t}-\Gamma_{t}'\|_{0}^{p}+E\big[\int_{t}^{T}|u(r)-u'(r)|^{p}dr\big|\mathscr{F}_{t}\big]\big),\\
E[\|X_{T}^{\Gamma_{t},u}\|_{0}^{p}|\mathscr{F}_{t}] & \leq C_{p}\big(1+\|\Gamma_{t}\|_{0}^{p}\big),\\
E[\|X_{r}^{\Gamma_{t},u}-\Gamma_{t,r}\|_{0}^{p}|\mathscr{F}_{t}] & \leq C_{p}\big(1+\|\Gamma_{t}\|_{0}^{p}\big)(r-t)^{\frac{p}{2}},\quad r\in[t,T].
\end{align*}
The constant $C_{p}$ only depends on the Lipschitz constant of $b$
and $\sigma$ in $(\gamma_{t},t)$. \end{lem}

Combing Lemmas~\ref{lem:bsde1} and~\ref{lem:FSDE}, we have

\begin{lem} \label{lem:FBSDE} Take $p\geq2$. Let (H1) hold. For
any $t\in[0,T]$, $\Gamma_{t}\in L^{p}(\Omega,\mathscr{F}_{t};\Lambda_{t},\mathscr{B}(\Lambda_{t}))$
and $u\in\mathcal{U}$, $X^{\Gamma_{t},u}$ is the solution of the
stochastic equation \eqref{eq:diffu proc}. Then BSDE
\begin{eqnarray}
Y^{\Gamma_{t},u}(s) & = & g(X_{T}^{\Gamma_{t},u})+\int_{s}^{T}f(X_{r}^{\Gamma_{t},u},Y^{\Gamma_{t},u}(r),Z^{\Gamma_{t},u}(r),u(r))\, dr\nonumber \\
 &  & -\int_{s}^{T}Z^{\Gamma_{t},u}(r)\, dB(r),\quad\text{a.s.-}\omega,\,\text{all }s\in[t,T],\label{eq:BSDE}
\end{eqnarray}
has a unique solution $(Y^{\Gamma_{t},u},Z^{\Gamma_{t},u})\in\mathscr{S}^{2}(t,T)\times\mathscr{M}^{2}(t,T).$
Furthermore, there is a constant $C_{p}$ such that for any $t\in[0,T]$,
$\Gamma_{t},\Gamma_{t}'\in L^{p}(\Omega,\mathscr{F}_{t};\Lambda_{t},\mathscr{B}(\Lambda_{t}))$,
and $u\in\mathcal{U}$, $P$-a.s.,
\begin{align}
E\Big[\sup_{t\leq s\leq T}|Y^{\Gamma_{t},u}(s)-Y^{\Gamma'_{t},u'}(s)|^{p}\big|\mathscr{F}_{t}\Big] & \leq C_{p}\big(\|\Gamma_{t}-\Gamma'_{t}\|_{0}^{p}+E\big[\!\!\int_{t}^{T}\!\!|u(r)-u'(r)|^{p}dr\big|\mathscr{F}_{t}\big]\big),\label{eq:regularity of Y 1}\\
E\Big[\sup_{t\leq s\leq T}|Y^{\Gamma_{t},u}(s)|^{p}\big|\mathscr{F}_{t}\Big] & \leq C_{p}(1+\|\Gamma_{t}\|_{0}^{p}),\label{eq:regularity of Y 2}\\
E\Big[\sup_{t\leq s\leq r}|Y^{\Gamma_{t},u}(s)-Y^{\Gamma_{t},u}(t)|^{p}\big|\mathscr{F}_{t}\Big] & \leq C_{p}\big(1+\|\Gamma_{t}\|_{0}^{p}\big)(r-t)^{\frac{p}{2}}.\label{eq:regularity of Y 3}
\end{align}
\end{lem}

For the particular case of a deterministic $\Gamma_{t}$, i.e. $\Gamma_{t}=\gamma_{t}\in\Lambda_{t}$:
\begin{equation}
\begin{cases}
X^{\gamma_{t},u}(s)=\gamma_{t}(s), & \text{all }\omega,s\in[0,t);\\
{\displaystyle X^{\gamma_{t},u}(s)=\gamma_{t}(t)+\int_{t}^{s}\! b(X_{r}^{\gamma_{t},u},u(r))\, dr+\int_{t}^{s}\!\sigma(X_{r}^{\gamma_{t},u},u(r))\, dW(r),} & \text{a.s.-}\omega,s\in[t,T].
\end{cases}\label{eq:diffu proc-2}
\end{equation}

\begin{eqnarray}
Y^{\gamma_{t},u}(s) & = & g(X_{T}^{\gamma_{t},u})+\int_{s}^{T}f(X_{r}^{\gamma_{t},u},Y^{\gamma_{t},u}(r),Z^{\gamma_{t},u}(r),u(r))\, dr\nonumber \\
 &  & -\int_{s}^{T}Z^{\gamma_{t},u}(r)\, dW(r),\quad\text{a.s.-}\omega,\,\text{all }s\in[t,T].\label{eq:BSDE-1}
\end{eqnarray}

Given the control process $u\in\mathcal{U}$, we introduce the following
cost functional:
\[
J(\gamma_{t},u):=Y^{\gamma_{t},u}(t),\quad\text{P-a.s.}\quad\forall\gamma_{t}\in\Lambda.
\]
The value functional of the optimal control is defined by
\begin{equation}
\tilde{v}(\gamma_{t}):=\esssup_{u\in\mathcal{U}}Y^{\gamma_{t},u}(t),\quad\forall\gamma_{t}\in\Lambda.\label{eq:defination Of value fun}
\end{equation}
We easily prove that for $t\in[0,T]$ and $\Gamma_{t}\in L^{2}(\Omega,\Lambda_{t};\mathscr{F}_{t}/\mathscr{B}(\Lambda_{t}))$,
\[
J(\Gamma_{t},u)=Y^{\Gamma_{t},u}(t),\quad P\text{-a.s.}.
\]

\begin{rem} The above essential supremum should be understood as
one with respect to indexed families of random variables (see Karatzas
and Shreve \cite[Appendix A]{karatzas1998methods} for details). For
the convenience of reader we recall the notion of esssup of random
variables. Given a family of real-valued random variables $\eta_{\alpha}$,
$\alpha\in I$, a random variable $\eta$ is said to be $\esssup_{\alpha\in I}\eta_{\alpha}$,
if\end{rem}
\begin{enumerate}
\item $\eta\leq\eta_{\alpha}$, $P$-a.s., for any $\alpha\in I$;
\item if there is another random variable $\xi$ such that $\xi\leq\eta_{\alpha}$,
$P$-a.s., for any $\alpha\in I$, then $\xi\leq\eta$, $P$-a.s..
\end{enumerate}
The existence of $\esssup_{\alpha\in I}\eta_{\alpha}$ is well known.

Under Assumption (H1), the random variable $\tilde{v}(\gamma_{t})\in L^{p}(\Omega)$
is $\mathscr{F}_{t}$-measurable. We have

\begin{prop} \label{determin}The value functional $\tilde{v}$ is
deterministic.\end{prop}

\begin{proof} Firstly we show that there exist $\{u_{n}\}_{n=1}^{\infty}\subset\mathcal{U}$
such that
\begin{equation}
\tilde{v}(\gamma_{t})=\esssup_{u(\cdot)\in\mathcal{U}}Y^{\gamma_{t},u}(t)=\lim_{n\to\infty}\nearrow Y^{\gamma_{t},u_{n}}(t).\label{eq:prop3.4-0}
\end{equation}
In view of \cite[Theorem A.3]{karatzas1998methods}, it is sufficient
to prove that, for any $u_{1},\, u_{2}\in\mathcal{U}$, we have
\begin{equation}
Y^{\gamma_{t},u_{1}}(t)\vee Y^{\gamma_{t},u_{2}}(t)=Y^{\gamma_{t},u}(t),\quad P\text{-a.e.}\label{eq:prop3.4-1}
\end{equation}
for $u\in\mathcal{U}$ such that $u(s):=u_{1}(s)\chi_{A_{1}}+u_{2}(s)\chi_{A_{2}},\, s\in[t,T]$,
where $A_{1}:=\{Y^{\gamma_{t},u_{1}}(t)>Y^{\gamma_{t},u_{2}}(t)\}$
and $A_{2}:=\{Y^{\gamma_{t},u_{1}}(t)\leq Y^{\gamma_{t},u_{2}}(t)\}$.
Since $\sum_{i=1,2}\varphi(x_{i})\chi_{A_{i}}=\sum_{i=1,2}\varphi(x_{i}\chi_{A_{i}})$,
we have

\begin{align}
\sum_{i=1,2}\chi_{A_{i}}X^{\gamma_{t},u_{i}}(s) & =\gamma(t)+\int_{t}^{s}\! b(\sum_{i=1,2}\chi_{A_{i}}X_{r}^{\gamma_{t},u_{i}},\sum_{i=1,2}\chi_{A_{i}}u_{i}(r))\, dr\label{eq:prop3.4-2}\\
 & +\int_{t}^{s}\!\!\sigma(\sum_{i=1,2}\chi_{A_{i}}X_{r}^{\gamma_{t},u_{i}},\sum_{i=1,2}\chi_{A_{i}}u_{i}(r))\, dW(r),\quad s\in[t,T]\nonumber
\end{align}
and
\begin{align}
 & \sum_{i=1,2}\chi_{A_{i}}Y^{\gamma_{t},u_{i}}(s)\label{eq:prop3.4-3}\\
= & \sum_{i=1,2}\chi_{A_{i}}g(X_{T}^{\gamma_{t},u_{i}})\nonumber \\
 & +\int_{s}^{T}\!\! f(\sum_{i=1,2}\chi_{A_{i}}X_{r}^{\gamma_{t},u_{i}},\sum_{i=1,2}\chi_{A_{i}}Y^{\gamma_{t},u_{i}}(r),\sum_{i=1,2}\chi_{A_{i}}Z^{\gamma_{t},u_{i}}(r),\sum_{i=1,2}\chi_{A_{i}}u_{i}(r))\, dr\nonumber \\
 & -\int_{s}^{T}\sum_{i=1,2}\chi_{A_{i}}Z^{\gamma_{t},u_{i}}(r)\, dW(r),\quad s\in[t,T].\nonumber
\end{align}
By the uniqueness of solution of BSDE, we have
\begin{equation}
Y^{\gamma_{t},u}=\sum_{i=1,2}\chi_{A_{i}}Y^{\gamma_{t},u_{i}},\quad P\text{-a.e.}.\label{eq:prop3.4-4}
\end{equation}
From Lemma \ref{lem:comparison-theorem}, we have $\sum_{i=1,2}\chi_{A_{i}}Y^{\gamma_{t},u_{i}}(t)=Y^{\gamma_{t},u_{1}}(t)\vee Y^{\gamma_{t},u_{2}}(t),$
which yields \eqref{eq:prop3.4-1}.

Suppose that $\{u_{i}(\cdot)\}_{i=1}^{\infty}\subset\mathcal{U}$
satisfy \eqref{eq:prop3.4-0}. Since $Y^{\gamma_{t},u}$ is continuous
in $u\in\mathcal{U}$, we suppose without lost of generality that
$u_{i}(\cdot)$ takes the following form:
\[
u_{i}(s)=\sum_{j=1}^{N}\chi_{A_{i,j}}u_{ij}(s),\quad s\in[t,T].
\]
Here, $u_{ij}(s)\in\mathscr{F}_{s}^{t},\, t\leq s\leq T$ and $\{A_{ij}\}_{j=1}^{N}\subset\mathscr{F}_{t}$
is a partition of $(\Omega,\mathscr{F}_{t})$, i.e., $\cup_{i=1}^{N}A_{i}=\Omega$
and $A_{i}\cap A_{j}=\emptyset$, $i\neq j$. Like \eqref{eq:prop3.4-2}
and \eqref{eq:prop3.4-3}, we know that
\[
J(\gamma_{t},u_{i}(\cdot))=\sum_{j=1}^{N}\chi_{A_{ij}}J(\gamma_{t},u_{ij}(\cdot)).
\]
It is easy to prove that $J(\gamma_{t},u_{ij}(\cdot))$ is deterministic.
Without lost of generality, we suppose
\[
J(\gamma_{t},u_{ij}(\cdot))\leq J(\gamma_{t},u_{i1}(\cdot)).
\]
Immediately, we have $J(\gamma_{t},u_{i}(\cdot))\leq J(\gamma_{t},u_{i1}(\cdot))$.
Combining \eqref{eq:prop3.4-0}, we have
\[
\lim_{i\to\infty}J(\gamma_{t},u_{i1}(\cdot))=\tilde{v}(\gamma_{t}).
\]
Therefore, $\tilde{v}(\gamma_{t})$ is deterministic. \end{proof}

From \eqref{eq:regularity of Y 1} and \eqref{eq:regularity of Y 2},
we have the following estimates on functional $\tilde{v}$.

\begin{lem} \label{lem:regularity of valued functional 1}There exists
a constant $C>0$ such that, for all $0\leq t\leq T$, $\gamma_{t},\gamma_{t}^{'}\in\Lambda_{t}$,
\begin{eqnarray*}
(1) & |\tilde{v}(\gamma_{t})-\tilde{v}(\gamma_{t}^{'})| & \leq C\|\gamma_{t}-\gamma_{t}^{'}\|_{0};\\
(2) & |\tilde{v}(\gamma_{t})| & \leq C(1+\|\gamma_{t}\|_{0}).
\end{eqnarray*}
\end{lem}

To formulate the DPP for the optimal control problem \eqref{eq:diffu proc-2}, \eqref{eq:BSDE-1} and \eqref{eq:defination Of value fun}, we define
the family of backward semi-groups generated by BSDE \eqref{eq:BSDE-1}
in the spirit of of Peng \cite{peng1997bsde}.

Given the initial path $\gamma_{t}\in\Lambda$, an $\mathscr{F}$-stopping
time $\hat{\tau}\geq t$, an admissible control process $u\in\mathcal{U}$,
and a real-valued random variable $\eta\in L^{2}(\Omega,\mathscr{F}_{\hat{\tau}},P;\mathbb{R})$,
we put
\[
\mathbb{G}_{s,\hat{\tau}}^{\gamma_{t},u}[\eta]:=\tilde{Y}^{\gamma_{t},u}(s),\quad s\in[t,\hat{\tau}],
\]
where the pair $(\tilde{Y}^{\gamma_{t},u},\tilde{Z}^{\gamma_{t},u})$
solves the following BSDE of the terminal time $\hat{\tau}$:
\begin{eqnarray}
\tilde{Y}^{\gamma_{t},u}(s) & = & \eta+\int_{s}^{\hat{\tau}}\!\! f(X_{r}^{\gamma_{t},u},\tilde{Y}^{\gamma_{t},u}(r),\tilde{Z}^{\gamma_{t},u}(r),u(r))\, dr\nonumber \\
 &  & -\int_{s}^{\hat{\tau}}\tilde{Z}^{\gamma_{t},u}(r)\, dW(r),\quad\text{a.s.-}\omega,\text{ all }s\in[t,\hat{\tau}],\label{eq:BSDE-1-1}
\end{eqnarray}
with $X^{\gamma_{t},u}$ being the solution to SDE \eqref{eq:diffu proc-2}.
Then, obviously, for the solution $(Y^{\gamma_{t},u},Z^{\gamma_{t},u})$
of BSDE \eqref{eq:BSDE-1}, the uniqueness of the BSDE yields
\begin{align*}
J(\gamma_{t},u) & =\mathbb{G}_{t,T}^{\gamma_{t},u}\left[g(X_{T}^{\gamma_{t},u})\right]=\mathbb{G}_{t,\hat{\tau}}^{\gamma_{t},u}\left[Y^{\gamma_{t},u}(\hat{\tau})\right]\\
 & =\mathbb{G}_{t,\hat{\tau}}^{\gamma_{t},u}\left[Y^{X_{\hat{\tau}}^{\gamma_{t},u},u}(\hat{\tau})\right]=\mathbb{G}_{t,\hat{\tau}}^{\gamma_{t},u}\left[J(X_{\hat{\tau}}^{\gamma_{t},u},u)\right].
\end{align*}

The following dynamic programming principle (DPP) is adapted from
the Markovian case, by mimicking the method of Peng \cite{peng1992stochastic,peng1997bsde}.

\begin{thm} \label{thm:DPP} Let Assumption (H1) be satisfied. Then
for any $\delta\in(0,T-t)$, the value functional $\tilde{v}$ obeys
the following:
\begin{equation}
\tilde{v}(\gamma_{t})=\esssup_{u\in\mathcal{U}}\mathbb{G}_{t,t+\delta}^{\gamma_{t},u}\left[\tilde{v}(X_{t+\delta}^{\gamma_{t},u})\right],\quad\gamma_{t}\in\Lambda.\label{eq:DPP}
\end{equation}
\end{thm}

Our proof requires the following lemma

\begin{lem} Let $t\in[0,T]$, $\Gamma_{t}\in L^{2}(\Omega,\mathscr{F}_{t};\Lambda_{t},\mathscr{B}(\Lambda_{t}))$.
For $u\in\mathcal{U}$, we have
\begin{equation}
\tilde{v}(\Gamma_{t})\geq Y^{\Gamma_{t},u}(t).\label{eq:lemma3.7-1}
\end{equation}
For any $\varepsilon>0$, there is $u\in\mathcal{U}$ such that
\begin{equation}
\tilde{v}(\Gamma_{t})\leq Y^{\Gamma_{t},u}(t)+\varepsilon.\label{eq:lemma3.7-2}
\end{equation}
\end{lem}

\begin{proof} Since $\tilde{v}$ is continuous in $\gamma_{t}\in\Lambda_{t}$
and $Y_{t}^{\gamma_{t},u}$ is continuous in $(\gamma_{t},u)\in\Lambda_{t}\times\mathcal{U}$,
it is sufficient to prove \eqref{eq:lemma3.7-1} for the following
class of $\Gamma_{t}$ and $u$:
\[
\Gamma_{t}=\sum_{i=1}^{N}\chi_{A_{i}}\gamma_{t}^{i}
\]
and
\[
u(s)=\sum_{i=1}^{N}\chi_{A_{i}}u^{i}(s),\quad t\leq s\leq T.
\]
Here, $N$ is a positive integer, $u^{i}$ is $\mathscr{F}^{t}$-adapted,
and $\{A_{i}\}_{i=1}^{N}\subset\mathscr{F}_{t}$ is a partition of
$(\Omega,\mathscr{F}_{t})$, that is, $\cup_{i=1}^{N}A_{i}=\Omega$
and $A_{i}\cap A_{j}=\emptyset$ for $i\neq j$. We have
\[
Y^{\Gamma_{t},u}(t)=\sum_{i=1}^{N}\chi_{A_{i}}Y^{\gamma_{t}^{i},u^{i}}(t)\leq\sum_{i=1}^{N}\chi_{A_{i}}\tilde{v}(\gamma_{t}^{i})=\tilde{v}(\Gamma_{t}).
\]
We obtain the first assertion \eqref{eq:lemma3.7-1}.

In a similar way, we prove \eqref{eq:lemma3.7-2}. Obviously there
exists $\bar{\Gamma}_{t}\in L^{2}(\Omega,\mathscr{F}_{t};\Lambda_{t},\mathscr{B}(\Lambda_{t}))$
of the form
\[
\bar{\Gamma}_{t}=\sum_{i=1}^{\infty}\chi_{A_{i}}\gamma_{t}^{i},
\]
such that
\[
\|\Gamma_{t}-\bar{\Gamma}_{t}\|\leq\frac{1}{3}C^{-1}\varepsilon,
\]
where $C$ is the constant in Lemmas \ref{lem:FSDE} and \ref{lem:FBSDE},
$\gamma_{t}^{i}\in\Lambda_{t}$ and $\{A_{i}\}_{i=1}^{\infty}\subset\mathscr{F}_{t}$
satisfies $\cup_{i=1}^{\infty}A_{i}=\Omega$ and $A_{i}\cap A_{j}=\emptyset$,
$i\neq j$. By Lemmas \ref{lem:FSDE} and \ref{lem:FBSDE}, we have
for any $u\in\mathcal{U}$, $P$-a.e.,
\begin{align}
|Y^{\Gamma_{t},u}(t)-Y^{\bar{\Gamma}_{t},u}(t)| & \leq\frac{1}{3}\varepsilon,\label{eq:lemma3.7-3}\\
|\tilde{v}(\Gamma_{t})-\tilde{v}(\bar{\Gamma}_{t})| & \leq\frac{1}{3}\varepsilon.\nonumber
\end{align}
Then for any $\gamma_{t}^{i}$, we can choose an $\mathscr{F}^{t}$-adapted
admissible control $u^{i}\in\mathcal{U}$ such that
\[
\tilde{v}(\gamma_{t}^{i})\leq Y^{\gamma_{t}^{i},u^{i}}(t)-\frac{1}{3}\varepsilon.
\]
Define
\[
u(\cdot)=\sum_{i=1}^{\infty}\chi_{A_{i}}u^{i}(\cdot).
\]
Combining \eqref{eq:lemma3.7-3}, we have
\begin{align*}
Y^{\Gamma_{t},u}(t) & \geq-|Y^{\Gamma_{t},u}(t)-Y^{\bar{\Gamma}_{t},u}(t)|+Y^{\bar{\Gamma}_{t},u}(t)\geq-\frac{1}{3}\varepsilon+\sum_{i=1}^{\infty}\chi_{A_{i}}Y^{\gamma_{t}^{i},u^{i}}(t)\\
 & \geq-\frac{1}{3}\varepsilon+\sum_{i=1}^{\infty}\chi_{A_{i}}(\tilde{v}(\gamma_{i}^{i})-\frac{1}{3}\varepsilon)=-\frac{2}{3}\varepsilon+\tilde{v}(\bar{\Gamma}_{t})\\
 & \geq-\varepsilon+\tilde{v}(\Gamma_{t}).
\end{align*}
The proof is complete. \end{proof}

\begin{proof}[Proof of Theorem \ref{thm:DPP}] On one hand,
\begin{align*}
\tilde{v}(\gamma_{t}) & =\sup_{u\in\mathcal{U}}\mathbb{G}_{t,T}^{\gamma_{t},u}\left[g(X_{T}^{\gamma_{t},u})\right]=\sup_{u\in\mathcal{U}}\mathbb{G}_{t,t+\delta}^{\gamma_{t},u}\left[Y^{X_{t+\delta}^{\gamma_{t},u},u}(t+\delta)\right].
\end{align*}
From \eqref{eq:lemma3.7-1} and the definition of $\tilde{v}$, we
have
\begin{align*}
\tilde{v}(\gamma_{t}) & \leq\sup_{u\in\mathcal{U}}\mathbb{G}_{t,t+\delta}^{\gamma_{t},u}\left[\tilde{v}(X_{t+\delta}^{\gamma_{t},u})\right].
\end{align*}

On the other hand, by \eqref{eq:lemma3.7-2} for any $\varepsilon>0$,
there exists $\bar{u}\in\mathcal{U}$ such that, a.e. $P$
\[
\tilde{v}(X_{t+\delta}^{\gamma_{t},u})\leq Y^{X_{t+\delta}^{\gamma_{t},u},\bar{u}}(t+\delta)+\varepsilon.
\]
From \ref{lem:comparison-theorem}, we have
\begin{align*}
\tilde{v}(\gamma_{t}) & \geq\sup_{u\in\mathcal{U}}\mathbb{G}_{t,t+\delta}^{\gamma_{t},u}\left[\tilde{v}(X_{t+\delta}^{\gamma_{t},u})-\varepsilon\right]\\
 & \geq\sup_{u\in\mathcal{U}}\mathbb{G}_{t,t+\delta}^{\gamma_{t},u}\left[\tilde{v}(X_{t+\delta}^{\gamma_{t},u})\right]-C\varepsilon.
\end{align*}
Since $\varepsilon$ is arbitrary, we have \eqref{eq:DPP}. \end{proof}

In Lemma \ref{lem:regularity of valued functional 1}, the value functional
$\tilde{v}$ is Lipschitz continuous in $\Lambda_{t}$, uniformly
in $t$. Theorem \ref{thm:DPP} implies the following continuity in
$t$.

\begin{lem} \label{lem:t-regu-value fun} Let Assumption (H1) be
satisfied. There is a constant $C$ such that for every $\gamma_{T}\in\Lambda_{T}$,
$t,t'\in[0,T]$,
\begin{equation}
|\tilde{v}(\gamma_{t})-\tilde{v}(\gamma_{t'})|\leq C(1+\|\gamma_{t\vee t'}\|_{0})(|t-t'|)^{\frac{1}{2}}.\label{eq:time-estimat}
\end{equation}
\end{lem}

\begin{proof} Suppose that $t\leq t'$. From Theorem \ref{thm:DPP},
we see that for any $\varepsilon>0$, there is $u\in\mathcal{U}$
such that
\[
|\tilde{v}(\gamma_{t})-\mathbb{G}_{t,t'}^{\gamma_{t},u}\left[\tilde{v}(X_{t'}^{\gamma_{t},u})\right]|\leq\varepsilon.
\]
Hence
\[
|\tilde{v}(\gamma_{t'})-\tilde{v}(\gamma_{t})|\leq\text{Part1}+\text{Part2}+\varepsilon,
\]
where
\begin{align*}
\text{Part1} & :=E\left|\mathbb{G}_{t,t'}^{\gamma_{t},u}\left[\tilde{v}(X_{t'}^{\gamma_{t},u})\right]-\mathbb{G}_{t,t'}^{\gamma_{t},u}\left[\tilde{v}(\gamma_{t'})\right]\right|,\\
\text{Part2} & :=E\left|\mathbb{G}_{t,t'}^{\gamma_{t},u}\left[\tilde{v}(\gamma_{t'})\right]-\tilde{v}(\gamma_{t'})\right|.
\end{align*}
Since $\tilde{v}$ is uniformly continuous in $\gamma_{t}$, we have
from Lemmas \ref{lem:FSDE} and \ref{lem:FBSDE},
\[
\text{Part1}\leq C(E\big|\tilde{v}(X_{t'}^{\gamma_{t},u})-\tilde{v}(\gamma_{t'})\big|^{2})^{\frac{1}{2}}\leq C(E\|X_{t'}^{\gamma_{t},u}-\gamma_{t'}\|_{0}^{2})^{\frac{1}{2}}\leq C(1+\|\gamma_{t'}\|_{0})(t'-t)^{\frac{1}{2}}
\]
for a positive constant $C$ being independent of $u$. By the definition
of $\mathbb{G}_{t,t'}^{\gamma_{t},u}$, we have
\begin{align*}
\text{Part2}= & \left|E\left[\tilde{v}(\gamma_{t'})+\int_{t}^{t'}\!\! f(X_{r}^{\gamma_{t},u},Y^{\gamma_{t},u}(r),Z^{\gamma_{t},u}(r),u(r))\, dr\right.\right.\\
 & \quad\quad\left.\left.+\int_{t}^{t'}Z^{\gamma_{t},u}(r)\, dW(r)\right]-\tilde{v}(\gamma_{t'})\right|\\
= & E\left[\int_{t}^{t'}\big|f(X_{r}^{\gamma_{t},u},Y^{\gamma_{t},u}(r),Z^{\gamma_{t},u}(r))\big|dr\right]\\
 & \leq C(1+\|\gamma_{t'}\|_{0})(t'-t)^{\frac{1}{2}}.
\end{align*}
Since $\varepsilon$ is arbitrary, we have \eqref{eq:time-estimat}.
\end{proof}

From Lemmas \ref{lem:regularity of valued functional 1} and \ref{lem:t-regu-value fun},
we have the regularity for the value functional $\tilde{v}$.

\begin{thm}\label{thm:regu-of-tilde-u} Let Assumption (H1) be satisfied.
There is a constant $C>0$ such that for any $0\leq t\leq t'\leq T$
and $\gamma_{t},\gamma'_{t'}\in\Lambda$, we have
\begin{equation}
|\tilde{v}(\gamma_{t})-\tilde{v}(\gamma'_{t'})|\leq C\left(\|\gamma_{t,t'}-\gamma'_{t'}\|_{0}+(1+\|\gamma'_{t'}\|_{0}+\|\gamma_{t}\|_{0})(t'-t)^{\frac{1}{2}}\right).\label{eq:regul-tilde-u}
\end{equation}
\end{thm}

From \eqref{eq:regul-tilde-u}, we have the stronger version of Theorem
\ref{thm:DPP}:

\begin{thm} \label{thm:DPP-1} Let Assumption (H1) be satisfied.
For any $\mathscr{F}$-stopping time $\hat{\tau}\geq t$, a.s., the
value functional $\tilde{v}$ obeys the following:
\begin{equation}
\tilde{v}(\gamma_{t})=\esssup_{u\in\mathcal{U}}\mathbb{G}_{t,\hat{\tau}}^{\gamma_{t},u}\left[\tilde{v}(X_{\hat{\tau}}^{\gamma_{t},u})\right],\quad\gamma_{t}\in\Lambda.\label{eq:DPP-1}
\end{equation}
\end{thm}

\section{\label{sec:FBSDE-and-viscosity} Associated path-dependent Bellman
equation}

\subsection{Path-dependent Bellman equation and viscosity solution \label{sub:definition of viscosity solutions}}

Define the Hamiltonian $\mathcal{H}:\Lambda\times\mathbb{R}\times\mathbb{R}^{n}\times\mathbb{R}^{n\times n}\times U\to\mathbb{R}$
by
\begin{equation}
\mathcal{H}(\gamma_{t},r,p,A,u):=\frac{1}{2}\text{Tr}\big(\sigma\sigma^{T}(\gamma_{t},u)A\big)+\langle b(\gamma_{t},u),p\rangle+f(\gamma_{t},r,\sigma^{T}(\gamma_{t},u)p,u)\quad
\end{equation}
for $(\gamma_{t},r,p,A,u)\in\Lambda\times\mathbb{R}\times\mathbb{R}^{n}\times\mathbb{R}^{n\times n}\times U$.
For each $(\gamma_{t},u)\in\Lambda\times U$, define the differential
operator $\mathscr{L}(\gamma_{t},u):\mathscr{C}^{1,2}(\Lambda)\to\mathbb{R}$
by
\begin{equation}
(\mathscr{L}\psi)(\gamma_{t},u):=D_{t}\psi(\gamma_{t})+\mathcal{H}(\gamma_{t},\psi(\gamma_{t}),D_{x}\psi(\gamma_{t}),D_{xx}\psi(\gamma_{t}),u),\quad\psi\in \mathscr{C}^{1,2}(\Lambda).\label{eq:L-operator}
\end{equation}

Consider the following path-dependent Bellman equation:
\begin{equation}
-D_{t}v(\gamma_{t})-\sup_{u\in U}\mathcal{H}(\gamma_{t},v(\gamma_{t}),D_{x}v(\gamma_{t}),D_{xx}v(\gamma_{t}),u)=0,\qquad\gamma_{t}\in\Lambda,t<T,\label{eq:PHJBE}
\end{equation}
with the terminal condition
\begin{equation}
v(\gamma_{T})=g(\gamma_{T}),\qquad\gamma_{T}\in\Lambda_{T}.\label{eq:terminal condition}
\end{equation}

\begin{defn} (Classical solution). A functional $v\in\mathscr{C}^{1,2}(\Lambda)$
is called a classical solution to the path-dependent Bellman equation
\eqref{eq:PHJBE} if it satisfies the path-dependent Bellman equation
\eqref{eq:PHJBE} point-wisely in the sense of Definition \ref{derivative}.
\end{defn}


For $(\kappa, \iota)\in (0, \infty)\times (0,T)$ 
and $\gamma_{t}\in\Lambda$ with $t\in[0,T-\iota)$, define the cylinder
\begin{equation}
\mathbf{Q}_{\kappa,\iota}(\gamma_{t}):=\{\gamma_{t'}'\in\Lambda:t\leq t'\leq t+\iota,\|\gamma_{t,t'}-\gamma_{t'}'\|_{0}<\kappa\}.\label{eq:cylinder}
\end{equation}
Write $\mathbf{Q}_{\kappa,\iota}$ for $\mathbf{Q}_{\kappa,\iota}(\mathbf{0})$,
i.e.,
\begin{equation}
\mathbf{Q}_{\kappa,\iota}:=\{\gamma_{t'}'\in\Lambda:0\leq t'\leq t+\iota,\|\gamma_{t'}'\|_{0}<\kappa\}.\label{eq:glob-cylinder}
\end{equation}
Throughout the rest of this paper, we fix $\alpha\in(0,\frac{1}{2})$
and $\beta\in(0,1)$.  

We now generalize the classical notions of semi-jets (see Crandall,
Ishii, and Lions \cite{crandall1992user}). 
For $\gamma_{t}\in\mathbf{C}^{\alpha}$ with $t\in[0,T)$, $(\mu,\kappa)\in (\|\gamma_t\|_0, \infty)\times (0,T-t)$, and
$v\in\mathscr{C}(\Lambda)$, define the super-jet of $v$ at $\gamma_{t}$ sliced by the double index of H\"older modulus $(\mu, \kappa)$:
\begin{flalign}
 & \mathcal{J}_{\mu,\kappa}^{+}(\gamma_{t},v)\label{eq:super-jet}\\
:= & \Big\{\psi\in\mathscr{C}^{1,2}(\Lambda):\,|\psi|_{2,\beta;\mathbf{Q}_{\kappa,\kappa}(\gamma_{t})\cap\mathbf{C}_{\mu}^{\alpha}}\leq\kappa^{-1},\,\text{and }\nonumber \\
 & \,0=\psi(\gamma_{t})-v(\gamma_{t})\leq\psi(\gamma'_{t'})-v(\gamma'_{t'}),\,\forall\gamma'_{t'}\in\mathbf{Q}_{\kappa,\kappa}(\gamma_{t})\cap\mathbf{C}_{\mu}^{\alpha}\Big\}\nonumber
\end{flalign}
and the sub-jet of $v$ at $\gamma_{t}$ sliced by the double index of H\"older modulus $(\mu, \kappa)$:
\begin{align}
 & \mathcal{J}_{\mu,\kappa}^{-}(\gamma_{t},v)\label{eq:sub-jet}\\
:= & \Big\{\psi\in \mathscr{C}^{1,2}(\Lambda):\,|\psi|_{2,\beta;\mathbf{Q}_{\kappa,\kappa}(\gamma_{t})\cap\mathbf{C}_{\mu}^{\alpha}}\leq\kappa^{-1},\,\text{ and }\nonumber \\
 & 0=\psi(\gamma_{t})-v(\gamma_{t})\geq\psi(\gamma'_{t'})-v(\gamma'_{t'}),\,\forall\gamma'_{t'}\in\mathbf{Q}_{\kappa,\kappa}(\gamma_{t})\cap\mathbf{C}_{\mu}^{\alpha}\Big\}.\nonumber
\end{align}

\begin{rem} Both $\mathcal{J}_{\mu,\kappa}^{+}(\gamma_{t},v)$ and
$\mathcal{J}_{\mu,\kappa}^{-}(\gamma_{t},v)$ may be empty.\end{rem}

Our notion of viscosity solutions is defined as follows.

\begin{defn}\label{defn:viscosity solution}

(i) We call $v\in\mathscr{C}(\Lambda)$ a viscosity sub-solution to
the path-dependent Bellman equation \eqref{eq:PHJBE}, if for any
$(M_{0}, \kappa)\in (0, \infty)\times (0,T)$, we have
\begin{align}
\varlimsup_{\mu\to\infty}\sup_{\begin{subarray}{c}
\psi\in\mathcal{J}_{\mu,\kappa}^{+}(\gamma_{t},v)\\
\gamma_{t}\in\mathbf{Q}_{M_{0},T-\kappa}\cap\,\mathbf{C}_{\mu}^{\alpha}
\end{subarray}}\Big\{-D_{t}\psi-\sup_{u\in U}\mathcal{H}(\cdot,\psi,D_{x}\psi,D_{xx}\psi,u)\Big\}(\gamma_{t})\leq0 & .\label{eq:subsolution}
\end{align}

(ii) We call $v\in\mathscr{C}(\Lambda)$ a viscosity super-solution
to the path-dependent Bellman equation \eqref{eq:PHJBE}, if for any
$(M_{0}, \kappa)\in (0, \infty)\times (0,T)$, we have
\begin{equation}
\varliminf_{\mu\to\infty}\inf_{\begin{subarray}{c}
\psi\in\mathcal{J}_{\mu,\kappa}^{-}(\gamma_{t},v)\\
\gamma_{t}\in\mathbf{Q}_{M_{0},T-\kappa}\cap\,\mathbf{C}_{\mu}^{\alpha}
\end{subarray}}\Big\{-D_{t}\psi-\sup_{u\in U}\mathcal{H}(\cdot,\psi,D_{x}\psi,D_{xx}\psi,u)\Big\}(\gamma_{t})\geq0.\label{eq:suppersolution}
\end{equation}

(iii) We call $v\in\mathscr{C}(\Lambda)$ a viscosity solution to
the path-dependent Bellman equation \eqref{eq:PHJBE} if it is both
the viscosity sub- and super-solution.

Note that in \eqref{eq:subsolution} and \eqref{eq:suppersolution}:
$\sup\emptyset:=-\infty$, $\inf\emptyset:=+\infty$. \end{defn}

\begin{rem}\label{rem:def-viscosity-solution} (1) A viscosity solution
of the path-dependent Bellman equation $u$ is a classical solution
if it furthermore lies in $\mathscr{C}^{1,2}(\Lambda)$.

(2) In the classical uniqueness proof of viscosity solution to state-dependent
PDEs in an unbounded domain (which is locally compact), a conventional
technique is to construct an auxiliary smooth function decaying outside
a compact domain. In our path-dependent case, we find it difficult
to construct such a smooth functionals. For the sake of the uniqueness
proof, our new notion of jets is enlarged to be defined only on $\mathbf{C}_{\mu}^{\alpha}$,
which is compact in $\Lambda$. However, at a cost, our modification
leads to additional difficulty in the existence proof.

(3) Assume that all the coefficients of Bellman equation \eqref{eq:PHJBE}
and terminal condition \eqref{eq:terminal condition} are state-dependent.
Let state-dependent function $u$ be a viscosity solution to \eqref{eq:PHJBE}
as a path-dependent functional. Then $u$ is also a classical viscosity
solution as a function of time and state. \end{rem}


\subsection{Main results}

Our main results on the existence and the uniqueness of the viscosity
solution to the path-dependent Bellman equation\eqref{eq:PHJBE} are
formulated below.

\begin{thm} \label{thm:ex-th}Let assumption (H1) be satisfied. Then
$\tilde{v}$ defined by \eqref{eq:defination Of value fun} is a viscosity
solution to the path-dependent Bellman equation \eqref{eq:PHJBE}.\end{thm}

The uniqueness is given on both non-degenerate and degenerate cases.
we first address the non-degenerate case.

We make the following assumption, extending our previous assumption
(H1) to the larger path space $\hat{\Lambda}$.

\medskip{}

(H2) \textit{ The functionals $b:\hat{\Lambda}\times U\rightarrow\mathbb{R}^{n}$,
$\sigma:\hat{\Lambda}\times U\rightarrow\mathbb{R}^{n\times d}$,
$g:\hat{\Lambda}_{T}\rightarrow\mathbb{R},$ and $f:\hat{\Lambda}\times\mathbb{R}\times\mathbb{R}^{d}\times U\rightarrow\mathbb{R}$
are all bounded. There is a constant $C>0$ such that for all $(t,\gamma_{T},y,z,u),\,(t',\gamma'_{T},y',z',u')\in[0,T]\times\hat{\Lambda}_{T}\times\mathbb{R}\times\mathbb{R}^{d}\times U$,
\begin{align}
|b(\gamma_{t},u)-b(\gamma'_{t'},u')| & \leq C(d_{p}(\gamma_{t},\gamma'_{t'})+|u-u'|),\nonumber \\
|\sigma(\gamma_{t},u)-\sigma(\gamma'_{t'},u')| & \leq C(d_{p}(\gamma_{t},\gamma'_{t'})+|u-u'|),\nonumber \\
|f(\gamma_{t},y,z,u)-f(\gamma'_{t'},y',z',u')| & \leq C(d_{p}(\gamma_{t},\gamma'_{t'})+|y-y'|+|z-z'|+|u-u'|),\label{eq:Lip-conditon}\\
|g(\gamma_{T})-g(\gamma'_{T})| & \leq C\|\gamma_{T}-\gamma'_{T}\|_{0}.\nonumber
\end{align}
Moreover, $\sigma$ satisfies the non-degenerate condition
\begin{eqnarray}
\sigma\sigma^{T} & > & \frac{1}{C}I_{n}.\label{eq:non-degen}
\end{eqnarray}
}

\medskip{}

The uniqueness of viscosity solutions of \eqref{eq:PHJBE} is an immediate
consequence of the following representation theorem.

\begin{thm}\label{thm:(repres-theorem-nondegen} Suppose that (H2)
holds. Let $v\in\mathscr{C}_{b}(\Lambda)\cap\mathscr{C}_{u}(\Lambda)$.
If $v$ is a viscosity solution to the path-dependent Bellman equation
\eqref{eq:PHJBE}, and $v=g$ on $\Lambda_{T}$, then $v$ is the
value functional $\tilde{v}$ defined by \eqref{eq:defination Of value fun}.
\end{thm}

In the degenerate case of $\sigma\sigma^{T}\geq0$, we have the following
extra smooth conditions on the coefficients.

\medskip{}

(H3) \textit{Functionals $b:\hat{\Lambda}\times U\rightarrow\mathbb{R}^{n}$,
$\sigma:\hat{\Lambda}\times U\rightarrow\mathbb{R}^{n\times d}$,
$g:\hat{\Lambda}_{T}\rightarrow\mathbb{R}$ and $f:\hat{\Lambda}\times\mathbb{R}\times\mathbb{R}^{d}\times U\rightarrow\mathbb{R}$
satisfy \eqref{eq:Lip-conditon}. Furthermore, for any $u\in U$,
the functionals $\sigma(\cdot,u),\, b(\cdot,u)\in\mathscr{C}_{b}^{1,2}(\hat{\Lambda}),$
$g\in\mathscr{C}_{b}^{1,2}(\hat{\Lambda}_{T})$, and $f(\cdot,\cdot,\cdot,u)\in\mathscr{C}_{b}^{1,2,2,2}(\hat{\Lambda}\times\mathbb{R}\times\mathbb{R}^{n})$
and all their differentials are bounded uniformly w.r.t. $u\in U$.}

\medskip{}

We have the following representation theorem.

\begin{thm} \label{thm:Repr-Th-degen} Suppose that (H3) holds. Let
$v\in\mathscr{C}_{b}(\Lambda)\cap\mathscr{C}_{u}(\Lambda)$ be a viscosity
solution to the path-dependent Bellman equation \eqref{eq:PHJBE},
and $v=g$ on $\Lambda_{T}$. Then $v=\tilde{v}$, where $\tilde{v}$
is defined by \eqref{eq:defination Of value fun}. \end{thm}

\begin{rem}\label{rem:bound-regul-tilde-u} Analogous to the proof
of Lemma \ref{lem:regularity of valued functional 1} and Theorem
\ref{thm:regu-of-tilde-u}, from the bounded and Lipschitz assumption
on the coefficients in (H2) or (H3), the value functional $\tilde{v}$
can be shown to be bounded and to satisfy
\begin{equation}
|v(\gamma_{t})-v(\gamma'_{t'})|\leq Cd_{p}(\gamma_{t},\gamma_{t'}'),\quad\forall\gamma_{t},\gamma_{t'}'\in\Lambda,\label{eq:regu-of-u}
\end{equation}
which implies $\tilde{v}\in\mathscr{C}_{b}(\Lambda)\cap\mathscr{C}_{u}(\Lambda)$.
Note that the non-degeneracy condition \eqref{eq:non-degen} is not
needed here. \end{rem}

In view of the comparison theorem of BSDEs, we immediately have the
following comparison theorem.

\begin{cor} Suppose that either (H2) or (H3) holds. Let $v_{1},v_{2}\in\mathscr{C}_{b}(\Lambda)$
satisfy \eqref{eq:regu-of-u}, and $g_{1},g_{2}\in\mathscr{C}_{b}(\Lambda_{T})$
satisfy $g_{1}\leq g_{2}$. Furthermore, let $v_{1}$ and $v_{2}$
be viscosity solutions to the path-dependent Bellman equation \eqref{eq:PHJBE}
with the terminal conditions:
\[
v_{1}(\gamma_{T})=g_{1}(\gamma_{T}),\quad v_{2}(\gamma_{T})=g_{2}(\gamma_{T}),\quad\forall\gamma_{T}\in\Lambda_{T}.
\]
Then $v_{1}\leq v_{2}$.\end{cor}\bigskip{}

For an initial path $\gamma_{t}\in\Lambda$, define
\[
\mathcal{U}_{t}:=\left\{ u\in[t,T]\times\Omega\to U\big|u\text{ is }\mathcal{F}^{t}\text{-progressive measurable}\right\} .
\]
We have

\begin{thm}[Verification Theorem]

Let $v$ be a classical solution to \eqref{eq:PHJBE} and \eqref{eq:terminal condition}.
Then we have the following two assertions:

(i) $v(\gamma_{t})\geq J(\gamma_{t},u),\quad\forall(u,\gamma_{t})\in\mathcal{U}_{t}\times\Lambda$.

(ii) If the following holds for an admissible control $u^{*}\in\mathcal{U}_{t}$:
for every $\gamma_{t}\in\Lambda$,

\[
0=(\mathscr{L}v)\left(X_{s}^{\gamma_{t},u^{*}},u^{*}(s)\right)=\argmax_{\beta\in U}\left\{ (\mathscr{L}v)\left(X_{s}^{\gamma_{t},u^{*}},\beta\right)\right\} ,\quad\text{a.s.-}\omega,\text{ a.e.}s\in[t,T],
\]
then $v(\gamma_{t})=J(\gamma_{t})$ for any $\gamma_{t}\in\Lambda$.
\end{thm}

\begin{proof} For $u\in\mathcal{U}_{t}$, since $v$ is a classical
solution, we have
\[
(\mathscr{L}u)\left(X_{s}^{\gamma_{t},u},u(s)\right)\leq0.
\]
Applying It\^o formula to compute $v(X_{s}^{\gamma_{t},u})$, we
have
\begin{align*}
 & v(X_{s}^{\gamma_{t},u})=g(X_{T}^{\gamma_{t},u})\\
 & +\int_{s}^{T}\left(f(\cdot,v,\sigma^{T}(\cdot,u(r))D_{x}v,u(r))-(\mathscr{L}v)(\cdot,u(r))\right)(X_{r}^{\gamma_{t},u})dr\\
 & -\int_{s}^{T}\sigma^{T}(X_{r}^{\gamma_{t},u},u(r))D_{x}v(X_{r}^{\gamma_{t},u})dW(r),\quad s\in[t,T].
\end{align*}
Define for $s\in[t,T]$,
\[
Y^{1}(s):=v(X_{s}^{\gamma_{t},u})-Y^{\gamma_{t},u}(s),\quad Z^{1}(s):=\sigma^{T}(X_{s}^{\gamma_{t},u},u(s))D_{x}v(X_{s}^{\gamma_{t},u})-Z^{\gamma_{t},u}(s).
\]
In view of \eqref{eq:BSDE-1}, we have
\begin{align*}
 & Y^{1}(s)\\
= & \int_{s}^{T}\left(f(\cdot,v,\sigma^{T}D_{x}v,u(r))-f(\cdot,Y^{\gamma_{t},u}(r),Z^{\gamma_{t},u}(r),u(r))\right)(X_{r}^{\gamma_{t},u})\\
 & -(\mathscr{L}v)(X_{r}^{\gamma_{t},u},u(r))dr-Z^{1}(r)dW(r)\\
= & \int_{s}^{T}\left(A(r)Y^{1}(r)+\langle\bar{A},Z^{1}\rangle(r)-(\mathscr{L}v)(X_{r}^{\gamma_{t},u},u(r))\right)dr-Z^{1}(r)dW(r).
\end{align*}
Denote by $\Gamma^{t}(\cdot)$ the unique solution of the linear SDE
\[
d\Gamma^{t}(s)=\Gamma^{t}(s)\Big(A(s)ds+\bar{A}(s)dW(s)\Big),\quad s\in[t,T];\quad\Gamma^{t}(t)=1.
\]
From \cite[Proposition 2.2]{el1997backward}, we have
\[
E\,[Y^{1}(t)]=-E\,\left[\int_{t}^{T}\Gamma^{t}(r)(\mathscr{L}v)(X_{r}^{\gamma_{t},u},u(r))dr\right]\geq0,
\]
and the equality holds for $u=u^{*}$. This proves Assertions (i)
and (ii). The proof is complete.\end{proof}

\begin{example} \label{relation-PHJB-SHJB} In what follows, we show
that the conventional non-Markovian optimal stochastic control problem
is included as a particular case of our problem \eqref{eq:diffu proc-2},
\eqref{eq:BSDE-1} and \eqref{eq:defination Of value fun}. Under
some suitable smooth conditions, the corresponding path-dependent
Bellman equation is associated to a backward stochastic Bellman equation
via Dupire's functional calculus.

Let $\{B_{t},0\leq t\leq T\}$ be a $d$-dimensional Winner process
on the probability space $(\Omega:=\Lambda_{T}(\mathbb{R}^{d}),P_{0})$.
Consider functionals $\bar{b}:\Lambda(\mathbb{R}^{d})\times\mathbb{R}^{n}\times U\rightarrow\mathbb{R}^{n}$,
$\bar{\sigma}:\Lambda(\mathbb{R}^{d})\times\mathbb{R}^{n}\times U\rightarrow\mathbb{R}^{n\times d}$,
$\bar{f}:\Lambda(\mathbb{R}^{d})\times\mathbb{R}^{n}\times\mathbb{R}\times\mathbb{R}^{d}\times U\rightarrow\mathbb{R}$
and $\bar{g}:\Lambda_{T}(\mathbb{R}^{d})\times\mathbb{R}^{n}\rightarrow\mathbb{R}$.
The non-Markovian stochastic optimal control problem is formulated
as follows. For any $u\in\mathcal{U}$ and $(t,x)\in[0,T]\times\mathbb{R}^{n}$,
consider the following forward and backward stochastic differential
systems:
\[
\left\{ \begin{array}{rcl}
d\bar{X}^{t,x,u}(s) & = & \bar{b}(B_{s},\bar{X}^{t,x,u}(s),u(s))ds+\bar{\sigma}(B_{s},\bar{X}^{t,x,u}(s),u(s))dB(s),\quad s\in[t,T];\\
\bar{X}^{t,x,u}(t) & = & x
\end{array}\right.
\]
and
\[
\left\{ \begin{array}{rcl}
-d\bar{Y}^{t,x,u}(s) & = & \bar{f}(B_{s},x,\bar{Y}^{t,x,u}(s),\bar{Z}^{t,x,u}(s),u(s))ds-\bar{Z}^{t,x,u}(s)dB(s),\quad s\in[t,T];\\
\bar{Y}^{t,x,u}(T) & = & \bar{g}(B_{T},\bar{X}^{t,x,u}(T)).
\end{array}\right.
\]
The optimal value field $\bar{v}:[0,T]\times\mathbb{R}^{n}\times\Omega\rightarrow\mathbb{R}$
is given by
\[
\bar{v}(t,x):=\esssup_{u\in\mathcal{U}}\bar{Y}^{t,x,u}(t).
\]
This problem depends on the Brownian path $B_{t}$ and the state $X(t)$.
Now we translate this problem into the path-dependent case. For any
$(\gamma_{t},\xi_{t})\in\Lambda(\mathbb{R}^{d})\times\Lambda(\mathbb{R}^{n})$
and $u\in U$, define $b:\Lambda(\mathbb{R}^{d+n})\times U\rightarrow\mathbb{R}^{d+n}$,
$\sigma:\Lambda(\mathbb{R}^{d+n})\times U\rightarrow\mathbb{R}^{(d+n)\times d}$,
$f:\Lambda(\mathbb{R}^{d+n})\times\mathbb{R}\times\mathbb{R}^{d}\times U\rightarrow\mathbb{R}$,
and $g:\Lambda_{T}(\mathbb{R}^{d+n})\rightarrow\mathbb{R}$ as follows:
\begin{align*}
b\left((\gamma_{t},\xi_{t}),u\right) & :=\left(\begin{array}{c}
\mathbf{0}\\
\bar{b}(\gamma_{t},\xi_{t}(t),u)
\end{array}\right),\\
\sigma\left((\gamma_{t},\xi_{t}),u\right) & :=\left(\begin{array}{c}
I_{d}\\
\bar{\sigma}(\gamma_{t},\xi_{t}(t),u)
\end{array}\right),\\
f\left(\gamma_{t},\xi_{t},y,z,u\right) & :=\bar{f}(\gamma_{t},\xi_{t}(t),y,z,u),\\
g(\gamma_{T},\xi_{T}) & :=\bar{g}(\gamma_{T},\xi(T)).
\end{align*}
Following \eqref{eq:diffu proc-2}, \eqref{eq:BSDE-1} and \eqref{eq:defination Of value fun},
for any $(\gamma_{t},\xi_{t})\in\Lambda(\mathbb{R}^{d})\times\Lambda(\mathbb{R}^{n})$
and $u\in\mathcal{U}$, we define $X^{(\gamma_{t},\xi_{t}),u}$, $Y^{(\gamma_{t},\xi_{t}),u}$,
and $\tilde{v}(\gamma_{t},\xi_{t}):=\esssup_{u\in\mathcal{U}}Y^{(\gamma_{t},\xi_{t}),u}(t)$.
Note that $\tilde{v}(\gamma_{t},\xi_{t})$ only depends on the state
$x=\xi_{t}(t)$ of the path $\xi_{t}$ at time $t$---instead of its
whole history up to time $t$, and thus we can rewrite $X^{(\gamma_{t},\xi_{t}),u}$,
$Y^{(\gamma_{t},\xi_{t}),u}$, and $\tilde{v}(\gamma_{t},\xi_{t})$
into $X^{\gamma_{t},x,u}$, $Y^{\gamma_{t},x,u}$, and $\tilde{v}(\gamma_{t},x)$,
respectively. The uniqueness of solution to the FBSDE implies that
for any $(t,x)\in[0,T]\times\mathbb{R}^{n}$,
\begin{align*}
\bar{X}^{t,x,u}(s) & =X^{B_{t},x,u}(s),\quad\text{\ensuremath{P}-a.s. },\quad t\leq s\leq T,\\
\bar{Y}^{t,x,u}(s) & =Y^{B_{t},x,u}(s),\quad\text{\ensuremath{P}-a.s. },\quad t\leq s\leq T,\\
\bar{v}(t,x) & =\tilde{v}(B_{t},x),\quad\text{\ensuremath{P}-a.s. },\quad t\leq s\leq T.
\end{align*}

Furthermore, in view of Theorem \ref{thm:ex-th}, $\tilde{v}(\gamma_{t},x)$
is a solution to the PDE:
\begin{align}
 & {\displaystyle -D_{t}\tilde{v}-\sup_{u\in U}\biggl[\frac{1}{2}\text{Tr}\big({\bar{\sigma}}{\bar{\sigma}}^{T}(\gamma_{t},x,u)\partial_{xx}\tilde{v}\big)+\left\langle {\bar{b}}(\gamma_{t},x,u),\partial_{x}\tilde{v}\right\rangle +\frac{1}{2}\text{Tr}D_{\gamma\gamma}\tilde{v}}\label{eq:path-eq-classic}\\
 & {\displaystyle +{\bar{\sigma}}^{T}(\gamma_{t},x,u)D_{x\gamma}\tilde{v}+{\bar{f}}(\gamma_{t},x,u,D_{\gamma}\tilde{v}+{\bar{\sigma}}^{T}(\gamma_{t},x,u)\partial_{x}\tilde{v},u)\biggr]=0,\quad(\gamma_{t},x)\in\Lambda(\mathbb{R}^{d})\times\mathbb{R}^{n}.}\nonumber
\end{align}
Here, $D_{\gamma}$ and $D_{\gamma\gamma}$ are the path vertical
derivatives in $\gamma_{t}\in\Lambda(\mathbb{R}^{d})$, and $\partial_{x}$
and $\partial_{xx}$ are the classical partial derivatives in the
state variable $x$.

If $\tilde{v}(\gamma_{t},x)$ is smooth enough, applying It\^o formula
to $\tilde{v}(B_{t},x)$, we have
\[
d\tilde{v}(B_{t},x)=(D_{t}\tilde{v}(B_{t},x)+\frac{1}{2}\text{Tr}D_{\gamma\gamma}\tilde{v}(B_{t},x))dt+D_{\gamma}\tilde{v}(B_{t},x)dB(t).
\]
In view of \eqref{eq:path-eq-classic}, we have
\begin{align*}
d\tilde{v}(B_{t},x) & =-\sup_{u\in U}\biggl[\frac{1}{2}\text{Tr}\big({\bar{\sigma}}{\bar{\sigma}}^{T}(B_{t},x,u)\partial_{xx}\tilde{v}\big)+\left\langle {\bar{b}}(B_{t},x,u),\partial_{x}\tilde{v}\right\rangle \\
 & +{\bar{\sigma}}^{T}(B_{t},x,u)D_{x\gamma}\tilde{v}+{\bar{f}}(B_{t},x,u,D_{\gamma}\tilde{v}+{\bar{\sigma}}^{T}(B_{t},x,u)\partial_{x}\tilde{v},u)\biggr]dt\\
 & +D_{\gamma}\tilde{v}(B_{t},x)dB(t).
\end{align*}
Define the pair of $\mathscr{F}_{t}$-adapted processes $(\bar{v}(t,x),p(t,x)):=\big(\tilde{v}(B_{t},x),D_{\gamma}\tilde{v}(B_{t},x)\big)$.
Then we have
\begin{align}
d\bar{v}(t,x) & =-\sup_{u\in U}\biggl\{\frac{1}{2}\text{Tr}\left({\bar{\sigma}}{\bar{\sigma}}^{T}(B_{t},x,u)\partial_{xx}\bar{v}\right)+\left\langle {\bar{b}}(B_{t},x,u),\partial_{x}\bar{v}\right\rangle \label{eq:classical-BSPDE}\\
 & +{\bar{\sigma}}^{T}(B_{t},x,u)\partial_{x}p+{\bar{f}}(B_{t},x,u,p+{\bar{\sigma}}^{T}(B_{t},x,u)\partial_{x}\bar{v},u)\biggr\} dt\nonumber \\
 & +p\, dB(t),\nonumber
\end{align}
with the terminal condition
\begin{equation}
\bar{v}(T,x)={\bar{g}}(B_{T},x).\label{eq:classic-BSPDE-terminal}
\end{equation}
The fully nonlinear BSPDE \eqref{eq:classical-BSPDE} and \eqref{eq:classic-BSPDE-terminal}
with ${\bar{f}}$ being invariant in the third and fourth arguments
$(y,z)$, is the so-called stochastic Bellman equation, introduced
by Peng \cite{peng1992stochastic,peng1997bsde}. \end{example}

\section{Existence of viscosity solutions}

In this section we give the solution of the path-dependent Bellman
equation \eqref{eq:PHJBE} with the help of FBSDEs \eqref{eq:diffu proc-2}
and \eqref{eq:BSDE-1}.

First, let us perturb a path $\gamma_{t}\in\mathbf{C}_{\mu}^{\alpha}$.
For $\mu>0$, $\varepsilon\in(0,\mu)$ and $\gamma_{t}\in\mathbf{C}_{\mu}^{\alpha}$,
define a perturbation of $\gamma_{t}$ in the following manner:
\begin{equation}
\gamma_{t}^{\varepsilon}(s):=\begin{cases}
\gamma_{t}(s), & |\gamma_{t}(s)-\gamma_{t}(t)|\leq(\mu-\varepsilon)|s-t|^{\alpha};\\
\gamma_{t}(t)+(\mu-\varepsilon)(t-s)^{\alpha}\frac{\gamma_{t}(s)-\gamma_{t}(t)}{|\gamma_{t}(s)-\gamma_{t}(t)|}, & |\gamma_{t}(s)-\gamma_{t}(t)|>(\mu-\varepsilon)|s-t|^{\alpha}.
\end{cases}\label{eq:perturbation}
\end{equation}

We have

\begin{lem} \label{lem:perturbation} Let $\mu>0,M_{0}>0$. Assume
that $\llbracket\gamma_{t}\rrbracket_{\alpha}\leq\mu,\,\|\gamma_{t}\|_{0}\leq M_{0},$
and $\varepsilon\leq\frac{1}{2}\mu$. We have

(i) $\|\gamma_{t}^{\varepsilon}-\gamma_{t}\|_{0}\leq2M_{0}\varepsilon(\mu-\varepsilon)^{-1}\leq4M_{0}\varepsilon\mu^{-1}$;

(ii) $\llbracket\gamma_{t}^{\varepsilon}\rrbracket_{\alpha}\leq\mu$;

(iii) there is a constant $C$, independent of $\mu$ and $u\in\mathcal{U}$,
such that for some $p,\, p(\frac{1}{2}-\alpha)>1$ and for all $\delta<T-t$,
\[
P\left\{ \llbracket X_{t+\delta}^{\gamma_{t}^{\varepsilon},u}\rrbracket_{\alpha}>\mu\right\} \leq C\delta^{p(\frac{1}{2}-\alpha)}\varepsilon^{-p}.
\]
\end{lem}

\begin{proof} Assertion (i) is obvious. Now we prove Assertion (ii).

Since $|\gamma_{t}^{\varepsilon}(s)-\gamma_{t}^{\varepsilon}(t)|=|\gamma_{t}^{\varepsilon}(s)-\gamma_{t}(t)|\leq\mu|s-t|^{\alpha}$
for $s\in[0,t)$, it is sufficient to show that for any $s_{1},s_{2}\in[0,t)$
such that $s_{1}>s_{2}$, we have
\begin{equation}
|\gamma_{t}^{\varepsilon}(s_{1})-\gamma_{t}^{\varepsilon}(s_{2})|\leq\mu|s_{1}-s_{2}|^{\alpha}.\label{eq:gamma-epsilon-holder}
\end{equation}

Define
\begin{alignat*}{1}
 & r_{1}:=|\gamma_{t}^{\varepsilon}(s_{1})-\gamma_{t}(t)|,\quad r_{2}:=|\gamma_{t}^{\varepsilon}(s_{2})-\gamma_{t}(t)|,\\
 & x_{1}:=|\gamma_{t}(s_{1})-\gamma_{t}(t)|-r_{1},\quad x_{2}:=|\gamma_{t}(s_{2})-\gamma_{t}(t)|-r_{2},\\
 & C_{\varepsilon}^{2}:=|\gamma_{t}^{\varepsilon}(s_{1})-\gamma_{t}^{\varepsilon}(s_{2})|^{2}=r_{1}^{2}+r_{2}^{2}-2r_{1}r_{2}\cos\theta,\\
 & C^{2}:=|\gamma_{t}(s_{1})-\gamma_{t}(s_{2})|^{2}=(r_{1}+x_{1})^{2}+(r_{2}+x_{2})^{2}-2(r_{1}+x_{1})(r_{2}+x_{2})\cos\theta.
\end{alignat*}
Here, $\theta$ is the angle between both vectors $\gamma_{t}(s_{1})-\gamma_{t}(t)$
and $\gamma_{t}(s_{2})-\gamma_{t}(t)$ and it is equal to the angle
between both vectors $\gamma_{t}^{\varepsilon}(s_{1})-\gamma_{t}(t)$
and $\gamma_{t}^{\varepsilon}(s_{2})-\gamma_{t}(t)$. We have
\begin{equation}
0\le r_{1}\le(\mu-\varepsilon)(t-s_{1})^{\alpha},\quad0\le r_{2}\le(\mu-\varepsilon)(t-s_{2})^{\alpha},\quad x_{1}\ge0,\quad x_{2}\ge0.
\end{equation}
and
\begin{equation}
C^{2}-C_{\varepsilon}^{2}=x_{1}^{2}+x_{2}^{2}+2r_{1}x_{1}+2r_{2}x_{2}-2(x_{1}x_{2}+r_{1}x_{2}+r_{2}x_{1})\cos\theta.
\end{equation}

The proof of inequality \eqref{eq:gamma-epsilon-holder} is divided
into the following three cases.

\medskip{}
 \textit{ The case of $x_{1}=0$. } We have
\begin{equation}
C^{2}-C_{\varepsilon}^{2}=x_{2}^{2}+2r_{2}x_{2}-2r_{1}x_{2}\cos\theta=x_{2}^{2}+2x_{2}(r_{2}-r_{1}\cos\theta).\label{difference}
\end{equation}
We assert that $C_{\varepsilon}^{2}\le C^{2}$, which implies \eqref{eq:gamma-epsilon-holder}
immediately. It is obvious if $x_{2}=0$. If $x_{2}>0$, we have
\[
r_{2}=(\mu-\varepsilon)|t-s_{2}|^{\alpha}>(\mu-\varepsilon)|t-s_{1}|^{\alpha}\ge r_{1}
\]
by the definition \eqref{eq:perturbation}, which together with equality
\eqref{difference} gives $C^{2}-C_{\varepsilon}^{2}\ge0$.

\medskip{}
 \textit{The case of $x_{1}>0$ and $r_{2}\le r_{1}$. } We have
\[
r_{2}\le r_{1}=(\mu-\varepsilon)|t-s_{1}|^{\alpha}<(\mu-\varepsilon)|t-s_{2}|^{\alpha}.
\]
Therefore, we have $x_{2}=0$ from the definition \eqref{eq:perturbation},
and thus
\[
C^{2}-C_{\varepsilon}^{2}=x_{1}^{2}+2r_{1}x_{1}-2r_{2}x_{1}\cos\theta=x_{1}^{2}+2x_{1}(r_{1}-r_{2}\cos\theta)\ge0.
\]

\medskip{}
 \textit{The case of $x_{1}>0$ and $r_{2}>r_{1}$. } We have
\[
r_{2}\leq(\mu-\varepsilon)|t-s_{2}|^{\alpha},\quad r_{1}=(\mu-\varepsilon)|t-s_{1}|^{\alpha}.
\]
If $C^{2}\ge C_{\varepsilon}^{2}$, the proof is complete. If $C^{2}<C_{\varepsilon}^{2}$,
we have
\[
\cos\theta>\frac{x_{1}^{2}+x_{2}^{2}+2r_{1}x_{1}+2r_{2}x_{2}}{2(x_{1}x_{2}+r_{1}x_{2}+r_{2}x_{1})}\geq\frac{r_{1}}{r_{2}}.
\]
Then
\begin{align*}
C_{\varepsilon}^{2} & =r_{1}^{2}+r_{2}^{2}-2r_{1}r_{2}\cos\theta\le r_{1}^{2}+r_{2}^{2}-2r_{1}^{2}=r_{2}^{2}-r_{1}^{2}\\
 & \le(\mu-\varepsilon)^{2}|t-s_{2}|^{2\alpha}-(\mu-\varepsilon)^{2}|t-s_{1}|^{2\alpha}\leq(\mu-\varepsilon)^{2}|s_{1}-s_{2}|^{2\alpha},
\end{align*}
and thus \eqref{eq:gamma-epsilon-holder} holds. The last inequality
is deduced from the following fact: if $2\alpha\in[0,1]$, then $a^{2\alpha}+b^{2\alpha}\geq(a+b)^{2\alpha}$
for all $a>0,\, b>0$.

\bigskip{}
 It remains to show Assertion (iii). For any $\delta<T-t$ and $\bar{\gamma}_{t+\delta}\in\Lambda$
such that
\[
\sup_{t\leq s_{1}<s_{2}\leq t+\delta}\frac{|\bar{\gamma}_{t+\delta}(s_{1})-\bar{\gamma}_{t+\delta}(s_{2})|}{|s_{1}-s_{2}|^{\alpha}}\leq\varepsilon,
\]
in view of \eqref{eq:perturbation} and Assertion (ii), we have $\llbracket\bar{\gamma}_{t+\delta}^{\gamma_{t}^{\varepsilon}}\rrbracket_{\alpha}\leq\mu$.
Therefore, we have
\[
\left\{ \llbracket X_{t+\delta}^{\gamma_{t}^{\varepsilon},u}\rrbracket_{\alpha}>\mu\right\} \subset\left\{ \sup_{t\leq s_{1}<s_{2}\leq t+\delta}\frac{|X^{\gamma_{t}^{\varepsilon},u}(s_{1})-X^{\gamma_{t}^{\varepsilon},u}(s_{2})|}{|s_{1}-s_{2}|^{\alpha}}>\varepsilon\right\} .
\]
Assertion (iii) then follows from Proposition \ref{Prop:Hold-norm-prob-est}
in the Appendix. \end{proof} \bigskip{}

\begin{proof}[Proof of Theorem \ref{thm:ex-th}] Firstly, we show
that $\tilde{v}$ is a viscosity sub-solution. Let $M_{0}>0,\mu>0,$
and $\kappa\in(0,T)$. For $\gamma_{t}\in\mathbf{Q}_{M_{0},T-\kappa}\cap\,\mathbf{C}_{\mu}^{\alpha}$
and $\psi\in\mathcal{J}_{\mu,\kappa}^{+}(\gamma_{t},\tilde{v})$.
Note that the cylinder $\mathbf{Q}_{M_{0},T-\kappa}(\gamma_{t})$
is defined by \eqref{eq:glob-cylinder}.

For any $\mu>1$ and $\varepsilon<\frac{1}{2}\wedge(\frac{1}{8}\kappa\mu M_{0}^{-1})$,
from Assertion (ii) of Lemma \ref{lem:perturbation}, we have
\[
\|\gamma_{t}^{\varepsilon}-\gamma_{t}\|_{0}\leq4M_{0}\varepsilon\mu^{-1}<\frac{1}{2}\kappa.
\]
For any $\mu>1$ and $u\in\mathcal{U}$, we define an $\mathscr{F}$-stopping
time
\[
\hat{\tau}^{\varepsilon}:=\inf\left\{ s>t:\llbracket X_{s}^{\gamma_{t}^{\varepsilon},u}\rrbracket_{\alpha}>\mu\right\} \wedge\inf\left\{ s>t:\|X_{s}^{\gamma_{t}^{\varepsilon},u}-\gamma_{t,s}\|_{0}>\kappa\right\} \wedge(t+\kappa).
\]
Obviously, $X_{\hat{\tau}^{\varepsilon}}^{\gamma_{t}^{\varepsilon},u}\in\mathbf{Q}_{\kappa,\kappa}(\gamma_{t})\cap\,\mathbf{C}_{\mu}^{\alpha}$,
and for any $\delta<\kappa$,
\[
\left\{ \hat{\tau}^{\varepsilon}\geq t+\delta\right\} \supset\left\{ \llbracket X_{t+\delta}^{\gamma_{t}^{\varepsilon},u}\rrbracket_{\alpha}\leq\mu\right\} \cap\left\{ \|X_{t+\delta}^{\gamma_{t}^{\varepsilon},u}-\gamma_{t,t+\delta}^{\varepsilon}\|_{0}\leq\frac{1}{2}\kappa\right\} .
\]
Therefore,
\[
P\left\{ \hat{\tau}^{\varepsilon}\geq t+\delta\right\} \ge P\left\{ \llbracket X_{t+\delta}^{\gamma_{t}^{\varepsilon},u}\rrbracket_{\alpha}\leq\mu\right\} -P\left\{ \|X_{t+\delta}^{\gamma_{t}^{\varepsilon},u}-\gamma_{t,t+\delta}^{\varepsilon}\|_{0}>\frac{1}{2}\kappa\right\} .
\]
From Lemma \ref{lem:FSDE} and Assertion (iii) of Lemma \ref{lem:perturbation},
we have
\begin{align*}
 & P\left\{ \llbracket X_{t+\delta}^{\gamma_{t}^{\varepsilon},u}\rrbracket_{\alpha}\leq\mu\right\} \geq1-C\delta^{p(\frac{1}{2}-\alpha)}\varepsilon^{-p},\\
 & P\left\{ \|X_{t+\delta}^{\gamma_{t}^{\varepsilon},u}-\gamma_{t,t+\delta}^{\varepsilon}\|_{0}>\frac{1}{2}\kappa\right\} \leq C\delta\kappa^{-2}.
\end{align*}
Note that $\hat{\tau}^{\varepsilon}$ depends on $u$ and $\mu$,
while the R.H.S. of both inequalities are independent of the pair
$(u,\mu)$. Hence, uniformly with respect to $(u,\mu)$,
\[
P\left\{ \hat{\tau}^{\varepsilon}\geq t+\delta\right\} \geq1-C\delta^{p(\frac{1}{2}-\alpha)}\varepsilon^{-p}-C\delta\kappa^{-2}\nearrow1\quad\text{ as }\delta\to0.
\]
where $p(\frac{1}{2}-\alpha)>1$. In particular, there is a positive
constant $\delta_{1}(\varepsilon,\kappa,p)<\kappa$ such that,
\begin{equation}
P\left\{ \hat{\tau}^{\varepsilon}\geq t+\delta\right\} \geq\frac{1}{2},\quad\forall\delta\in(0,\delta_{1}(\varepsilon,\kappa,p)).\label{eq:low-bound-stopping}
\end{equation}

Define
\begin{equation}
\hat{\tau}^{\varepsilon,\delta}:=\hat{\tau}^{\varepsilon}\wedge(t+\delta).\label{eq:ex-th-stopping}
\end{equation}
Applying the functional It\^o formula \eqref{eq:ito formu} to $\psi$
on interval $[t,\hat{\tau}^{\varepsilon,\delta}]$, we have
\begin{align}
\psi(\gamma_{t}^{\varepsilon})= & \psi(X_{\hat{\tau}^{\varepsilon,\delta}}^{\gamma_{t}^{\varepsilon},u})-\int_{t}^{\hat{\tau}^{\varepsilon,\delta}}(\mathscr{L}\psi)(X_{r}^{\gamma_{t}^{\varepsilon},u},u(r))dr\label{eq:ex-th-1}\\
 & +\int_{t}^{\hat{\tau}^{\varepsilon,\delta}}f(X_{r}^{\gamma_{t}^{\varepsilon},u},\psi(X_{r}^{\gamma_{t}^{\varepsilon},u}),\sigma^{T}(X_{r}^{\gamma_{t}^{\varepsilon},u},u(r))D_{x}\psi(X_{r}^{\gamma_{t}^{\varepsilon},u}),u(r))\, dr\nonumber \\
 & -\int_{t}^{\hat{\tau}^{\varepsilon,\delta}}\big[D_{x}\psi(X_{r}^{\gamma_{t}^{\varepsilon},u})\big]^{T}\sigma(X_{r}^{\gamma_{t}^{\varepsilon},u},u(r))\, dW(r),\nonumber
\end{align}
where $\mathscr{L}$ is defined as \eqref{eq:L-operator}. Let $(Y^{1,\varepsilon,\delta,u},Z^{1,\varepsilon,\delta,u})$
be the solution of the following BSDE
\begin{equation}
\begin{cases}
-dY(r)= & f(X_{r}^{\gamma_{t}^{\varepsilon},u},Y(r),Z(r),u(r))\, dr-Z(r)\, dW(r),\quad r\in[t,\hat{\tau}^{\varepsilon,\delta}];\\
Y(\hat{\tau}^{\varepsilon,\delta})= & \tilde{v}(X_{\hat{\tau}^{\varepsilon,u}}^{\gamma_{t}^{\varepsilon},\delta}).
\end{cases}\label{eq:ex-th-2}
\end{equation}
Set
\begin{align}
Y^{2,\varepsilon,\delta,u}(s):= & \psi(X_{s}^{\gamma_{t},u})-Y^{1,\varepsilon,\delta,u}(s),\label{eq:delta-Y}\\
Z^{2,\varepsilon,\delta,u}(s):= & \sigma^{T}(X_{s}^{\gamma_{t},u},u(s))D_{x}\psi(X_{s}^{\gamma_{t},u})-Z^{1,\varepsilon,\delta,u}(s).\nonumber
\end{align}
Comparing \eqref{eq:ex-th-1} and \eqref{eq:ex-th-2}, we have for
$r\in[t,\hat{\tau}^{\varepsilon,\delta}]$, $P$-a.s.,
\begin{eqnarray}
-dY^{2,\varepsilon,\delta,u}(r) & = & {\displaystyle \biggl[-(\mathscr{L}\psi)(X_{r}^{\gamma_{t}^{\varepsilon},u},u(r))}\nonumber \\
 &  & {\displaystyle +f(X_{r}^{\gamma_{t}^{\varepsilon},u},\psi(X_{r}^{\gamma_{t}^{\varepsilon},u}),\sigma^{T}(X_{r}^{\gamma_{t}^{\varepsilon},u},u(r))D_{x}\psi(X_{r}^{\gamma_{t}^{\varepsilon},u}),u(r))}\label{eq:diff-ex-th}\\
 &  & {\displaystyle -f(X_{r}^{\gamma_{t}^{\varepsilon},u},Y^{1,\varepsilon,\delta,u}(r),Z^{1,\varepsilon,\delta,u}(r),u(r))\biggr]\, dr-Z^{2,\varepsilon,u,\delta}(r)\, dW(r)}\nonumber \\
 & = & \left[-(\mathscr{L}\psi)(X_{r}^{\gamma_{t}^{\varepsilon},u},u(r))+A(r)Y^{2,\varepsilon,u,\delta}(r)+\langle\bar{A},Z^{2,\varepsilon,u,\delta}\rangle(r)\right]\, dr\nonumber \\
 &  & {\displaystyle -Z^{2,\varepsilon,\delta,u}(r)\, dW(r),}\nonumber
\end{eqnarray}
where $|A|,|\bar{A}|\leq C$ ($C$ depends on Lipschitz constant of
$f$, and is independent of the triplet $(u,\varepsilon,\delta)$.
Therefore, we have (see \cite[Proposition 2.2]{el1997backward})
\begin{eqnarray}
Y^{2,\varepsilon,\delta,u}(t)\!\!\! & = & \!\!\!{\displaystyle E\!\left[Y^{2,\varepsilon,\delta,u}(\hat{\tau}^{\varepsilon,\delta})\Gamma^{t}(\hat{\tau}^{\varepsilon,\delta})-\!\int_{t}^{\hat{\tau}^{\varepsilon,\delta}}\Gamma^{t}(r)(\mathscr{L}\psi)(X_{r}^{\gamma_{t}^{\varepsilon},u},u(r))dr\Big|\mathscr{F}_{t}\right],}\label{eq:pres-diff-Exis-Th}
\end{eqnarray}
where $\Gamma^{t}(\cdot)$ solves the linear SDE
\[
d\Gamma^{t}(s)=\Gamma^{t}(s)\Big(A(s)\, ds+\bar{A}(s)\, dW(s)\Big),\: s\in[t,\hat{\tau}^{\varepsilon,\delta}];\quad\Gamma^{t}(t)=1.
\]
Obviously, $\Gamma^{t}\geq0$. Since $\psi\in\mathcal{J}_{\mu,\kappa}^{+}(\gamma_{t},\tilde{v})$,
$\psi-\tilde{v}$ is minimized at $\gamma_{t}$ over $\mathbf{Q}_{\kappa,\iota}(\gamma_{t})\cap\,\mathbf{C}_{\mu}^{\alpha}$,
and in view of Theorem \ref{thm:DPP-1} and Proposition \ref{determin},
we have
\begin{align}
Y^{2,\varepsilon,\delta,u}(\hat{\tau}^{\varepsilon,\delta}) & \geq Y^{2,\varepsilon,\delta,u}(t)=0,\label{eq:sub-case-term-cond}\\
\inf_{u\in\mathcal{U}}EY^{2,\varepsilon,\delta,u}(t)=\essinf_{u\in\mathcal{U}}Y^{2,\varepsilon,\delta,u}(t) & =\psi(\gamma_{t}^{\varepsilon})-\esssup_{u\in\mathcal{U}}Y^{1,\varepsilon,\delta,u}(\hat{\tau}^{\varepsilon,\delta})\label{eq:sub-case-ini-cond}\\
 & =\psi(\gamma_{t}^{\varepsilon})-\tilde{v}(\gamma_{t}^{\varepsilon}).\nonumber
\end{align}
From equation \eqref{eq:pres-diff-Exis-Th}, we have
\begin{align}
\psi(\gamma_{t}^{\varepsilon})-\tilde{v}(\gamma_{t}^{\varepsilon})\geq & \inf_{u\in\mathcal{U}}E[Y^{2,\varepsilon,\delta,u}(t)]\label{eq:ex-th-3}\\
= & \inf_{u\in\mathcal{U}}E\left[Y^{2,\varepsilon,\delta,u}(\hat{\tau}^{\varepsilon,\delta})\Gamma^{t}(\hat{\tau}^{\varepsilon,\delta})-\int_{t}^{\hat{\tau}^{\varepsilon,\delta}}\Gamma^{t}(r)(\mathscr{L}\psi)(X_{r}^{\gamma_{t}^{\varepsilon},u},u(r))dr\right]\nonumber \\
\geq & -\sup_{u\in\mathcal{U}}E\left[\int_{t}^{\hat{\tau}^{\varepsilon,\delta}}\Gamma^{t}(r)(\mathscr{L}\psi)(X_{r}^{\gamma_{t}^{\varepsilon},u},u(r))dr\right]\nonumber \\
= & -\sup_{u\in\mathcal{U}}E\left[\int_{t}^{\hat{\tau}^{\varepsilon,\delta}}(\mathscr{L}\psi)(\gamma_{t}^{\varepsilon},u(r))dr\right]\nonumber \\
 & -\sup_{u\in\mathcal{U}}E\left[\int_{t}^{\hat{\tau}^{\varepsilon,\delta}}\left[(\mathscr{L}\psi)(X_{r}^{\gamma_{t}^{\varepsilon},u},u(r))-(\mathscr{L}\psi)(\gamma_{t}^{\varepsilon},u(r))\right]dr\right]\nonumber \\
 & -\sup_{u\in\mathcal{U}}E\left[\int_{t}^{\hat{\tau}^{\varepsilon,\delta}}(\Gamma^{t}(r)-1)(\mathscr{L}\psi)(X_{r}^{\gamma_{t}^{\varepsilon},u},u(r))dr\right]\nonumber \\
:= & -\sup_{u\in\mathcal{U}}\text{Part1}-\sup_{u\in\mathcal{U}}\text{Part2}-\sup_{u\in\mathcal{U}}\text{Part3}.\nonumber
\end{align}

Since the coefficients in $\mathscr{L}$ are Lipschitz continuous,
combining the regularity of $\psi$ (see \eqref{eq:super-jet}), we
have for any $\gamma_{t^{1}}^{1},\gamma_{t^{2}}^{2}\in\mathbf{Q}_{\kappa,\kappa}(\gamma_{t})\cap\,\mathbf{C}_{\mu}^{\alpha}$
and $\bar{u}\in U$,
\begin{align}
|\psi(\gamma_{t^{1}}^{1})-\psi(\gamma_{t^{2}}^{2})|\leq & d_{p}^{\beta}(\gamma_{t^{1}}^{1},\gamma_{t^{2}}^{2}),\nonumber \\
\left|\mathscr{L}\psi(\gamma_{t^{1}}^{1},\bar{u})-\mathscr{L}\psi(\gamma_{t^{2}}^{2},\bar{u})\right|\leq & Cd_{p}^{\beta}\left(\gamma_{t^{1}}^{1},\gamma_{t^{2}}^{2}\right).\label{eq:op-L-cont}
\end{align}

Thus we have
\begin{align}
 & \psi(\gamma_{t}^{\varepsilon})-\tilde{v}(\gamma_{t}^{\varepsilon})=(\psi-\tilde{v})(\gamma_{t})+\psi(\gamma_{t}^{\varepsilon})-\psi(\gamma_{t})+\tilde{v}(\gamma_{t}^{\varepsilon})-\tilde{v}(\gamma_{t})\label{eq:ex-th-0}\\
\leq & C\left(\|\gamma_{t}^{\varepsilon}-\gamma_{t}\|_{0}^{\beta}+\|\gamma_{t}^{\varepsilon}-\gamma_{t}\|_{0}\right)\leq C\left(4M_{0}\varepsilon\mu^{-1}\right)^{\beta},\nonumber
\end{align}
and
\begin{align}
 & \sup_{u\in\mathcal{U}}\text{Part1}\label{eq:ex-th-I}\\
\leq & \sup_{u\in\mathcal{U}}E\left[(\hat{\tau}^{\varepsilon,\delta}-t)\sup_{\bar{u}\in U}\mathscr{L}\psi(\gamma_{t}^{\varepsilon},\bar{u})\right]\leq\sup_{\bar{u}\in U}\mathscr{L}\psi(\gamma_{t}^{\varepsilon},\bar{u})\sup_{u\in\mathcal{U}}E[(\hat{\tau}^{\varepsilon,\delta}-t)]\nonumber \\
\leq & \left(\sup_{\bar{u}\in U}\mathscr{L}\psi(\gamma_{t},\bar{u})+C\left(4M_{0}\varepsilon\mu^{-1}\right)^{\beta}\right)\sup_{u\in\mathcal{U}}E[(\hat{\tau}^{\varepsilon,\delta}-t)].\nonumber
\end{align}

Now we estimate higher order terms Part2 and Part3. In view of Lemma
\ref{lem:FSDE} and \eqref{eq:op-L-cont}, we have
\begin{align*}
 & E\left[\sup_{t\leq r\leq\hat{\tau}^{\varepsilon,\delta}}|\mathscr{L}\psi(X_{r}^{\gamma_{t}^{\varepsilon},u},u(r))-\mathscr{L}\psi(\gamma_{t}^{\varepsilon},u(r))|\right]\\
\leq & CEd_{p}^{\beta}(X_{\hat{\tau}^{\varepsilon,\delta}}^{\gamma_{t}^{\varepsilon},u},\gamma_{t}^{\varepsilon})\leq C\delta^{\frac{\beta}{2}}.
\end{align*}
Hence
\begin{align}
|\text{Part2}| & \leq E\left[(\hat{\tau}^{\varepsilon,\delta}-t)\sup_{t\leq r\leq\hat{\tau}^{\varepsilon,\delta}}|\mathscr{L}\psi(X_{r}^{\gamma_{t}^{\varepsilon},u},u(r))-\mathscr{L}\psi(\gamma_{t}^{\varepsilon},u(r))|\right]\nonumber \\
 & \leq\delta E\left[\sup_{t\leq r\leq\hat{\tau}^{\varepsilon,\delta}}|\mathscr{L}\psi(X_{r}^{\gamma_{t}^{\varepsilon},u},u(r))-\mathscr{L}\psi(\gamma_{t}^{\varepsilon},u(r))|\right]\nonumber \\
 & =C\delta^{1+\frac{\beta}{2}}\label{eq:ex-th-II}
\end{align}
and
\begin{align}
|\text{Part3}| & \leq CE\int_{t}^{\hat{\tau}^{\varepsilon,\delta}}|\Gamma^{t}(r)-1|dr\leq CE\left[(\hat{\tau}^{\varepsilon,\delta}-t)\sup_{t\leq r\leq\hat{\tau}^{\varepsilon_{1},\delta}}|\Gamma^{t}(r)-1|\right]\nonumber \\
 & \leq C\delta E\left[\sup_{t\leq r\leq\hat{\tau}^{\varepsilon,\delta}}|\Gamma^{t}(r)-1|\right]\leq C\delta^{\frac{3}{2}}.\label{eq:ex-th-III}
\end{align}

Substituting \eqref{eq:ex-th-0} - \eqref{eq:ex-th-III} into \eqref{eq:ex-th-3},
we have
\begin{align}
 & -C\left(4M_{0}\varepsilon\mu^{-1}\right)^{\beta}\label{eq:ex-th-4}\\
\leq & \left(\sup_{\bar{u}\in U}\mathscr{L}\psi(\gamma_{t},\bar{u})+C\left(4M_{0}\varepsilon\mu^{-1}\right)^{\beta}\right)\sup_{u\in\mathcal{U}}E[(\hat{\tau}^{\varepsilon,\delta}-t)]+C\delta^{1+\frac{\beta}{2}}.\nonumber
\end{align}
In view of \eqref{eq:low-bound-stopping}, then for all $\delta\in(0,\delta_{1}(\kappa,\varepsilon,p))$,
uniformly for every $u\in\mathcal{U}$, we have
\begin{equation}
E[\hat{\tau}^{\varepsilon,\delta}-t]\geq E\left[\chi_{\{\hat{\tau}^{\varepsilon,\delta}\geq t+\delta\}}(\hat{\tau}^{\varepsilon,\delta}-t)\right]\geq\frac{1}{2}\delta.\label{eq:est-of-stopping}
\end{equation}
Therefore,
\begin{equation}
-\sup_{u\in U}\mathscr{L}\psi(\gamma_{t},u)\leq C\left(4M_{0}\varepsilon\mu^{-1}\right)^{\beta}(2\delta^{-1}+1)+2C\delta^{\frac{\beta}{2}}.\label{eq:last-estim}
\end{equation}
Note that the constants $C$ throughout this proof only depend on
$M_{0},\kappa$ and the Lipschitz constants of the coefficients in
$\mathscr{L}$, and they do not depend on $\psi\in\mathcal{J}_{\mu,\kappa}^{+}(\gamma_{t},\tilde{v})$,
$\gamma_{t}\in\mathbf{Q}_{M_{0},T-\kappa}\cap\,\mathbf{C}_{\mu}^{\alpha}$,
and $\delta\in(0,\delta_{1}(\varepsilon,\kappa))$. Taking the supremum
on both sides over $\psi\in\mathcal{J}_{\mu,\kappa}^{+}(\gamma_{t},\tilde{v})$
and $\gamma_{t}\in\mathbf{Q}_{M_{0},T-\kappa}\cap\,\mathbf{C}_{\mu}^{\alpha}$,
setting $\delta:=\mu^{-\frac{\beta}{2}}$, and then sending $\mu$
to $\infty$, we have \eqref{eq:subsolution}. This shows that $\tilde{v}$
is a viscosity sub-solution to the path-dependent Bellman equation
\eqref{eq:PHJBE}.

In a symmetric (also easier) way, we show that $\tilde{v}$ is a super-solution
to the path-dependent Bellman equation \eqref{eq:PHJBE}. The proof
is complete.\end{proof}

\begin{rem}\label{rem:exist-th} (1) Our existence proof is more
complicated than the classical counterpart (for the state-dependent
case). The complication arises from the fact that we start the dynamic
programming at the perturbation $\gamma_{t}^{\varepsilon}$ instead
of directly at the minimum path $\gamma_{t}$ of $\psi-{\tilde{v}}$
like the conventional arguments. Since our jets are defined on some
compact subset $\mathbf{C}_{\mu}^{\alpha}$ of $\Lambda$, the minimum
path $\gamma_{t}$ might happen to be at the boundary of $\mathbf{C}_{\mu}^{\alpha}$,
i.e. $\llbracket\gamma_{t}\rrbracket_{\alpha}=\mu$. If we started
at $\gamma_{t}$, BSDE \eqref{eq:diff-ex-th} would be trivial and
nothing from the localized dynamic programming principle could be
derived if
\begin{equation}
P\{\llbracket X_{s}^{\gamma_{t}}\rrbracket_{\alpha}\leq\mu,\,\exists s>t\}=0.\label{eq:court-example}
\end{equation}
The following example illustrates that \eqref{eq:court-example} might
happen, and therefore explains why we have to start the dynamic programming
at the perturbation $\gamma_{t}^{\varepsilon}$.

Let $W$ be a one-dimensional standard Brownian Motion and $\gamma_{t}\in\Lambda(\mathbb{R})$
such that for some $t_{1}\in[0,t)$
\[
\gamma_{t}(t)-\gamma_{t}(t_{1})=\mu|t-t_{1}|^{\alpha}.
\]
Define
\[
W^{\gamma_{t}}(s):=\gamma_{t}(s)\chi_{[0,t)}(s)+\big(W(s)-W(t)+\gamma_{t}(t)\big)\chi_{[t,T]}(s),\quad s\in[0,T].
\]
Then
\begin{align*}
 & \left\{ \exists\delta>0,\text{ s.t. }\llbracket W_{t+\delta}^{\gamma_{t}}\rrbracket_{\alpha}\leq\mu\right\} \\
\subset & \left\{ \exists\delta>0,\text{ s.t. }W^{\gamma_{t}}(s)-\gamma_{t}(t_{1})\leq\mu|s-t_{1}|^{\alpha},\,\forall s\in(t,t+\delta)\right\} \\
= & \left\{ \exists\delta>0,\text{ s.t. }W^{\gamma_{t}}(s)-\gamma_{t}(t)+\gamma_{t}(t)-\gamma_{t}(t_{1})\leq\mu|s-t_{1}|^{\alpha},\,\forall s\in(t,t+\delta)\right\} \\
= & \left\{ \exists\delta>0,\text{ s.t. }W^{\gamma_{t}}(s)-\gamma_{t}(t)\leq\mu(|s-t_{1}|^{\alpha}-|t-t_{1}|^{\alpha}),\,\forall s\in(t,t+\delta)\right\} .
\end{align*}
Since the function $\mu(|\cdot-t_{1}|^{\alpha}-|t-t_{1}|^{\alpha})\in C^{1}[t,t+\delta]$,
by the law of iterated logarithm (see \cite[Theorem 9.23, Chapter 2]{karatzas1991brownian}),
we have
\[
P\left\{ \exists\delta>0,\text{ s.t. }\llbracket W_{t+\delta}^{\gamma_{t}}\rrbracket_{\alpha}\leq\mu\right\} =0.
\]
This example enlightens us to perturb the left $\mu$-H\"older modulus
of $\gamma_{t}$ at time $t$ in \eqref{eq:perturbation}.

(2) The introduction of $\mathbf{Q}_{M_{0},T-\kappa}$ in Definition
\ref{defn:viscosity solution} plays a crucial role in the proof of
Theorem \ref{thm:ex-th}. Otherwise, we only have the following too
rough estimate on our perturbation: $\|\gamma_{t}^{\varepsilon}-\gamma_{t}\|_{0}\leq C\varepsilon$,
from which and \eqref{eq:last-estim} only results the following inequality
\begin{align*}
-\sup_{u\in U}\mathscr{L}\psi(\gamma_{t},u) & \leq C\varepsilon^{\beta}(2\delta^{-1}+1)+2C\delta^{\frac{1}{2}}.
\end{align*}
It does not help us, for the relation of $\delta^{\frac{1}{2}-\alpha}=\text{o}(\varepsilon)$
is required in the estimate \eqref{eq:est-of-stopping} by Proposition
\ref{Prop:Hold-norm-prob-est} and implies that $\varepsilon^{\beta}\delta^{-1}$
increases to $\infty$ as $\delta$ is decreasing to zero.

However, with the restriction of $\gamma_{t}\in\mathbf{Q}_{M_{0},T-\kappa}$,
in \eqref{eq:last-estim} we could fix $\varepsilon$, while sending
$\delta\to0$ and $(4M_{0}\varepsilon\mu^{-1})^{\beta}\delta^{-1}\to0$
simultaneously.

(3) In the above proof, both parameters $\mu$
and $M_{0}$ in our definition of viscosity sub-solutions play a key role, while the parameter $\kappa$ is fixed
such that the following associated family of path functionals
\[
\{\psi,D_{t}\psi,D_{x}\psi,D_{xx}\psi:\mathbf{Q}_{\kappa,\iota}(\gamma_{t})\cap\mathbf{C}_{\mu}^{\alpha}\to\mathbb{R};\gamma_{t}\in\mathbf{Q}_{M_{0},T-\kappa}\cap\mathbf{C}_{\mu}^{\alpha},\psi\in\mathcal{J}_{\mu,\kappa}^{+}(\gamma_{t},u),\mu\ge\mu_{0}\}
\]
for some sufficiently large $\mu_{0}$, share a common H\"older modulus,
which implies the so-called equi-continuous but with the underlying
functionals being considered on varying domains.

\end{rem}

\section{\label{sec:Uniqueness-of-viscosity}Uniqueness of viscosity solution}

\subsection{Non-degenerate case}

We assume without loss of generality that, there exists a constant
$K>0$, such that, for all $(\gamma_{t},p,A,u)\in\Lambda\times\mathbb{R}^{n}\times\mathbb{R}^{n\times n}\times U$
and $r_{1},r_{2}\in\mathbb{R}$ such that $r_{1}<r_{2}$ ,
\begin{equation}
\mathcal{H}(\gamma_{t},r_{1},p,A,u)-\mathcal{H}(\gamma_{t},r_{2},p,A,u)\geq K(r_{2}-r_{1}).\label{eq:Monotonicity-of-h}
\end{equation}
Otherwise, define $\bar{v}(\gamma_{t})=e^{-\lambda t}v(\gamma_{t})$
for $\lambda>0$. Then $v$ is a viscosity solution of PHJB equation
\eqref{eq:PHJBE} if and only if $\bar{v}$ is a viscosity solution
of the following PPDE
\[
\begin{cases}
{\displaystyle -D_{t}\bar{v}-\sup_{u\in U}\bar{\mathcal{H}}(\gamma_{t},\bar{v},D_{x}\bar{v},D_{xx}^{2}\bar{v},u)=0,} & \quad\gamma_{t}\in\Lambda;\\
\bar{v}(\gamma_{T})=e^{-\lambda T}g(\gamma_{T}), & \quad\gamma_{T}\in\Lambda_{T},
\end{cases}
\]
where
\[
\bar{\mathcal{H}}(\gamma_{t},r,p,A,u):=-\lambda r+e^{-\lambda t}\mathcal{H}(\gamma_{t},e^{\lambda t}r,e^{\lambda t}p,e^{\lambda t}A,u).
\]
Obviously, $\bar{\mathcal{H}}$ satisfies \eqref{eq:Monotonicity-of-h}
for sufficiently large $\lambda$.

\subsubsection{\label{sub:Smooth-approximations} State-dependent smooth approximations}

First, we construct the state-dependent approximations of the path
functional $\tilde{v}$ defined by \eqref{eq:defination Of value fun}.
Let $m$ be a positive integer, and $t_{i}:=\frac{i}{m}T,i=0,1,\cdots m,$
which divide the time interval $[0,T]$ into $m$ equal parts. Now
for all $t\in[0,T]$, we define the truncating operator $\mathbf{P}^{m}:\Lambda_{t}\rightarrow\hat{\Lambda}_{t}$
by
\begin{eqnarray*}
(\mathbf{P}^{m}\gamma_{t})(r) & = & \sum_{i=0}^{k-2}\gamma_{t}(t_{i})\chi_{[t_{i},t_{i+1})}(r)+\gamma_{t}(t_{k-1})\chi_{[t_{k-1},t)}(r)+\gamma_{t}(t)\chi_{\{t\}}(r)\\
 & = & \sum_{i=1}^{k-1}\gamma_{t}\big|_{t_{i-1}}^{t_{i}}\chi_{[t_{i-1},t]}(r)+\gamma_{t}\big|_{t_{k-1}}^{t}\chi_{\{t\}}(r),
\end{eqnarray*}
where $\gamma_{t}\big|_{t_{i-1}}^{t_{i}}:=\gamma_{t}(t_{i})-\gamma_{t}(t_{i-1})$,
and $k$ is the positive integer such that $t\in(t_{k-1},t_{k}]$.

We define functions $b^{m}:\hat{\Lambda}\times U\rightarrow\mathbb{R}^{n}$,
$\sigma^{m}:\hat{\Lambda}\times U\rightarrow\mathbb{R}^{n\times d}$,
$f^{m}:\hat{\Lambda}\times\mathbb{R}\times\mathbb{R}^{d}\times U\rightarrow\mathbb{R}$,
$g^{m}:\hat{\Lambda}_{T}\rightarrow\mathbb{R}$, and $\mathcal{H}^{m}:\hat{\Lambda}\times\mathbb{R}\times\mathbb{R}^{n}\times\mathbb{R}^{n\times n}\times U\rightarrow\mathbb{R}$
as follows:
\begin{eqnarray*}
b^{m}(\gamma_{t},u) & := & b(\mathbf{P}^{m}\gamma_{t},u),\\
\sigma^{m}(\gamma_{t},u) & := & \sigma(\mathbf{P}^{m}\gamma_{t},u),\\
f^{m}(\gamma_{t},y,z,u) & := & f(\mathbf{P}^{m}\gamma_{t},y,z,u),\\
g^{m}(\gamma_{T}) & := & g(\mathbf{P}^{m}\gamma_{T}),\\
\mathcal{H}^{m}(\gamma_{t},r,p,A,u) & := & \mathcal{H}(\mathbf{P}^{m}\gamma_{t},r,p,A,u).
\end{eqnarray*}
Assumption (H2) implies the following estimates
\begin{align}
|b^{m}(\gamma_{t},u)-b(\gamma_{t},u)| & \le C\|\mathbf{P}^{m}\gamma_{t}-\gamma_{t}\|_{0},\nonumber \\
|\sigma^{m}(\gamma_{t},u)-\sigma(\gamma_{t},u)| & \le C\|\mathbf{P}^{m}\gamma_{t}-\gamma_{t}\|_{0},\nonumber \\
|f^{m}(\gamma_{t},y,z,u)-f(\gamma_{t},y,z,u)| & \le C\|\mathbf{P}^{m}\gamma_{t}-\gamma_{t}\|_{0},\label{eq:coeffient-first-estmat}\\
|g^{m}(\gamma_{t})-g(\gamma_{t})| & \le C\|\mathbf{P}^{m}\gamma_{t}-\gamma_{t}\|_{0},\nonumber \\
|\mathcal{H}^{m}(\gamma_{t},r,p,A,u)-\mathcal{H}(\gamma_{t},r,p,A,u)| & \le C(1+|p|+|A|)\|\mathbf{P}^{m}\gamma_{t}-\gamma_{t}\|_{0}.\nonumber
\end{align}

Consider the following FBSDE: for any $\gamma_{t}\in\Lambda,\, t<T$,
and $u\in\mathcal{U}$,
\begin{equation}
\begin{cases}
X^{m,\gamma_{t},u}(s)=(\mathbf{P}^{m}\gamma_{t})(s), & \text{ all }\omega,\, s\in[0,t],\\
{\displaystyle X^{m,\gamma_{t},u}(s)=(\mathbf{P}^{m}\gamma_{t})(t)+\int_{t}^{s}b^{m}(X_{r}^{m,\gamma_{t},u},u(r))dr}\\[3mm]
{\displaystyle \quad\quad+\int_{t}^{s}\sigma^{m}(X_{r}^{m,\gamma_{t},u},u(r))\, dW(r),\qquad} & \text{a.s.-}\omega,\, s\in[t,T];
\end{cases}\label{eq:diffu proc-1}
\end{equation}

\begin{eqnarray*}
Y^{m,\gamma_{t},u}(s) & = & g^{m}(X_{T}^{m,\gamma_{t},u})+\int_{s}^{T}f^{m}(X_{r}^{m,\gamma_{t},u},Y^{m,\gamma_{t},u}(r),Z^{m,\gamma_{t},u}(r),u(r))\, dr\\
 &  & -\int_{s}^{T}Z^{m,\gamma_{t},u}(r)\, dW(s),\quad s\in[t,T].
\end{eqnarray*}
Define the first approximating value functional
\[
v^{m}(\gamma_{t}):=\esssup_{u\in\mathcal{U}}Y^{m,\gamma_{t},u}(t),\quad\gamma_{t}\in\Lambda.
\]

\begin{prop} For $(u,\gamma_{t})\in\left(\mathcal{U}\times\Lambda\right)$,
and $p>2$, we have
\begin{equation}
E\big[\|X_{T}^{m,\gamma_{t},u}-X_{T}^{\gamma_{t},u}\|_{0}^{p}\big]\leq C\Big(\text{Osc}(\gamma_{t},m^{-1})^{p}+m^{-\frac{p}{2}}\Big),
\end{equation}

\begin{equation}
E\Big[\sup_{s\in[t,T]}|Y^{m,\gamma_{t},u}(s)-Y^{\gamma_{t},u}(s)|^{p}\Big]\leq C\Big(\text{Osc}(\gamma_{t},m^{-1})^{p}+m^{-\frac{p}{2}}\Big),\label{eq:first-ESTIMAT}
\end{equation}
where
\[
\text{Osc}(\gamma_{t},m^{-1}):=\max_{0<s<s+\delta<t,0<\delta<m^{-1}}|\gamma(s+\delta)-\gamma(s)|
\]
is the oscillating amplitude with time $m^{-1}$ of $\gamma_{t}$
in the interval $(0,t)$.\end{prop}

\begin{proof} From \eqref{eq:coeffient-first-estmat}, we have
\begin{align*}
 & |b^{m}(X_{s}^{m,\gamma_{t},u})-b(X_{s}^{\gamma_{t},u})|\leq C\|\mathbf{P}^{m}X_{s}^{m,\gamma_{t},u}-X_{s}^{\gamma_{t},u}\|_{0}\\
\leq & C(\|\mathbf{P}^{m}X_{s}^{m,\gamma_{t},u}-\mathbf{P}^{m}X_{s}^{\gamma_{t},u}\|_{0}+\|\mathbf{P}^{m}X_{s}^{\gamma_{t},u}-X_{s}^{\gamma_{t},u}\|_{0})\\
\leq & C(\|X_{s}^{m,\gamma_{t},u}-X_{s}^{\gamma_{t},u}\|_{0}+\|\mathbf{P}^{m}X_{s}^{\gamma_{t},u}-X_{s}^{\gamma_{t},u}\|_{0}).
\end{align*}
In a similar way, we have
\begin{align*}
 & |\sigma^{m}(X_{s}^{m,\gamma_{t},u})-\sigma(X_{s}^{\gamma_{t},u})|\leq C(\|X_{s}^{m,\gamma_{t},u}-X_{s}^{\gamma_{t},u}\|_{0}+\|\mathbf{P}^{m}X_{s}^{\gamma_{t},u}-X_{s}^{\gamma_{t},u}\|_{0}),\\
\\
 & |f^{m}(X_{s}^{m,\gamma_{t},u},y',z',u)-f(X_{s}^{\gamma_{t},u},y,z,u)|\\
\leq & C(\|X_{s}^{m,\gamma_{t},u}-X_{s}^{\gamma_{t},u}\|_{0}+\|\mathbf{P}^{m}X_{s}^{\gamma_{t},u}-X_{s}^{\gamma_{t},u}\|_{0})+|y-y'|+|z-z'|,\\
\\
 & |g^{m}(X_{T}^{m,\gamma_{t},u})-g(X_{T}^{m,\gamma_{t},u})|\leq C\left(\|X_{s}^{m,\gamma_{t},u}-X_{s}^{\gamma_{t},u}\|_{0}+\|\mathbf{P}^{m}X_{s}^{\gamma_{t},u}-X_{s}^{\gamma_{t},u}\|_{0}\right).
\end{align*}
Applying It\^o formula, BDG and Gronwall inequality, using standard
arguments, we have
\begin{align*}
E\|X_{T}^{m,\gamma_{t},u}-X_{T}^{\gamma_{t},u}\|_{0}^{p} & \leq CE\|\mathbf{P}^{m}X_{T}^{\gamma_{t},u}-X_{T}^{\gamma_{t},u}\|_{0}^{p}\\
 & =CE\left[\max_{\begin{array}{c}
1\leq i\leq m\\
t_{i-1}\leq s\leq t_{i}
\end{array}}|X^{\gamma_{t},u}(s)-X^{\gamma_{t},u}(t_{i})|\right]\\
 & \leq C\Big(\text{Osc}(\gamma_{t},m^{-1})^{p}+m^{-\frac{p}{2}}\Big),
\end{align*}

\begin{align*}
 & E\|Y^{m,\gamma_{t},u}(T)-Y^{\gamma_{t},u}(T)\|_{0}^{p}\leq C\Big(E\|X_{T}^{m,\gamma_{t},u}-X_{T}^{\gamma_{t},u}\|_{0}^{p}+E\|\mathbf{P}^{m}X_{T}^{\gamma_{t},u}-X_{T}^{\gamma_{t},u}\|_{0}^{p}\Big)\\
\leq & C\Big(\text{Osc}(\gamma_{t},m^{-1})^{p}+m^{-\frac{p}{2}}\Big).
\end{align*}

\end{proof}

Obviously, \eqref{eq:first-ESTIMAT} yields, for any $\gamma_{t}\in\Lambda$
and positive integer $m$,
\begin{equation}
|\tilde{v}(\gamma_{t})-v^{m}(\gamma_{t})|\leq C\Big(\text{Osc}(\gamma_{t},m^{-1})+m^{-\frac{1}{2}}\Big).\label{eq:first-approx-path-rate}
\end{equation}

Similar to the state-dependent optimal stochastic control problem,
$v^{m}$ has a PDE interpretation. For each $m$ and $i=1,\ldots,m$,
define functions $B^{m,i}:(t_{i-1},t_{i}]\times\mathbb{R}^{i\times n}\times U\rightarrow\mathbb{R}^{n}$,
$\Sigma^{m,i}:(t_{i-1},t_{i}]\times\mathbb{R}^{i\times n}\times U\rightarrow\mathbb{R}^{n\times d}$,
$F^{m,i}:(t_{i-1},t_{i}]\times\mathbb{R}^{i\times n}\times\mathbb{R}\times\mathbb{R}^{d}\times U\rightarrow\mathbb{R}$,
$G^{m}:\mathbb{R}^{m\times n}\rightarrow\mathbb{R}$, and $H^{m,i}:(t_{i-1},t_{i}]\times\mathbb{R}^{i\times n}\times\mathbb{R}\times\mathbb{R}^{n}\times\mathbb{R}^{n\times n}\times U\rightarrow\mathbb{R}$
as follows (with $\overrightarrow{x_{i}}=(x_{1},\cdots,x_{i})\in\mathbb{R}^{i\times n}$):
\begin{align*}
B^{m,i}(t,\overrightarrow{x_{i}},u):= & b\Big(\sum_{j=1}^{i-1}x_{j}\chi_{[t_{j},t]}+x_{i}\chi_{\{t\}},u\Big),\\
\Sigma^{m,i}(t,\overrightarrow{x_{i}},u):= & \sigma\Big(\sum_{j=1}^{i-1}x_{j}\chi_{[t_{j},t]}+x_{i}\chi_{\{t\}},u\Big),\\
F^{m,i}(t,\overrightarrow{x_{i}},y,z,u):= & f\Big(\sum_{j=1}^{i-1}x_{j}\chi_{[t_{j},t]}+x_{i}\chi_{\{t\}},y,z,u\Big),\\
G^{m}(x_{1},\cdots,x_{m}):= & g\Big(\sum_{j=1}^{m-1}x_{j}\chi_{[t_{j},T]}+x_{m}\chi_{\{T\}}\Big),\\
H^{m,i}(t,\overrightarrow{x_{i}},r,p,A,u):= & \frac{1}{2}\text{Tr}\left(\Sigma^{m,i}(\Sigma^{m,i})^{T}(t,\overrightarrow{x_{i}},u)A\right)+\langle B^{m,i}(t,\overrightarrow{x_{i}},u),p\rangle\\
 & +F^{m,i}(t,\overrightarrow{x_{i}},r,(\Sigma^{m,i})^{T}(t,\overrightarrow{x_{i}},u)p,u).
\end{align*}
Obviously, for any $\gamma\in\Lambda_{T}$, $u\in U$ and $t\in(t_{i-1},t_{i}]$,
\begin{eqnarray*}
b^{m}(\gamma_{t},u) & = & B^{m,i}\big(t,\gamma_{t}(t_{1}),\gamma_{t}\big|_{t_{1}}^{t_{2}},\cdots,\gamma_{t}\big|_{t_{i-2}}^{t_{i-1}},\gamma_{t}\big|_{t_{i-1}}^{t},u\big),\\
\sigma^{m}(\gamma_{t},u) & = & \Sigma^{m,i}\big(t,\gamma(t_{1}),\gamma_{t}\big|_{t_{1}}^{t_{2}},\cdots,\gamma_{t}\big|_{t_{i-2}}^{t_{i-1}},\gamma_{t}\big|_{t_{i-1}}^{t},u\big),\\
f^{m}(\gamma_{t},y,z,u) & = & F^{m,i}\big(t,\gamma(t_{1}),\gamma_{t}\big|_{t_{1}}^{t_{2}},\cdots,\gamma_{t}\big|_{t_{i-2}}^{t_{i-1}},\gamma_{t}\big|_{t_{i-1}}^{t},y,z,u\big),\\
g^{m}(\gamma,u) & = & G^{m,i}\big(t,\gamma(t_{1}),\gamma_{t}\big|_{t_{1}}^{t_{2}},\cdots,\gamma_{t}\big|_{t_{m-2}}^{T}\big),\\
\mathcal{H}^{m}(\gamma_{t},r,p,A,u) & = & H^{m,i}\big(t,\gamma(t_{1}),\gamma_{t}\big|_{t_{1}}^{t_{2}},\cdots,\gamma_{t}\big|_{t_{i-2}}^{t_{i-1}},\gamma_{t}\big|_{t_{i-1}}^{t},r,p,A,u\big).
\end{eqnarray*}
Here $\gamma\big|_{s}^{t}:=\gamma(t)-\gamma(s)$. Furthermore, from
Assumption (H2), $B^{m,i},\sigma^{m,i},F^{m,i},$ and $G^{m}$ are
uniformly Lipschitz continuous in $(t,\overrightarrow{x_{i}},y,z,u)$,
and for $i<m$,
\begin{eqnarray*}
B^{m,i}(t_{i},\overrightarrow{x_{i}},u) & = & B^{m,i+1}(t_{i}+,\overrightarrow{x_{i-1}},\mathbf{0},x_{i},u),\\
\Sigma^{m,i}(t_{i},\overrightarrow{x_{i}},u) & = & \Sigma^{m,i+1}(t_{i}+,\overrightarrow{x_{i-1}},\mathbf{0},x_{i},u),\\
F^{m,i}(t_{i},\overrightarrow{x_{i}},y,z,u) & = & F^{m,i+1}(t_{i}+,\overrightarrow{x_{i-1}},\mathbf{0},x_{i},y,z,u),\\
H^{m,i}(t_{i},\overrightarrow{x_{i}},r,p,A,u) & = & H^{m,i+1}(t_{i}+,\overrightarrow{x_{i-1}},\mathbf{0},x_{i},r,p,A,u).
\end{eqnarray*}
Here $f(t_{i}+)$ is the right limit of the function $f$ at time
$t_{i}$.

Let $V^{m,i},i=1,\cdots m$ be the unique viscosity solutions of second
order parametrized nonlinear parabolic equations
\begin{equation}
\begin{cases}
{\displaystyle -\partial_{t}V^{m,i}(t,\overrightarrow{x_{i}})-\sup_{u\in U}H^{m,i}\left(t,\overrightarrow{x_{i}},(1,\partial_{x_{i}},\partial_{x_{i}x_{i}})V^{m,i},u\right)=0,}\\
\qquad\qquad\qquad\qquad\qquad\qquad\qquad\,\,(t,\overrightarrow{x_{i}})\in(t_{i-1},t_{i})\times\mathbb{R}^{i\times n},\, i=1,\cdots,m;\\
V^{m,i}(t_{i},\overrightarrow{x_{i}})=V^{m,i+1}(t_{i}+,\overrightarrow{x_{i-1}},\mathbf{0},x_{i}),\quad\quad i<m,\,\overrightarrow{x_{i}}\in\mathbb{R}^{i\times n};\\
V^{m,m}(T,\overrightarrow{x_{m}})=G^{m}(\overrightarrow{x_{m}}),\qquad\qquad\quad\,\overrightarrow{x_{m}}\in\mathbb{R}^{m\times n}.
\end{cases}\label{eq:first-approx-PDE}
\end{equation}
According to the the relationship between viscosity solution of Bellman
equations and the optimal control problems, we have, for any $\gamma_{t}\in\Lambda$,
\[
v^{m}(\gamma_{t})=\sum_{i=1}^{m}\chi_{[t_{i-1},t_{i})}(t)V^{m,i}(t,\gamma_{t}(t_{1}),\gamma_{t}\big|_{t_{1}}^{t_{2}},\cdots,\gamma_{t}\big|_{t_{i-2}}^{t_{i-1}},\gamma_{t}\big|_{t_{i-1}}^{t}).
\]
\medskip{}

Second, we construct the smooth approximations of $v^{m}$. For this
purpose, we mollify $\sup_{u}H^{m,i}$. Consider the following mollifier
$\varphi_{i}:\mathbb{R}^{i}\to\mathbb{R},\, i=1,2,\cdots,$
\begin{align*}
\varphi_{i}(x) & :=\begin{cases}
C_{i}\exp\big(-(1-|x|^{2})^{-1}\big), & |x|^{2}<1;\\
0, & \text{else},
\end{cases}
\end{align*}
where $C_{i}$ is the constant such that $\int\varphi_{i}=1$. Let
\[
\varphi_{\varepsilon}(t,x,r,p,A):=\varepsilon^{-(n^{2}+2n+2)}\varphi_{1}(\frac{t+1/\varepsilon}{\varepsilon})\varphi_{n}(\frac{x}{\varepsilon})\varphi_{1}(\frac{r}{\varepsilon})\varphi_{n}(\frac{p}{\varepsilon})\varphi_{n^{2}}(\frac{A}{\varepsilon}).
\]
Now we extend $H^{m,i},\, i<m,$ on the interval $t\in[t_{i},t_{i+1})$
by
\[
H^{m,i}(t,\overrightarrow{x_{i}},r,p,A,u):=H^{m,i+1}(t,\overrightarrow{x_{i-1}},\mathbf{0},x_{i},r,p,A,u),
\]
and $H^{m,m}$ on the interval $t\in[T,T+1/m)$ by
\[
H^{m,m}(t,\overrightarrow{x_{m}},r,p,A,u):=H^{m,m}(T,\overrightarrow{x_{m}},r,p,A,u),
\]
and mollify $\sup_{u}H^{m,i}$ on interval $(t_{i-1},t_{i}]$ as
\begin{equation}
\hat{H}^{m,i;\varepsilon_{1}}(\cdot,\overrightarrow{x_{i-1}},\cdot,\cdot,\cdot,\cdot):=\left(\sup_{u}H^{m,i}(\cdot,\overrightarrow{x_{i-1}},\cdot,\cdot,\cdot,\cdot,u)\right)*\varphi_{\varepsilon_{1}}(\cdot),\label{eq:mollif}
\end{equation}
where $*$ is the convolution operator in $(t,x_{i},r,p,A)$ and $\varepsilon_{1}<m^{-1}$.
Obviously, $\hat{H}^{m,i;\varepsilon_{1}}$ is differentiable in $(t,x_{i},r,p,A)$
and Lipschitz continuous in $\overrightarrow{x_{i-1}}$. Noting the
structure condition, we obtain that
\begin{equation}
|\hat{H}^{m,i;\varepsilon_{1}}(t,\overrightarrow{x_{i}},r,p,A)-\sup_{u}H^{m,i}(t,\overrightarrow{x_{i}},r,p,A,u)|\le C(1+|p|+|A|)\varepsilon_{1},\label{eq:Hmi-second-approx}
\end{equation}
and $\hat{H}^{m,i;\varepsilon}$ satisfy the following structure conditions:
\begin{align}
 & |\partial_{t}\hat{H}^{m,i;\varepsilon}|+|\partial_{x_{i}}\hat{H}^{m,i;\varepsilon}|\le C(1+|p|+|A|),\nonumber \\
 & |\partial_{p}\hat{H}^{m,i;\varepsilon}|+|\hat{H}^{m,i;\varepsilon}(t,\overrightarrow{x_{i}},r,p,\mathbf{0})|\le C,\nonumber \\
 & C^{-1}I_{n}\le\partial_{A}\hat{H}^{m,i;\varepsilon}\le CI_{n},\quad\partial_{r}\hat{H}^{m,i;\varepsilon}\le-C,\label{eq:structure-condition}\\
 & \hat{H}^{m,i;\varepsilon}\left(t,\overrightarrow{x_{i}},r,p,A\right)\text{ is convex in }A.\nonumber
\end{align}
Similarly, we define $G_{\varepsilon}^{m}(\overrightarrow{x_{m-1}},\cdot):=G^{m}(\overrightarrow{x_{m-1}},\cdot)*\varphi_{\varepsilon}^{n}$.

Consider the following system of fully nonlinear parameterized state-dependent
PDE
\begin{equation}
\begin{cases}
-\partial_{t}V_{\varepsilon}^{m,i;\varepsilon_{1}}(t,\overrightarrow{x_{i}})-\hat{H}^{m,i;\varepsilon_{1}}(t,\overrightarrow{x_{i}},(1,\partial_{x_{i}},\partial_{x_{i}x_{i}})V_{\varepsilon}^{m,i;\varepsilon_{1}})=0,\\
\qquad\qquad\qquad\qquad\qquad\qquad\qquad\qquad(t,\overrightarrow{x_{i}})\in(t_{i-1},t_{i})\times\mathbb{R}^{i\times n};\\
V_{\varepsilon}^{m,i;\varepsilon_{1}}(t_{i},\overrightarrow{x_{i}})=V_{\varepsilon}^{m,i+1;\varepsilon_{1}}(t_{i}+,\overrightarrow{x_{i-1}},\mathbf{0},x_{i}),\qquad\qquad i<m,\,\overrightarrow{x_{i}}\in\mathbb{R}^{i\times n};\\
V_{\varepsilon}^{m,m;\varepsilon_{1}}(T,x_{1},\cdots,x_{m})=G_{\varepsilon}^{m}(x_{1},\cdots,x_{m}),\qquad\,\,(x_{1},\cdots,x_{m})\in\mathbb{R}^{m\times n}.
\end{cases}\label{eq:second-approx-PDE}
\end{equation}
Define
\[
G_{0}:=\sup_{m,\overrightarrow{x_{m}}}|G_{\varepsilon}^{m}(\overrightarrow{x_{m}})|,\quad H_{0}:=\sup_{m,i,t,\overrightarrow{x_{i}}}|\hat{H}^{m,i;\varepsilon_{1}}(t,\overrightarrow{x_{i}},0,\mathbf{0},\mathbf{0})|.
\]
We have the following key lemma.

\begin{lem} \label{lem:estimate-approx-St-Depen-equation-nindegen}
Assume (H2). Then the system \eqref{eq:first-approx-PDE} has unique
viscosity solutions $\{V_{\varepsilon}^{m,i;\varepsilon_{1}}\}_{i=1}^{m}$.
Moreover, there is some positive constants $C$ which are independent
of $m$, $i$ and $\varepsilon_{1}\,(\varepsilon_{1}<m^{-1})$, such
that:

(1) $V_{\varepsilon}^{m,i;\varepsilon_{1}}(\cdot,\overrightarrow{x_{i-1}},\cdot)\in\mathscr{C}^{1,2}([t_{i-1},t_{i})\times\mathbb{R}^{n})$,
and for any $t\in[t_{i-1},t_{i})$ and $\overrightarrow{x_{i}}\in\mathbb{R}^{i\times n}$,
\begin{equation}
|V_{\varepsilon}^{m,i;\varepsilon_{1}}(t,\overrightarrow{x_{i}})|\leq G_{0}e^{-K(T-t)}+(1-e^{-K(T-t)})H_{0}K^{-1},\label{eq:est-in-lemma}
\end{equation}

\begin{equation}
|\partial_{t}V_{\varepsilon}^{m,i;\varepsilon_{1}}(t,\overrightarrow{x_{i}})|+|\partial_{x_{i}}V_{\varepsilon}^{m,i;\varepsilon_{1}}(t,\overrightarrow{x_{i}})|+|\partial_{x_{i}x_{i}}V_{\varepsilon}^{m,i;\varepsilon_{1}}(t,\overrightarrow{x_{i}})|\le C(\varepsilon),\label{eq:driv-estimate-in-lemma}
\end{equation}
where $K$ is the constant in \eqref{eq:Monotonicity-of-h}.

(2) H\"older continuity: for any $t\in[t_{i-1},t_{i})$, $\overrightarrow{x_{i}},\,\overrightarrow{y_{i}}\in\mathbb{R}^{i\times n}$,
and $s\in[t_{j-1},t_{j})$, $i\leq j\leq m$,
\begin{align}
 & \left|(1,\partial_{x_{i}},\partial_{x_{i}x_{i}},\partial_{t})V_{\varepsilon}^{m,i;\varepsilon_{1}}(t,\cdot)\big|_{\overrightarrow{y_{i}}}^{\overrightarrow{x_{i}}}\right|\le C(\varepsilon)\max_{1\le k\le i}|(x_{1}-y_{1})+\cdots+(x_{k}-y_{k})|^{\beta},\label{eq:lipsch-cont-in-lemma-space}\\
 & \left|(1,\partial_{x_{i}},\partial_{x_{i}x_{i}},\partial_{t})V_{\varepsilon}^{m,i;\varepsilon_{1}}(t,\overrightarrow{x_{i}})-(1,\partial_{x_{i}},\partial_{x_{i}x_{i}},\partial_{t})V_{\varepsilon}^{m,j;\varepsilon_{1}}(s,\overrightarrow{x_{i-1}},\underbrace{\mathbf{0},\cdots,\mathbf{0}}_{j-i},x_{i})\right|\label{eq:lipsch-cont-in-lemma-time}\\
\le & \, C(\varepsilon)|s-t|^{\frac{\beta}{2}},\nonumber
\end{align}

(3) smoothly approximating rate:
\begin{equation}
|V_{\varepsilon}^{m,i;\varepsilon_{1}}-V^{m,i}|\le C(\varepsilon_{1}+\varepsilon).\label{eq:solution-second-approx-pde-rate}
\end{equation}
\end{lem}

\begin{proof} Firstly we prove the existence of viscosity solution.
Define
\begin{align*}
H_{T} & :=\sup_{\overrightarrow{x_{m}}\in\mathbb{R}^{m\times n}}\hat{H}^{m,m;\varepsilon_{1}}(T,\overrightarrow{x_{m}},G_{\varepsilon}^{m},\partial_{x_{m}}G_{\varepsilon}^{m},\partial_{x_{m}x_{m}}G_{\varepsilon}^{m}).
\end{align*}
It is easy to verify that
\[
G_{\varepsilon}^{m}+(1-e^{-K(T-t)})H_{T}K^{-1},t\in[0,T]\text{ and }G_{\varepsilon}^{m}-(1-e^{-K(T-t)})H_{T}K^{-1},t\in[0,T]
\]
are respectively viscosity super- and sub-solutions of system \eqref{eq:second-approx-PDE}
on interval $[t_{m-1},T]$, where $K$ is the constant in \eqref{eq:Monotonicity-of-h}.
By Perron's method and comparison principle (see Crandall, Ishii and
Lions \cite{crandall1992user}), system \eqref{eq:second-approx-PDE}
has unique viscosity solutions $V_{\varepsilon}^{m,m;\varepsilon_{1}}$
on $[t_{m-1},T]$. If assertion (1) holds on $[t_{m-1},T]$, then
$V_{\varepsilon}^{m,m}(t_{m-1}+,\overrightarrow{x_{m-2}},0,x_{m-1})$,
the terminal value of $V_{\varepsilon}^{m,m-1;\varepsilon_{1}}$,
is bounded and twice differentiable in $x_{m-1}\in\mathbb{R}^{n}$.
Similarly we have the existence of the unique viscosity solution $V_{\varepsilon}^{m,i;\varepsilon_{1}}$
on $[t_{i-1},t_{i}]$ ($i=m-1,m-2,\cdots,1$) recursively.

In view of Wang \cite[Theorems 1.1 and 1.3]{wang1992regularity},
using the interior $\mathscr{C}^{1,\alpha}$ and $\mathscr{C}^{2,\alpha}$
estimates for the equations of the structure conditions \eqref{eq:structure-condition}
(see Lieberman \cite[Chapter 14, Sections 2-4]{lieberman2005second}),
we have that $V_{\varepsilon}^{m,i;\varepsilon_{1}}(\cdot,\overrightarrow{x_{i-1}},\cdot)\in\mathscr{C}^{1,2}([t_{i-1},t_{i})\times\mathbb{R}^{n})$.
Since
\[
G_{0}e^{-K(T-t)}+(1-e^{-K(T-t)})H_{0}K^{-1},\, t\in[0,T],
\]
and
\[
-G_{0}e^{-K(T-t)}-(1-e^{-K(T-t)})H_{0}K^{-1},\, t\in[0,T],
\]
are viscosity super- and sub-solutions of system \eqref{eq:second-approx-PDE},
respectively, we have \eqref{eq:est-in-lemma}. Then combining the
interior $\mathscr{C}^{2,\alpha}$ estimates and interpolation inequality,
we have \eqref{eq:driv-estimate-in-lemma}. Assertion (1) is proved.

Noting that $\bar{V}:=V_{\varepsilon}^{m,i;\varepsilon_{1}}(\cdot,\overrightarrow{x_{i-1}},\cdot)-V_{\varepsilon}^{m,i;\varepsilon_{1}}(\cdot,\overrightarrow{y_{i-1}},\cdot)$
is the solution of
\[
-\partial_{t}\bar{V}-\text{Tr}(a\partial_{x_{i}x_{i}}\bar{V})-\langle b,\partial_{x_{i}}\bar{V}\rangle-c\bar{V}-h_{0}=0,
\]
where
\begin{align*}
a:= & \int_{0}^{1}\partial_{A}\hat{H}^{m,i;\varepsilon_{1}}\Big(t,\overrightarrow{x_{i-1}},x_{i},(1,\partial_{x_{i}})V_{\varepsilon}^{m,i;\varepsilon_{1}}(t,\overrightarrow{x_{i-1}},x_{i}),\\
 & \partial_{x_{i}x_{i}}V_{\varepsilon}^{m,i;\varepsilon_{1}}(t,\overrightarrow{y_{i-1}},x_{i})+\theta\partial_{x_{i}x_{i}}V_{\varepsilon}^{m,i;\varepsilon_{1}}(t,\cdot,x_{i})\Big|_{\overrightarrow{y_{i-1}}}^{\overrightarrow{x_{i-1}}}\Big)d\theta,\\
b:= & \int_{0}^{1}\partial_{p}\hat{H}^{m,i;\varepsilon_{1}}\Big(t,\overrightarrow{x_{i-1}},x_{i},V_{\varepsilon}^{m,i;\varepsilon_{1}}(t,\overrightarrow{x_{i-1}},x_{i}),\partial_{x_{i}}V_{\varepsilon}^{m,i;\varepsilon_{1}}(t,\overrightarrow{y_{i-1}},x_{i})\\
 & +\theta\partial_{x_{i}}V_{\varepsilon}^{m,i;\varepsilon_{1}}(t,\cdot,x_{i})\Big|_{\overrightarrow{y_{i-1}}}^{\overrightarrow{x_{i-1}}},\partial_{x_{i}x_{i}}V_{\varepsilon}^{m,i;\varepsilon_{1}}(t,\overrightarrow{y_{i-1}},x_{i})\Big)d\theta,\\
c:= & \int_{0}^{1}\partial_{r}\hat{H}^{m,i;\varepsilon_{1}}\Big(t,\overrightarrow{x_{i-1}},x_{i},V_{\varepsilon}^{m,i;\varepsilon_{1}}(t,\overrightarrow{y_{i-1}},x_{i})+\theta V_{\varepsilon}^{m,i;\varepsilon_{1}}(t,\cdot,x_{i})\Big|_{\overrightarrow{y_{i-1}}}^{\overrightarrow{x_{i-1}}},\\
 & \qquad(\partial_{x_{i}},\partial_{x_{i}x_{i}})V_{\varepsilon}^{m,i;\varepsilon_{1}}(t,\overrightarrow{y_{i-1}},x_{i})\Big)d\theta,\\
h_{0}:= & \hat{H}_{\varepsilon}^{m,i;\varepsilon_{1}}\left(t,\cdot,x_{i},(1,\partial_{x_{i}},\partial_{x_{i}x_{i}})V_{\varepsilon}^{m,i;\varepsilon_{1}}(t,\overrightarrow{y_{i-1}},x_{i})\right)\Big|_{\overrightarrow{y_{i-1}}}^{\overrightarrow{x_{i-1}}}.
\end{align*}
In view of Assertion (1), we know $|a|+|b|+|c|<C,$ and
\[
|h_{0}|\leq C\max_{1\leq k\leq i-1}|(x_{1}-y_{1})+\cdots+(x_{k}-y_{k})|.
\]
Similar to recursive method in Assertion (1), we have
\begin{align*}
 & |\bar{V}(t,x_{i})|=|V_{\varepsilon}^{m,i;\varepsilon_{1}}(t,\overrightarrow{x_{i}})-V_{\varepsilon}^{m,i;\varepsilon_{1}}(t,\overrightarrow{y_{i-1}},x_{i})|\\
\le & e^{-K(T-t)}G_{\overrightarrow{x_{i-1}},\overrightarrow{y_{i-1}}}+(1-e^{-K(T-t)})K^{-1}L_{\overrightarrow{x_{i-1}},\overrightarrow{y_{i-1}}}\\
\le & C\left(1\wedge\max_{1\leq k\leq i-1}|(x_{1}-y_{1})+\cdots+(x_{k}-y_{k})|\right),
\end{align*}
where
\begin{align*}
G_{\overrightarrow{x_{i-1}},\overrightarrow{y_{i-1}}}:= & \sup_{x_{i},\cdots,x_{m}}\left|\, G_{\varepsilon}^{m}(z,x_{i},\cdots,x_{m})\Big|_{z=\overrightarrow{y_{i-1}}}^{z=\overrightarrow{x_{i-1}}}\,\right|\\
\le & \, C\left(1\wedge\max_{1\leq k\leq i-1}|(x_{1}-y_{1})+\cdots+(x_{k}-y_{k})|\right)
\end{align*}
and
\begin{align*}
L_{\overrightarrow{x_{i-1}},\overrightarrow{y_{i-1}}}:= & \sup_{k\geq i,x_{i+1},\cdots x_{m}}\left|\,\hat{H}^{m,k;\varepsilon_{1}}(t,z,x_{i},\cdots,x_{k},0,\mathbf{0},\mathbf{0})\Big|_{z=\overrightarrow{y_{i-1}}}^{z=\overrightarrow{x_{i-1}}}\,\right|\\
\le & \, C(\varepsilon)\left(1\wedge\max_{1\leq k\leq i-1}|(x_{1}-y_{1})+\cdots+(x_{k}-y_{k})|\right).
\end{align*}
In view of the interior Schauder estimate for linear parabolic equation
(see Lieberman \cite[Theorem 4.9]{lieberman2005second}), we have
\begin{align*}
 & \left|(1,\partial_{x_{i}},\partial_{x_{i}x_{i}},\partial_{t})V_{\varepsilon}^{m,i;\varepsilon_{1}}(t,\cdot,x_{i})\big|_{\overrightarrow{y_{i-1}}}^{\overrightarrow{x_{i-1}}}\right|=\left|(1,\partial_{x_{i}},\partial_{x_{i}x_{i}},\partial_{t})\bar{V}(t,x_{i})\right|\\
\le & C(\varepsilon)\left(1\wedge\max_{1\leq k\leq i-1}|(x_{1}-y_{1})+\cdots+(x_{k}-y_{k})|\right).
\end{align*}
Incorporating the following interior $\mathscr{C}^{2,\alpha}$ estimate
of $V_{\varepsilon}^{m,i;\varepsilon_{1}}$:
\[
\left|(1,\partial_{x_{i}},\partial_{x_{i}x_{i}},\partial_{t})V_{\varepsilon}^{m,i;\varepsilon_{1}}(t,\overrightarrow{x_{i-1}},\cdot)\big|_{y_{i}}^{x_{i}}\right|\le C(\varepsilon)|x_{i}-y_{i}|^{\beta},
\]
we have \eqref{eq:lipsch-cont-in-lemma-space}.

For any $t\in[t_{i-1},T)$ and $\overrightarrow{x_{i}}\in\mathbb{R}^{i\times n}$,
define
\begin{align*}
\bar{\bar{V}}(t,\overrightarrow{x_{i}}) & :=\sum_{k=i}^{m}\chi_{[t_{k-1},t_{k})}(t)V_{\varepsilon}^{m,k;\varepsilon_{1}}(t,\overrightarrow{x_{i-1}},\underbrace{\mathbf{0},\cdots,\mathbf{0}}_{k-i},x_{i}),\\
\bar{\bar{H}}(t,\overrightarrow{x_{i}},r,p,A) & :=\sum_{k=i}^{m}\chi_{[t_{k-1},t_{k})}(t)\hat{H}^{m,k;\varepsilon_{1}}(t,\overrightarrow{x_{i-1}},\underbrace{\mathbf{0},\cdots,\mathbf{0}}_{k-i},x_{i},r,p,A).
\end{align*}
Obviously, $\bar{\bar{H}}$ is smooth and satisfies structure condition
\eqref{eq:structure-condition}, and $\bar{\bar{V}}$ is the classical
solution of
\[
\begin{cases}
\partial_{t}\bar{\bar{V}}+\bar{\bar{H}}(t,\overrightarrow{x_{i}},\bar{\bar{V}},\partial_{x_{i}}\bar{\bar{V}},\partial_{x_{i}x_{i}}\bar{\bar{V}})=0, & t\in[t_{i-1},T);\\
\bar{\bar{V}}(T,\overrightarrow{x_{i}})=G_{\varepsilon}^{m}(\overrightarrow{x_{i-1}},\underbrace{\mathbf{0},\cdots,\mathbf{0}}_{m-i},x_{i}).
\end{cases}
\]
By the Schauder interior estimate, we have for any $t\in[t_{i-1},t_{i})$,
$s\in[t_{j-1},t_{j})$, $i\leq j\leq m$, and $\overrightarrow{x_{i}}\in\mathbb{R}^{i\times n}$,
\begin{align*}
 & \Big|(1,\partial_{x_{i}},\partial_{x_{i}x_{i}},\partial_{t})V_{\varepsilon}^{m,i;\varepsilon_{1}}(t,\overrightarrow{x_{i}})-(1,\partial_{x_{i}},\partial_{x_{i}x_{i}},\partial_{t})V_{\varepsilon}^{m,j;\varepsilon_{1}}(s,\overrightarrow{x_{i-1}},\underbrace{\mathbf{0},\cdots,\mathbf{0}}_{j-i},x_{i})\Big|\\
= & \Big|\bar{\bar{V}}(t,\overrightarrow{x_{i}})-\bar{\bar{V}}(s,\overrightarrow{x_{i}})\Big|\leq C(\varepsilon)|t-s|^{\frac{\beta}{2}}.
\end{align*}
We proved \eqref{eq:lipsch-cont-in-lemma-time}.

It remains to show Assertion (3). Form \eqref{eq:driv-estimate-in-lemma}
and \eqref{eq:Hmi-second-approx}, we see that, when $C$ is sufficiently
large, $V_{\varepsilon}^{m,i;\varepsilon_{1}}+C(\varepsilon_{1}+\varepsilon)$
is a viscosity super-solution and $V_{\varepsilon}^{m,i;\varepsilon_{1}}-C(\varepsilon_{1}+\varepsilon)$
a viscosity sub-solution of equation \eqref{eq:first-approx-PDE},
which imply \eqref{eq:solution-second-approx-pde-rate} by the comparison
principle. \end{proof}

\begin{rem} Note that the constant in estimate \eqref{eq:est-in-lemma}
does not depend on $m$, which allows us to conclude that those constants
in the estimates \eqref{eq:driv-estimate-in-lemma} - \eqref{eq:lipsch-cont-in-lemma-time}
do not depend on $m$.\end{rem}

\medskip{}

Define the smooth approximating functional
\begin{equation}
v_{\varepsilon}^{m;\varepsilon_{1}}(\gamma_{t}):=\sum_{i=1}^{m}\chi_{[t_{i-1},t_{i})}(t)V_{\varepsilon}^{m,i;\varepsilon_{1}}\left(t,\gamma_{t}(t_{1}),\gamma_{t}\big|_{t_{1}}^{t_{2}},\cdots,\gamma_{t}\big|_{t_{i-2}}^{t_{i-1}},\gamma_{t}\big|_{t_{i-1}}^{t}\right),\quad\gamma_{t}\in\Lambda.\label{eq:-second-approx-path}
\end{equation}
Then, $v_{\varepsilon}^{m;\varepsilon_{1}}\in\mathscr{C}^{1,2}(\Lambda)$.
In fact, \eqref{eq:-second-approx-path} is well-defined as well for
$\gamma\in\hat{\Lambda}$. It is obvious that $v_{\varepsilon}^{m;\varepsilon_{1}}\in\mathscr{C}^{1,2}(\hat{\Lambda})$,
which implies by definition that the restriction of $v_{\varepsilon}^{m;\varepsilon_{1}}$
on $\Lambda$ lies in $\mathscr{C}^{1,2}(\Lambda)$.

Define
\[
\hat{\mathcal{H}}^{m;\varepsilon_{1}}(\gamma_{t},r,p,A):=\sum_{i=1}^{m}\chi_{[t_{i-1},t_{i})}(t)\hat{H}^{m,i;\varepsilon_{1}}(t,\gamma_{t}(t_{1}),\gamma_{t}\big|_{t_{1}}^{t_{2}},\cdots,\gamma_{t}\big|_{t_{i-2}}^{t_{i-1}},\gamma_{t}\big|_{t_{i-1}}^{t},r,p,A).
\]
Obviously, $v_{\varepsilon}^{m;\varepsilon_{1}}$ is the classical
solution of the following path-dependent PDE
\begin{equation}
\begin{cases}
-D_{t}v_{\varepsilon}^{m;\varepsilon_{1}}-\hat{\mathcal{H}}^{m;\varepsilon_{1}}(\gamma_{t},v_{\varepsilon}^{m;\varepsilon_{1}},D_{x}v_{\varepsilon}^{m;\varepsilon_{1}},D_{xx}v_{\varepsilon}^{m;\varepsilon_{1}})=0, & \forall\gamma_{t}\in\Lambda,0<t<T;\\
v_{\varepsilon}^{m;\varepsilon_{1}}(\gamma_{T})=G_{\varepsilon}^{m}(\gamma_{T}), & \forall\gamma_{T}\in\Lambda_{T}.
\end{cases}\label{eq:second-approx-path-Equation}
\end{equation}
Moreover, we have the path version of Lemma \ref{lem:estimate-approx-St-Depen-equation-nindegen}.

\begin{prop} \label{prop:estimate-approx-equation-nondegen} Let
(H2) hold. There are some positive constants $C(\varepsilon)$ (independent
of $m$ and $\varepsilon_{1}\,(\varepsilon_{1}<m^{-1})$), such that
for all $\gamma_{t},\bar{\gamma}_{\bar{t}}\in\Lambda,\, t,\bar{t}<T$,
we have:

(1) $\mathscr{C}^{1+\frac{\beta}{2},2+\beta}$ boundedness:
\begin{equation}
\big|v_{\varepsilon}^{m;\varepsilon_{1}}(\gamma_{t})\big|_{2,\beta;\Lambda}\le C(\varepsilon).\label{eq:solution-second-approx-drivetive-estimat}
\end{equation}

(2) smoothly approximating rate:
\begin{equation}
|v_{\varepsilon}^{m;\varepsilon_{1}}(\gamma_{t})-\tilde{v}(\gamma_{t})|\le C(\text{Osc}(\gamma_{t},m^{-1})+m^{-\frac{1}{2}}+\varepsilon+\varepsilon_{1}).\label{eq:solution-second-approx-Path-rate}
\end{equation}
\end{prop} \begin{proof} Assertion (1) is immediate consequence
of Assertions (1) and (2) of Lemma \ref{lem:estimate-approx-St-Depen-equation-nindegen}.
Assertion (2) follows from \eqref{eq:first-approx-path-rate} and
\eqref{eq:solution-second-approx-pde-rate}. \end{proof}\medskip{}

At the end of this subsection, we introduce the following auxiliary
path functional
\[
\tilde{v}_{0}(\gamma_{t}):=E\left[\|W_{T}^{\gamma_{t}}\|_{0}\right],\quad\gamma_{t}\in\Lambda,
\]
and the smooth approximating functional
\begin{equation}
v_{0,\varepsilon}^{m}(\gamma_{t}):=E[g_{0,\varepsilon}^{m}(W_{T}^{\gamma_{t}})],\quad\gamma_{t}\in\Lambda.\label{eq:smooth-func-linear-grow}
\end{equation}
Here,
\[
W^{\gamma_{t}}(s):=\gamma_{t}(s)\chi_{[0,t)}(s)+\big(W(s)-W(t)+\gamma_{t}(t)\big)\chi_{[t,T]}(s),\quad s\in[0,T],
\]
and
\[
g_{0,\varepsilon}^{m}(\gamma_{t}):=G_{0,\varepsilon}^{m}\left(\gamma_{T}(t_{1}),\gamma_{T}\big|_{t_{1}}^{t_{2}},\cdots,\gamma_{T}\big|_{t_{i-2}}^{t_{i-1}},\gamma_{T}\big|_{t_{i-1}}^{t}\right),
\]

\[
G_{0,\varepsilon}^{m}(\overrightarrow{x_{m-1}},\cdot):=\left(\left(\max_{1\leq k\leq m-1}|x_{1}+\cdots+x_{k}|\right)\vee|x_{1}+\cdots+x_{m-1}+\cdot|\right)*\varphi_{n,\varepsilon},
\]
$\varphi_{n,\varepsilon}$ is a mollifier in $\mathbb{R}^{n}$. Obviously
\[
V_{0,\varepsilon}^{m,i}(t,\overrightarrow{x_{i}}):=v_{0,\varepsilon}^{m}\big(\sum_{j=1}^{i-1}x_{j}\chi_{[t_{j},t]}+x_{i}\chi_{\{t\}}\big),\,(t,\overrightarrow{x_{i}})\in[t_{i-1},t_{i}]\times\mathbb{R}^{i\times n},i=1,\cdots,m,
\]
are the classical solutions of
\[
\begin{cases}
\partial_{t}V_{0,\varepsilon}^{m,i}(t,\overrightarrow{x_{i}})+\frac{1}{2}\Delta_{x_{i}x_{i}}V_{0,\varepsilon}^{m,i}(t,\overrightarrow{x_{i}})=0, & (t,\overrightarrow{x_{i}})\in(t_{i-1},t_{i})\times\mathbb{R}^{i\times n};\\
V_{0,\varepsilon}^{m,i}(t_{i},\overrightarrow{x_{i}})=V_{0,\varepsilon}^{m,i+1}(t_{i}+,\overrightarrow{x_{i-1}},\mathbf{0},x_{i}), & \overrightarrow{x_{i}}\in\mathbb{R}^{i\times n},i=1,\cdots m-1;\\
V_{0,\varepsilon}^{m,i}(T,\overrightarrow{x_{m}})=G_{0,\varepsilon}^{m}(\overrightarrow{x_{m}}), & \overrightarrow{x_{m}}\in\mathbb{R}^{m\times n}.
\end{cases}
\]
Similarly as in Proposition \ref{prop:estimate-approx-equation-nondegen},
$v_{0,\varepsilon}^{m}\in\mathscr{C}^{1,2}(\Lambda)$ satisfies the
following estimates:
\begin{align}
\big|v_{0,\varepsilon}^{m}\big|_{2,\beta;\mathbf{Q}_{M_{0},T}} & \le C(\varepsilon)(1+M_{0}),\label{eq:bounded-estimat-v}\\
|v_{0,\varepsilon}^{m}(\gamma_{t})-\tilde{v}_{0}(\gamma_{t})| & \le C\left(\text{Osc}(\gamma_{t},m^{-1})+m^{-\frac{1}{2}}+\varepsilon\right).\label{eq:approx-rate-v}
\end{align}

\subsubsection{Proof of Theorem \ref{thm:(repres-theorem-nondegen}}

Let $v\in\mathscr{C}_{b}(\Lambda)\cap\mathscr{C}_{u}(\Lambda)$ satisfying
\eqref{eq:regu-of-u} be a viscosity solution to the path-dependent
Bellman equation \eqref{eq:PHJBE}, and $\tilde{v}$ be defined by
\eqref{eq:defination Of value fun}. From Remark \ref{rem:bound-regul-tilde-u}
we know $\tilde{v}\in\mathscr{C}_{b}(\Lambda)\cap\mathscr{C}_{u}(\Lambda)$.
It is sufficient to show $\tilde{v}\geq v$ since the inverse inequality
can be proved in a similar way.

Otherwise, we have
\[
\inf_{\gamma_{t}\in\Lambda}\big(\tilde{v}(\gamma_{t})-v(\gamma_{t})\big):=-r_{0}<0.
\]
 For sufficiently large $\lambda>1$, we have
\[
\inf(\tilde{v}+e^{-\lambda(t+1)}\tilde{v}_{0}-v)<-\frac{7}{8}r_{0}.
\]
We fix $\lambda$. Since $\cup_{\mu>0}\mathbf{C}_{\mu}^{\alpha}$
is dense in $\Lambda$, then for sufficiently large number $\mu$,
\[
\inf_{\gamma_{t}\in\mathbf{C}_{\mu}^{\alpha}}\left[\tilde{v}(\gamma_{t})+e^{-\lambda(t+1)}\tilde{v}_{0}(\gamma_{t})-v(\gamma_{t})\right]<-\frac{3}{4}r_{0}.
\]
Besides, from the definition of $\tilde{v}_{0}$ we know $\tilde{v}_{0}(\gamma_{t})\geq\|\gamma_{t}\|_{0}$,
and noting that $v,\,\tilde{v}\in\mathscr{C}_{b}(\Lambda)$, thus
there is $M_{0}=M_{0}(\lambda)>0$ such that
\begin{equation}
\tilde{v}+e^{-\lambda(t+1)}\tilde{v}_{0}-v>0,\qquad\text{on }\Lambda\backslash\mathbf{Q}_{M_{0},T}.\label{eq:decay in proof}
\end{equation}
Now, we fix $\mu$ and $M_{0}$ firstly. From \eqref{eq:solution-second-approx-Path-rate}
and \eqref{eq:approx-rate-v}, we have for all $\varepsilon<\frac{1}{32}r_{0}C$,
$m>\overline{m}:=(32C\mu r_{0}^{-1})^{\frac{1}{\alpha}}\vee[(32C)^{2}r_{0}^{-2}]$
and $\varepsilon_{1}<\overline{\varepsilon_{1}}:=(\frac{1}{32}r_{0}C^{-1})\wedge m^{-1}$,
for any $\gamma_{t}\in\mathbf{C}_{\mu}^{\alpha}$,
\begin{gather*}
|(v_{\varepsilon}^{m;\varepsilon_{1}}-\tilde{v})(\gamma_{t})|\le C(\text{Osc}(\gamma_{t},m^{-1})+m^{-\frac{1}{2}}+\varepsilon_{1}+\varepsilon)\le C(\mu m^{-\alpha}+m^{-\frac{1}{2}}+\varepsilon_{1}+\varepsilon)<\frac{1}{8}r_{0},\\
|(v_{0,\varepsilon}^{m}-\tilde{v}_{0})(\gamma_{t})|\le C(\text{Osc}(\gamma_{t},m^{-1})+m^{-\frac{1}{2}}+\varepsilon)\le C(\mu m^{-\alpha}+m^{-\frac{1}{2}}+\varepsilon)<\frac{1}{8}r_{0}.
\end{gather*}
Hence
\begin{gather}
\inf_{\gamma_{t}\in\mathbf{C}_{\mu}^{\alpha}}\big(v_{\varepsilon}^{m;\varepsilon_{1}}-e^{-\lambda(t+1)}v_{0,\varepsilon}^{m}-v\big)(\gamma_{t})<-\frac{1}{2}r_{0},\label{eq:maximum-degen}\\
(v_{\varepsilon}^{m;\varepsilon_{1}}+e^{-\lambda(t+1)}v_{0,\varepsilon}^{m}-v)(\gamma_{t})>-\frac{1}{4}r_{0},\quad\forall\gamma_{t}\in\mathbf{C}_{\mu}^{\alpha}\backslash\mathbf{Q}_{M_{0},T},\label{eq:decay-infty}
\end{gather}
and
\begin{align}
 & (v_{\varepsilon}^{m;\varepsilon_{1}}+e^{-\lambda(T+1)}v_{0,\varepsilon}^{m}-v)(\gamma_{T})\label{eq:decay-terminal}\\
> & (\tilde{v}+e^{-\lambda(T+1)}\tilde{v}_{0}-v)(\gamma_{T})-\frac{1}{4}r_{0}>-\frac{1}{4}r_{0},\quad\forall\gamma_{T}\in\mathbf{C}_{\mu}^{\alpha}.\nonumber
\end{align}
Since $v\in\mathscr{C}_{u}(\Lambda)$, $v_{\varepsilon}^{m;\varepsilon_{1}}$
satisfies \eqref{eq:solution-second-approx-drivetive-estimat} and
$v_{0,\varepsilon}^{m}$ satisfies \eqref{eq:bounded-estimat-v} uniformly
w.r.t. all $m$ and $\varepsilon_{1}$, therefore, there is a constant
$\kappa_{1}=\kappa_{1}(\varepsilon,\lambda,M_{0})\in(0,T)$ such that
for any $\gamma_{t}\in\mathbf{C}_{\mu}^{\alpha},\, t>T-\kappa_{1}$,
\begin{equation}
\left|(v_{\varepsilon}^{m,\varepsilon_{1}}+e^{-\lambda(t+1)}v_{0,\varepsilon}^{m}-v)(\gamma_{t})-(v_{\varepsilon}^{m,\varepsilon_{1}}+e^{-\lambda(T+1)}v_{0,\varepsilon}^{m}-v)(\gamma_{t,T})\right|\leq\frac{1}{4}r_{0}.\label{eq:terminal-cont}
\end{equation}
Combining \eqref{eq:decay-terminal}, we have
\begin{equation}
|(v_{\varepsilon}^{m,\varepsilon_{1}}+e^{-\lambda(t+1)}v_{0,\varepsilon}^{m}-v)(\gamma_{t})|\leq\frac{1}{2}r_{0},\quad\forall\gamma_{t}\in\mathbf{C}_{\mu}^{\alpha},t>T-\kappa_{1}.\label{eq:decay-near-termin}
\end{equation}
This together with \eqref{eq:maximum-degen} and \eqref{eq:decay-infty}
yield that there is $\bar{\gamma}_{\bar{t}}\in\mathbf{C}_{\mu}^{\alpha}\cap\mathbf{Q}_{M_{0},T-\iota}$
where the functional $v_{\varepsilon}^{m;\varepsilon_{1}}+e^{-\lambda(t+1)}v_{0,\varepsilon}^{m}-v$
is minimized over $\mathbf{C}_{\mu}^{\alpha}$.

Define
\begin{equation}
\psi(\gamma_{t}):=v_{\varepsilon}^{m;\varepsilon_{1}}(\gamma_{t})+e^{-\lambda(t+1)}v_{0,\varepsilon}^{m}(\gamma_{t})-(v_{\varepsilon}^{m;\varepsilon_{1}}+e^{-\lambda(\bar{t}+1)}v_{0,\varepsilon}^{m}-v)(\bar{\gamma}_{\bar{t}}),\quad\gamma_{t}\in\Lambda.\label{eq:sup-jet}
\end{equation}
By \eqref{eq:solution-second-approx-drivetive-estimat} and \eqref{eq:bounded-estimat-v},
there is $\kappa=\kappa(M_{0},\varepsilon,\lambda)<\kappa_{1}$ such
that $\psi\in\mathcal{J}_{\mu,\kappa}^{+}(\bar{\gamma}_{\bar{t}},v)$
for all $m>\overline{m}$ and small $\varepsilon_{1}>\overline{\varepsilon_{1}}$.
Consider the following estimates:
\begin{align}
 & -D_{t}\psi(\bar{\gamma}_{\bar{t}})-\sup_{u\in U}\mathcal{H}(\bar{\gamma}_{\bar{t}},\psi(\bar{\gamma}_{\bar{t}}),D_{x}\psi(\bar{\gamma}_{\bar{t}}),D_{xx}\psi(\bar{\gamma}_{\bar{t}}),u)\label{eq:contradict-th-uniq}\\
= & -D_{t}(v_{\varepsilon}^{m;\varepsilon_{1}}+e^{-\lambda(\bar{t}+1)}v_{0,\varepsilon}^{m})(\bar{\gamma}_{\bar{t}})\nonumber \\
 & -\sup_{u\in U}\mathcal{H}\left(\bar{\gamma}_{\bar{t}},v(\bar{\gamma}_{\bar{t}}),(D_{x},D_{xx})(v_{\varepsilon}^{m;\varepsilon_{1}}+e^{-(\lambda+1)\bar{t}}v_{0,\varepsilon}^{m})(\bar{\gamma}_{\bar{t}}),u\right)\nonumber \\
\geq & e^{-\lambda(\bar{t}+1)}\left(\lambda v_{0,\varepsilon}^{m}-D_{t}v_{0,\varepsilon}^{m}-\sup_{u}|\sigma|^{2}|D_{xx}v_{0,\varepsilon}^{m}|-\sup_{u}(|b|+|\sigma|)|D_{x}v_{0,\varepsilon}^{m}|)\right)(\bar{\gamma}_{\bar{t}})\nonumber \\
 & -D_{t}v_{\varepsilon}^{m;\varepsilon_{1}}(\bar{\gamma}_{\bar{t}})-\sup_{u\in U}\mathcal{H}\left(\bar{\gamma}_{\bar{t}},v(\bar{\gamma}_{\bar{t}}),(D_{x},D_{xx})v_{\varepsilon}^{m;\varepsilon_{1}}(\bar{\gamma}_{\bar{t}}),u\right)\nonumber \\
\geq & e^{-\lambda(\bar{t}+1)}\left(\lambda v_{0,\varepsilon}^{m}-C(|D_{t}v_{0,\varepsilon}^{m}|+|D_{xx}v_{0,\varepsilon}^{m}|+|D_{x}v_{0,\varepsilon}^{m}|)\right)(\bar{\gamma}_{\bar{t}})\nonumber \\
 & +\left[-D_{t}v_{\varepsilon}^{m;\varepsilon_{1}}-\hat{\mathcal{H}}^{m;\varepsilon_{1}}\left(\cdot,(1,D_{x},D_{xx})v_{\varepsilon}^{m;\varepsilon_{1}}\right)\right](\bar{\gamma}_{\bar{t}})\nonumber \\
 & +\left[\hat{\mathcal{H}}^{m;\varepsilon_{1}}\left(\cdot,(1,D_{x},D_{xx})v_{\varepsilon}^{m;\varepsilon_{1}}\right)-\sup_{u\in U}\mathcal{H}^{m}(\cdot,(1,D_{x},D_{xx})v_{\varepsilon}^{m;\varepsilon_{1}},u)\right](\bar{\gamma}_{\bar{t}})\nonumber \\
 & +\left[\sup_{u\in U}\mathcal{H}^{m}\left(\cdot,(1,D_{x},D_{xx})v_{\varepsilon}^{m;\varepsilon_{1}},u\right)-\sup_{u\in U}\mathcal{H}\left(\cdot,(1,D_{x},D_{xx})v_{\varepsilon}^{m;\varepsilon_{1}},u\right)\right](\bar{\gamma}_{\bar{t}})\nonumber \\
 & +\left[\sup_{u\in U}\mathcal{H}\left(\cdot,v_{\varepsilon}^{m;\varepsilon_{1}},(D_{x},D_{xx})v_{\varepsilon}^{m;\varepsilon_{1}},u\right)-\sup_{u\in U}\mathcal{H}\left(\cdot,v,(D_{x},D_{xx})v_{\varepsilon}^{m;\varepsilon_{1}},u\right)\right](\bar{\gamma}_{\bar{t}})\nonumber \\
:= & \text{Part1}+\text{Part2}+\text{Part3}+\text{Part4}+\text{Part5}.\nonumber
\end{align}
From estimates \eqref{eq:bounded-estimat-v}-\eqref{eq:approx-rate-v},
if $\lambda$ is sufficiently large,
\begin{align*}
 & \text{Part1}\\
\leq & e^{-\lambda}\left(\lambda(\tilde{v}_{0}-\text{Osc}(\bar{\gamma}_{\bar{t}},m^{-1})-m^{-\frac{1}{2}}-\varepsilon)-C(|D_{t}v_{0,\varepsilon}^{m}|+|D_{xx}v_{0,\varepsilon}^{m}|+|D_{x}v_{0,\varepsilon}^{m}|)\right)(\bar{\gamma}_{\bar{t}})\\
\leq & e^{-\lambda}\left(\lambda(\|\bar{\gamma}_{\bar{t}}\|_{0}-\mu m^{-\alpha}-m^{-\frac{1}{2}}-\varepsilon)-C(1+\|\bar{\gamma}_{\bar{t}}\|_{0})\right)\quad(\text{here choosing }\lambda>C)\\
\leq & -e^{-\lambda}(\mu m^{-\alpha}+m^{-\frac{1}{2}}+\varepsilon+C).
\end{align*}

Since $v_{\varepsilon}^{m;\varepsilon_{1}}$ is a classical solution
to PPDE \eqref{eq:second-approx-path-Equation}, we have
\[
\text{Part2}=0.
\]
From \eqref{eq:Hmi-second-approx}, we have
\[
|\hat{\mathcal{H}}^{m;\varepsilon_{1}}-\sup_{u}\mathcal{H}^{m}|\le C(1+|p|+|A|)\varepsilon_{1},
\]
which together with \eqref{eq:solution-second-approx-drivetive-estimat}
gives
\[
|\text{Part3}|\le C\varepsilon_{1}.
\]
From the estimates \eqref{eq:coeffient-first-estmat} and \eqref{eq:solution-second-approx-drivetive-estimat},
noting $\bar{\gamma}_{\bar{t}}\in\mathbf{C}_{\mu}^{\alpha}$, we have
\[
|\text{Part4}|\le C\|\mathbf{P}^{m}\bar{\gamma}_{\bar{t}}-\bar{\gamma}_{\bar{t}}\|_{0}\leq C\mu m^{-\alpha}.
\]
Both \eqref{eq:Monotonicity-of-h} and \eqref{eq:maximum-degen} imply
\[
\text{Part5}\geq-C(v_{\varepsilon}^{m;\varepsilon_{1}}(\bar{\gamma_{t}})-v(\bar{\gamma_{t}}))\geq\frac{1}{2}Cr_{0}.
\]
Note that the constants $C$ in this proof all do not depend on $m,\varepsilon_{1},\mu$
and $\lambda$. Setting $m\to\infty,\varepsilon_{1}\to0$, and then
considering the upper-limit as $\mu\to\infty$ on both sides of \eqref{eq:contradict-th-uniq},
we have from Definition \eqref{eq:suppersolution} the following inequality:
\[
0\geq-e^{-\lambda}(\varepsilon+C)+\frac{1}{2}Cr_{0},
\]
which is a contradiction when $\lambda$ tends to $\infty$. The proof
is complete.

\begin{rem} 
Assertions of Theorems \ref{thm:(repres-theorem-nondegen} and \ref{thm:Repr-Th-degen}
are still true if the coefficients in Assumptions (H2) and (H3) are
relaxed to grow in a linear way. 
\end{rem}

\subsection{Degenerate case}

In this subsection, we prove Theorem \ref{thm:Repr-Th-degen} using
the vanishing viscosity method (see \cite{lions1983optimal}).

\begin{proof}[Proof of Theorem \ref{thm:Repr-Th-degen}] Similarly
as in the non-degenerate case, we assume that $\mathcal{H}$ strictly
decreases in $y\in\mathbb{R}$ without loss of generality, i.e., \eqref{eq:Monotonicity-of-h}
holds. We only prove $\tilde{v}\geq v$, and the reverse inequality
can be proved in a symmetric (also easier) way.

First we construct an approximation of $\tilde{v}$. For any $\theta>0$,
$\gamma_{t}\in\Lambda$ and $u\in\mathcal{U}$, let $X^{\gamma_{t},u;\theta}$
and $Y^{\gamma_{t},u;\theta}$ solve following stochastic equations
\begin{equation}
\begin{cases}
X^{\gamma_{t},u;\theta}(s)=\gamma_{t}(s), & \text{all }\omega,\, s\in[0,t);\\
X^{\gamma_{t},u;\theta}(s)=\gamma_{t}(t)+\int_{t}^{s}b(X_{r}^{\gamma_{t},u;\theta},u(r))dr\\[3mm]
\quad\quad+\int_{t}^{s}\sigma(X_{r}^{\gamma_{t},u;\theta},u(r))\, dW(r)+\theta(\tilde{W}(t)-\tilde{W}(s)), & \text{a.s. }\omega,s\in[t,T];
\end{cases}\label{eq:diffu proc-3}
\end{equation}

\begin{eqnarray*}
Y^{\gamma_{t},u;\theta}(s) & = & g(X_{T}^{\gamma_{t},u;\theta})+\int_{s}^{t}f(X_{r}^{\gamma_{t},u;\theta},Y^{\gamma_{t},u;\theta}(r),Z^{\gamma_{t},u;\theta}(r),u(r))\, dr\\
 &  & -\int_{s}^{t}Z^{\gamma_{t},u;\theta}(r)\, dW(s)-\int_{s}^{t}\tilde{Z}^{\gamma_{t},u;\theta}(r)\, d\tilde{W}(s),\quad s\in[t,T].
\end{eqnarray*}
where $\{\tilde{W}_{t},0\leq t\leq T\}$ is an $n$-dimensional Brownian
motion, independent of $W$. Define
\[
\tilde{v}^{\theta}(\gamma_{t}):=\esssup_{u\in\mathcal{U}}Y^{\gamma_{t},u;\theta}(t).
\]
Obviously, we have
\begin{equation}
|\tilde{v}^{\theta}(\gamma_{t})-\tilde{v}(\gamma_{t})|\le C\theta.\label{eq:5.2-first-estimation}
\end{equation}
For any positive integer $m$, $\varepsilon>0$, and $\varepsilon_{1},\,(\varepsilon_{1}<m^{-1})$,
let $V_{\varepsilon}^{m,i;\theta,\varepsilon_{1}}:[t_{i-1},t_{i}]\times\mathbb{R}^{i\times n}\to\mathbb{R},\, i=1,\cdots,m$
be the viscosity solutions to the following state-dependent PDEs
\begin{equation}
\begin{cases}
-\partial_{t}V_{\varepsilon}^{m,i;\theta,\varepsilon_{1}}(t,\overrightarrow{x_{i}})-\frac{1}{2}\theta^{2}\Delta_{x_{i}x_{i}}V_{\varepsilon}^{m,i;\theta,\varepsilon_{1}}(t,\overrightarrow{x_{i}})\\
\qquad\qquad\qquad\qquad-\hat{H}^{m,i;\varepsilon_{1}}(t,\overrightarrow{x_{i}},(1,\partial_{x_{i}},\partial_{x_{i}x_{i}})V_{\varepsilon}^{m,i;\theta,\varepsilon_{1}}(t,\overrightarrow{x_{i}}))=0,\\
\qquad\qquad\qquad\qquad\qquad\qquad\qquad\,(t,\overrightarrow{x_{i}})\in(t_{i-1},t_{i})\times\mathbb{R}^{i\times n},i=1,\cdots,m;\\
V_{\varepsilon}^{m,i;\theta,\varepsilon_{1}}(t_{i},\overrightarrow{x_{i}})=V_{\varepsilon}^{m,i+1;\theta,\varepsilon_{1}}(t_{i}+,\overrightarrow{x_{i-1}},\mathbf{0},x_{i}),\qquad\quad\quad i<m,\overrightarrow{x_{i}}\in\mathbb{R}^{i\times n};\\
V_{\varepsilon}^{m,m;\theta,\varepsilon_{1}}(T,\overrightarrow{x_{m}})=G_{\varepsilon}^{m}(\overrightarrow{x_{m}}),\qquad\qquad\qquad\qquad\qquad\quad\overrightarrow{x_{m}}\in\mathbb{R}^{m\times n}.
\end{cases}\label{eq:first-approx-PDE-1}
\end{equation}
Like in Subsection \ref{sub:Smooth-approximations}, we know $V_{\varepsilon}^{m,i;\theta,\varepsilon_{1}}(\cdot,\overrightarrow{x_{i-1}},\cdot)\in\mathscr{C}^{1,2}([t_{i-1},t_{i}]\times\mathbb{R}^{n})$
and, there are some constants $C$, independent of $m,\, i,$ and
$\varepsilon_{1}$, such that
\begin{equation}
|V_{\varepsilon}^{m,i;\theta,\varepsilon_{1}}|+|\partial_{x_{i}}V_{\varepsilon}^{m,i;\theta,\varepsilon_{1}}|+|\partial_{x_{i}x_{i}}V_{\varepsilon}^{m,i;\theta,\varepsilon_{1}}|\le C(\theta),\label{eq:bound-estimat-degen-state}
\end{equation}
and for all $(t,\overrightarrow{x_{i}})\in[t_{i-1},t_{i}]\times\mathbb{R}^{i\times n}$,
$(s,\overrightarrow{y_{j}})\in[t_{j-1},t_{j}]\times\mathbb{R}^{j\times n}$,
$i\leq j\leq m$,
\begin{align}
 & \left|(1,\partial_{x_{i}},\partial_{x_{i}x_{i}},\partial_{t})V_{\varepsilon}^{m,i;\theta,\varepsilon_{1}}(t,\cdot)\big|_{\overrightarrow{y_{i}}}^{\overrightarrow{x_{i}}}\right|\le C(\theta)\max_{1\le k\le i}|(x_{1}-y_{1})+\cdots+(x_{k}-y_{k})|^{\beta},\label{eq:lipsch-in-space-deg}\\
 & \left|(1,\partial_{x_{i}},\partial_{x_{i}x_{i}},\partial_{t})V_{\varepsilon}^{m,i;\theta,\varepsilon_{1}}(t,\overrightarrow{x_{i}})-(1,\partial_{x_{i}},\partial_{x_{i}x_{i}},\partial_{t})V_{\varepsilon}^{m,j;\theta,\varepsilon_{1}}(s,\overrightarrow{x_{i-1}},\underbrace{\mathbf{0},\cdots,\mathbf{0}}_{j-i},x_{i})\right|\label{eq:lipsch-in--time-deg}\\
\le & \, C(\theta)|s-t|^{\frac{\beta}{2}},\nonumber
\end{align}
 Since the coefficients in $H^{m,i}$ are twice differentiable, in
view of the method in Krylov \cite[Lemma 1, Scetion 7.1 ]{krylov1987nonlinear},
we have the following lower bound estimate
\begin{equation}
\partial_{x_{i}x_{i}}V_{\varepsilon}^{m,i;\theta\varepsilon_{1}}\ge-C.\label{eq:2-orderlow-bound-degen-stat}
\end{equation}

Hence
\[
u_{\varepsilon}^{m;\theta,\varepsilon_{1}}(\gamma_{t}):=\sum_{i=1}^{m}\chi_{[t_{i-1},t_{i})}(t)V_{\varepsilon}^{m,i;\theta,\varepsilon_{1}}(t,\gamma_{t}(t_{1}),\gamma\big|_{t_{1}}^{t_{2}},\cdots,\gamma\big|_{t_{i-2}}^{t_{i-1}},\gamma\big|_{t_{i-1}}^{t})
\]
is the classical solution of PPDE

\begin{equation}
\begin{cases}
-D_{t}v_{\varepsilon}^{m;\theta,\varepsilon_{1}}-\frac{1}{2}\theta^{2}\Delta v_{\varepsilon}^{m;\theta,\varepsilon_{1}}-\hat{\mathcal{H}}^{m;\varepsilon_{1}}\left(\gamma_{t},(1,D_{x},D_{xx})v_{\varepsilon}^{m;\theta,\varepsilon_{1}}\right)=0, & \gamma_{t}\in\Lambda,t<T;\\
v_{\varepsilon}^{m;\theta,\varepsilon_{1}}(\gamma_{T})=g_{\varepsilon}^{m}(\gamma_{T}), & \gamma_{T}\in\Lambda_{T},
\end{cases}\label{eq:smooth approx}
\end{equation}
and satisfies the following estimates from \eqref{eq:bound-estimat-degen-state}
- \eqref{eq:2-orderlow-bound-degen-stat},
\begin{align}
\left|v_{\varepsilon}^{m;\theta,\varepsilon_{1}}\right|_{2,\beta;\Lambda} & \le C(\theta),\label{eq:holde-estimat-degen-path}\\
D_{xx}v_{\varepsilon}^{m;\theta,\varepsilon_{1}} & \ge-C.\label{eq:2order-low-bound-path}
\end{align}
Besides, we easily have the following approximating rate
\[
|v_{\varepsilon}^{m;\theta,\varepsilon_{1}}(\gamma_{t})-\tilde{v}(\gamma_{t})|<C\big(\text{Osc}(\gamma_{t},m^{-1})+m^{-\frac{1}{2}}+\theta+\varepsilon_{1}+\varepsilon\big).
\]
\medskip{}

Now we assert that $\tilde{v}\geq v$. Otherwise, we have
\[
\inf_{\gamma_{t}\in\Lambda}\left\{ \tilde{v}(\gamma_{t})-v(\gamma_{t})\right\} :=-r_{0}<0.
\]
 As in the proof of Theorem \ref{thm:(repres-theorem-nondegen}, for
any sufficiently large $\lambda>1$, there are $M_{0}=M_{0}(\lambda)>0$
and $\kappa_{1}=\kappa_{1}(\lambda,M_{0})\in(0,T)$, such that for
all sufficiently large number $\mu$, $\theta=\theta(r_{0}),\varepsilon=\varepsilon(r_{0}),\varepsilon_{1}<\overline{\varepsilon_{1}}(\mu,r_{0})$
and $m>\overline{m}(\mu,r_{0}),$ we have
\[
(v_{\varepsilon}^{m;\theta,\varepsilon_{1}}+e^{-\lambda(\bar{t}+1)}v_{0,\varepsilon}^{m}-v)(\bar{\gamma}_{\bar{t}})=\inf_{\gamma_{t}\in\mathbf{C}_{\mu}^{\alpha}}(v_{\varepsilon}^{\theta,m}+e^{-\lambda(t+1)}v_{0,\varepsilon}^{m}-v)(\gamma_{t})>-\frac{1}{2}r_{0},
\]
here $\bar{\gamma}_{\bar{t}}\in\mathbf{C}_{\mu}^{\alpha}\cap\mathbf{Q}_{M_{0},T-\kappa_{1}}$.
Noting \eqref{eq:bounded-estimat-v}, \eqref{eq:holde-estimat-degen-path},
there is a constant $\kappa=\kappa(M_{0},\theta,\lambda,\varepsilon)<\kappa_{1}$
such that
\[
\psi:=v_{\varepsilon}^{m;\theta,\varepsilon_{1}}+e^{-\lambda(t+1)}v_{0,\varepsilon}^{m}-(v_{\varepsilon}^{m,\theta,\varepsilon_{1}}+e^{-\lambda(\bar{t}+1)}v_{0,\varepsilon}^{m}-v)(\bar{\gamma}_{\bar{t}})
\]
lies in $\mathcal{J}_{\mu,\kappa}^{+}(\bar{\gamma}_{\bar{t}},v)$.
Since $v$ is a viscosity sub-solution to the path-dependent Bellman
equation \eqref{eq:PHJBE}, we consider the following formula:
\begin{align}
 & -D_{t}\psi(\bar{\gamma}_{\bar{t}})-\sup_{u\in U}\mathcal{H}(\bar{\gamma}_{\bar{t}},\psi(\bar{\gamma}_{\bar{t}}),D_{x}\psi(\bar{\gamma}_{\bar{t}}),D_{xx}\psi(\bar{\gamma}_{\bar{t}}),u)\label{eq:contradict-th-uniq-1}\\
= & -D_{t}(v_{\varepsilon}^{m;\theta,\varepsilon_{1}}+e^{-\lambda(\bar{t}+1)}v_{0,\varepsilon}^{m})(\bar{\gamma}_{\bar{t}})\nonumber \\
 & -\sup_{u\in U}\mathcal{H}\left(\bar{\gamma}_{\bar{t}},v(\bar{\gamma}_{\bar{t}}),(D_{x},D_{xx})(v_{\varepsilon}^{m;\theta,\varepsilon_{1}}+e^{-\lambda(\bar{t}+1)}v_{0,\varepsilon}^{m})(\bar{\gamma}_{\bar{t}}),u\right)\nonumber \\
\geq & e^{-\lambda(\bar{t}+1)}\left(\lambda v_{0,\varepsilon}^{m}-D_{t}v_{0,\varepsilon}^{m}-\sup_{u}|\sigma|^{2}|D_{xx}v_{0,\varepsilon}^{m}|-\sup_{u}(|b|+|\sigma|)|D_{x}v_{0,\varepsilon}^{m}|)\right)(\bar{\gamma}_{\bar{t}})\nonumber \\
 & -D_{t}v_{\varepsilon}^{m;\theta,\varepsilon_{1}}(\bar{\gamma}_{\bar{t}})-\sup_{u\in U}\mathcal{H}(\bar{\gamma}_{\bar{t}},v(\bar{\gamma}_{\bar{t}}),(D_{x},D_{xx})v_{\varepsilon}^{m;\theta,\varepsilon_{1}}(\bar{\gamma}_{\bar{t}}),u)\nonumber \\
\geq & e^{-\lambda(\bar{t}+1)}\left(\lambda v_{0,\varepsilon}^{m}-C(|D_{t}v_{0,\varepsilon}^{m}|+|D_{xx}v_{0,\varepsilon}^{m}|+|D_{x}v_{0,\varepsilon}^{m}|)\right)(\bar{\gamma}_{\bar{t}})+\frac{1}{2}\theta^{2}\Delta v_{\varepsilon}^{m;\theta,\varepsilon_{1}}(\bar{\gamma}_{\bar{t}})\nonumber \\
 & +\left[-D_{t}v_{\varepsilon}^{m;\theta,\varepsilon_{1}}-\frac{1}{2}\theta^{2}\Delta v_{\varepsilon}^{m;\theta,\varepsilon_{1}}-\hat{\mathcal{H}}^{m;\varepsilon_{1}}\left(\cdot,(1,D_{x},D_{xx})v_{\varepsilon}^{m;\theta,\varepsilon_{1}}\right)\right](\bar{\gamma}_{\bar{t}})\nonumber \\
 & +\left[\hat{\mathcal{H}}^{m;\varepsilon_{1}}\left(\cdot,(1,D_{x},D_{xx})v_{\varepsilon}^{m;\theta,\varepsilon_{1}}\right)-\sup_{u\in U}\mathcal{H}^{m;\varepsilon_{1}}\left(\cdot,(1,D_{x},D_{xx})v_{\varepsilon}^{m;\theta,\varepsilon_{1}},u\right)\right](\bar{\gamma}_{\bar{t}})\nonumber \\
 & +\left[\sup_{u\in U}\mathcal{H}^{m;\varepsilon_{1}}\left(\cdot,(1,D_{x},D_{xx})v_{\varepsilon}^{m;\theta,\varepsilon_{1}},u\right)-\sup_{u\in U}\mathcal{H}\left(\cdot,(1,D_{x},D_{xx})v_{\varepsilon}^{m;\theta,\varepsilon_{1}},u\right)\right](\bar{\gamma}_{\bar{t}})\nonumber \\
 & +\left[\sup_{u\in U}\mathcal{H}\left(\cdot,v_{\varepsilon}^{m;\theta,\varepsilon_{1}},(D_{x},D_{xx})v_{\varepsilon}^{m;\theta,\varepsilon_{1}},u\right)-\sup_{u\in U}\mathcal{H}\left(\cdot,v,(D_{x},D_{xx})v_{\varepsilon}^{m;\theta,\varepsilon_{1}},u\right)\right](\bar{\gamma}_{\bar{t}})\nonumber \\
:= & e^{-\lambda(\bar{t}+1)}\left(\lambda v_{0,\varepsilon}^{m}-C(|D_{t}v_{0,\varepsilon}^{m}|+|D_{xx}v_{0,\varepsilon}^{m}|+|D_{x}v_{0,\varepsilon}^{m}|)\right)(\bar{\gamma}_{\bar{t}})+\frac{1}{2}\theta^{2}\Delta v_{\varepsilon}^{m;\theta,\varepsilon_{1}}(\bar{\gamma}_{\bar{t}})\nonumber \\
 & +\text{Part1}+\text{Part2}+\text{Part3}+\text{Part4}.\nonumber
\end{align}
From Assumption (H3) and the corresponding estimates, similar to the
proof of Theorem \ref{thm:(repres-theorem-nondegen}, we have
\begin{align*}
 & e^{-\lambda(\bar{t}+1)}\left(\lambda v_{0,\varepsilon}^{m}-C(|D_{t}v_{0,\varepsilon}^{m}|+|D_{xx}v_{0,\varepsilon}^{m}|+|D_{x}v_{0,\varepsilon}^{m}|)\right)(\bar{\gamma}_{\bar{t}})\\
\geq & -e^{-\lambda}(\mu m^{-\alpha}+m^{-\frac{1}{2}}+\varepsilon+C),\quad\text{if }\lambda\text{ sufficiently large},\\
 & \theta^{2}\Delta v_{\varepsilon}^{m;\theta,\varepsilon_{1}}(\bar{\gamma}_{\bar{t}})\ge-\theta^{2}C,
\end{align*}
and
\[
\text{Part1}=0,\quad|\text{Part2}|\ge-C(\theta)\varepsilon_{1},\quad|\text{Part3}|\ge-C(\theta)\mu m^{-\alpha},\quad\text{Part4}\ge\frac{1}{2}Cr_{0}.
\]
Since the constants $C$ and $C(\theta)$ do not depend on $\mu,m,\varepsilon_{1}$
and $\lambda$, first setting $m\to\infty,\varepsilon_{1}\to0$ and
then considering the upper-limit as $\mu\to\infty$ on both sides
of \eqref{eq:contradict-th-uniq-1}, we have from the definition \eqref{eq:subsolution}
the following inequality
\[
0>-e^{-\lambda}(C+\varepsilon)-\frac{1}{2}\theta^{2}C+\frac{1}{2}Cr_{0},
\]
which yields a contradiction when sending $\lambda$ to $\infty$
and $\theta$ to $0$. The proof is complete. \end{proof}

\section{Appendix}

In this Appendix, we prove the $\alpha$-H\"older continuity of the
path for the following SDE:
\begin{equation}
X(t)=\int_{0}^{t}b(X_{r})dr+\int_{0}^{t}\sigma(X_{r})dW(r),\label{eq:hold-SDE}
\end{equation}
where $b:\Lambda\rightarrow\mathbb{R}^{n}$ and $\sigma:\Lambda\rightarrow\mathbb{R}^{n\times d}$
are uniformly Lipschitz continuous.

\begin{prop}\label{Prop:Hold-norm-prob-est} Let any $\alpha\in(0,\frac{1}{2})$,
$\mu>0$, and $X$ be the unique strong solution of \eqref{eq:hold-SDE}.
For any $p$ satisfying $(\frac{1}{2}-\alpha)p>1$, there is a constant
$C=C(p,\alpha)$ such that
\begin{equation}
P\{\llbracket X_{T}\rrbracket_{\alpha}\geq\mu\}\leq CT^{(\frac{1}{2}-\alpha)p}\mu^{-p}.
\end{equation}

\end{prop}

Now we give an auxiliary result. \begin{lem} \label{lem:Hold-lem}Fix
$\delta\in(0,T]$, $\mu>0$, $\alpha\in(0,\frac{1}{2})$ and $p>1$.
Then we have
\[
P\left\{ \max_{\begin{subarray}{c}
0\leq s<t\leq T\\
|s-t|\leq\delta
\end{subarray}}|X(s)-X(t)|>\mu\delta^{\alpha}\right\} \leq2\cdot3^{p}C_{p}\delta^{p(\frac{1}{2}-\alpha)-1}\mu^{-p}.
\]
\end{lem}

\begin{proof} Let $m=m(\delta)\geq2$ be the unique integer satisfying
$T/m<\delta\leq T/(m-1)$. Suppose that $|X(t)-X(s)|>\mu\delta^{\alpha}$
for some $s$ and $t$ such that $0\leq s<t\leq T$ and $|s-t|\leq\delta$.
Then there is a unique integer $q\in[0,m-1]$ such that $s\in[q\delta,(q+1)\delta)$.
There are two possibilities for $t$. One is $t\in[q\delta,(q+1)\delta)$,
where we have either of both inequalities:
\[
|X(q\delta)-X(s)|>\frac{1}{3}\mu\delta^{\alpha}\quad\text{and}\quad|X(q\delta)-X(t)|>\frac{1}{3}\mu\delta^{\alpha}.
\]
The other is $t\in[(q+1)\delta,(q+2)\delta)$ (with $q\leq m+2$),
where we have one of the three inequalities:
\[
|X(q\delta)-X(s)|>\frac{1}{3}\mu\delta^{\alpha},|X(q\delta)-X((q+1)\delta)|>\frac{1}{3}\mu\delta^{\alpha},\text{ and }|X((q+1)\delta)-X(t)|>\frac{1}{3}\mu\delta^{\alpha}.
\]
In conclusion, we have
\[
\left\{ \max_{\begin{subarray}{c}
0\leq s<t\leq T\\
|s-t|\leq\delta
\end{subarray}}|X(s)-X(t)|>\mu\delta^{\alpha}\right\} \subset\bigcup_{q=0}^{m-1}\left\{ \max_{q\delta\leq s\leq(q+1)\delta}|X(s\wedge T)-X(q\delta)|>\frac{1}{3}\mu\delta^{\alpha}\right\} .
\]
In view of Chebyshev's inequality and Lemma \ref{lem:FSDE}, we have
\begin{align*}
P\left\{ \max_{\begin{subarray}{c}
0\leq s<t\leq T\\
|s-t|\leq\delta
\end{subarray}}|X(s)-X(t)|>\mu\delta^{\alpha}\right\}  & \leq\sum_{q=0}^{m-1}P\left\{ \max_{q\delta\leq s\leq(q+1)\delta}|X(s\wedge T)-X(q\delta)|>\frac{1}{3}\mu\delta^{\alpha}\right\} \\
 & \leq\sum_{q=0}^{m-1}E\left[\max_{q\delta\leq s\leq(q+1)\delta}|X(s\wedge T)-X(q\delta)|^{p}\right]3^{p}(\mu\delta^{\alpha})^{-p}\\
 & \leq3^{p}(T\delta^{-1}+1)C_{p}\delta^{\frac{p}{2}}(\mu\delta^{\alpha})^{-p}\\
 & \leq2\cdot3^{p}C_{p}T\delta^{p(\frac{1}{2}-\alpha)-1}\mu^{-p}.
\end{align*}

\end{proof} \medskip{}

\begin{proof}[Proof of Proposition \ref{Prop:Hold-norm-prob-est}]
For any $s,t\in[0,T]$ such that $s<t$, there is a unique integer
$q\geq0$ such that $2^{-(q+1)}T<t-s\leq2^{-q}T.$ Obviously,
\begin{align*}
 & \left\{ |X(s)-X(t)|>\mu|s-t|^{\alpha}\right\} \subset\left\{ |X(s)-X(t)|>\mu T^{\alpha}2^{-\alpha(q+1)}\right\} \\
\subset & \left\{ \max_{\begin{subarray}{c}
0\leq s<t\leq T\\
|s-t|\leq2^{-q}T
\end{subarray}}|X(s)-X(t)|>2^{-\alpha}\mu(2^{-q}T)^{\alpha}\right\} .
\end{align*}
Thus, applying Lemma \ref{lem:Hold-lem},
\begin{align*}
P\left\{ \max_{0\leq s<t\leq T}\frac{|X(s)-X(t)|}{|s-t|^{\alpha}}>\mu\right\}  & \leq\sum_{q=0}^{\infty}P\left\{ \max_{\begin{subarray}{c}
0\leq s<t\leq T\\
|s-t|\leq2^{-q}T
\end{subarray}}|X(s)-X(t)|>2^{-\alpha}\mu(2^{-q}T)^{\alpha}\right\} \\
 & \leq\sum_{q=0}^{\infty}2\cdot3^{p}C_{p}T(2^{-q}T)^{p(\frac{1}{2}-\alpha)-1}(2^{-\alpha}\mu)^{-p}\\
 & =CT^{(\frac{1}{2}-\alpha)p}\mu^{-p}.
\end{align*}
This completes the proof.\end{proof}

\medskip{}

\textbf{Acknowledgment}: The authors would thank Professor Nizar Touzi
for pointing out an error in our first version of this paper, and
as well for helpful comments and suggestions on the second version. Of course, both authors are responsible for all errors occurring in the new version. 
 \bibliographystyle{plain}
\bibliography{bibtexOfPHJBE}

\end{document}